\documentclass{article}

\usepackage{amsthm, mathtools,amssymb,stmaryrd}
\usepackage{subfig}
\usepackage{hhline}
\usepackage{multirow}
\usepackage{multicol}
\usepackage{appendix}
\usepackage{xparse}
\usepackage{mathtools}
\usepackage{cleveref}
\usepackage{autonum}
\usepackage{everypage}
\usepackage{xspace}
\usepackage{ifthen}
\usepackage{enumitem}
\usepackage{url}
\usepackage{algorithm}
\usepackage{algpseudocode}
\usepackage{tikz,pgfplots}
\usepackage{bbm}
\usepackage{nicefrac}

\makeatletter
\def\@tvsp{\mathchoice{{}\mkern-3mu}{{}\mkern-3mu}{{}\mkern-3mu}{}}
\def\ltrivert{|\@tvsp|\@tvsp|}
\def\rtrivert{|\@tvsp|\@tvsp|}
\makeatother

\makeatletter
\def\@avgsp{\mathchoice{{}\mkern-6mu}{{}\mkern-6mu}{{}\mkern-6mu}{}}
\def\llaverage{\{\@avgsp\{}
\def\rraverage{\}\@avgsp\}}
\makeatother

\makeatletter
\def\@avgsp{\mathchoice{{}\mkern-6mu}{{}\mkern-6mu}{{}\mkern-6mu}{}}
\def\lharmonic{\langle\ \@avgsp}
\def\rharmonic{\@avgsp\ \rangle}
\makeatother

\NewDocumentCommand{\normDG}{O{\cdot} O{j} O{}}{\ensuremath{\|\ifthenelse{\equal{#1}{}}{\cdot}{#1}\|_{\ifthenelse{\equal{#2}{}}{j}{#2}}\ifthenelse{\equal{#3}{}}{}{^{#3}}}\xspace}

\NewDocumentCommand{\normL}{O{\cdot} O{2} O{\Om} O{}}{\ensuremath{\|\ifthenelse{\equal{#1}{}}{\cdot}{#1}\|_{L^{\ifthenelse{\equal{#2}{}}{1}{#2}}(\ifthenelse{\equal{#3}{}}{\Om}{#3})}\ifthenelse{\equal{#4}{}}{}{^{#4}}}\xspace}

\NewDocumentCommand{\normH}{O{\cdot} O{1} O{\Om} O{}}{\ensuremath{\|\ifthenelse{\equal{#1}{}}{\cdot}{#1}\|_{H^{\ifthenelse{\equal{#2}{}}{1}{#2}}(\ifthenelse{\equal{#3}{}}{\Om}{#3})}\ifthenelse{\equal{#4}{}}{}{^{#4}}}\xspace}

\NewDocumentCommand{\Aa}{O{j} O{\cdot} O{\cdot}}{\ensuremath{\mathcal{A}_{\ifthenelse{\equal{#1}{}}{j}{#1}}(\ifthenelse{\equal{#2}{}}{\cdot}{#2},\ifthenelse{\equal{#3}{}}{\cdot}{#3})}\xspace}

\NewDocumentCommand{\mcal}{O{} O{} O{}}{\ensuremath{\mathcal{#1}\ifthenelse{\equal{#2}{}}{}{_{#2}}\ifthenelse{\equal{#3}{}}{}{^{#3}}}\xspace}

\newcommand{\average}[1]{\ensuremath{\llaverage #1\rraverage}\xspace}
\newcommand{\jump}[1]{\ensuremath{\llbracket #1\rrbracket}\xspace}
\newcommand{\Om}{\ensuremath{\Omega}\xspace}
\newcommand{\elem}{\ensuremath{\kappa}\xspace}
\newcommand{\xvec}{\ensuremath{\mathbf{x}}\xspace}
\newcommand{\yvec}{\ensuremath{\mathbf{y}}\xspace}
\newcommand{\diff}{\ensuremath{\mathrm{d}}\xspace}
\newcommand{\harmonic}[1]{\ensuremath{\lharmonic #1\rharmonic}\xspace}

\newcommand{\hfine}{\ensuremath{h}\xspace}
\newcommand{\hlocal}{\ensuremath{\mathbb{H}}\xspace}
\newcommand{\hcoarse}{\ensuremath{H}\xspace}

\newcommand{\idj}{\ensuremath{\mathbbm{1}_{\{ \mcal[D][j] \}}}\xspace} 

\newcommand{\RomanNumeralCaps}[1]
    {\MakeUppercase{\romannumeral #1}}

\providecommand{\keywords}[1]{\textbf{Keywords:} #1}

\definecolor{newcolor}{rgb}{0,0,0.9}

\definecolor{newcolor3}{rgb}{0.9,0,0}

\newtheorem{theorem}{Theorem}[section]
\newtheorem{lemma}[theorem]{Lemma}

\newtheorem{definition}[theorem]{Definition}
\newtheorem{assumption}[theorem]{Assumption}
\theoremstyle{definition}

\newtheorem{remark}[theorem]{Remark}%

\title{An agglomeration-based massively parallel non-overlapping additive Schwarz preconditioner for high-order discontinuous Galerkin methods on polytopic grids\thanks{Paola F. Antonietti and Giorgio Pennesi have been partially funded by the SIR Project n. RBSI14VT0S funded by MIUR - Italian Ministry of Education, Universities and Research. Paola F. Antonietti and Giorgio Pennesi also acknowledge the financial support given by GNCS-INdAM.}}
\author{
P.F. Antonietti\\
MOX-Dipartimento di Matematica, Politecnico di Milano, \\
Piazza Leonardo da Vinci 32, 20133 Milano, Italy. \\
Email: paola.antonietti@polimi.it
\and 
P.~Houston\\
School of Mathematical Sciences,University of Nottingham,\\
University Park, Nottingham, NG7 2RD, UK\\
Email: Paul.Houston@nottingham.ac.uk
\and
G. Pennesi\\
MOX-Dipartimento di Matematica, Politecnico di Milano, \\
Piazza Leonardo da Vinci 32, 20133 Milano, Italy. \\
Email: giorgio.pennesi@polimi.it
\and
E. S\"uli\\
Mathematical Institute, University of Oxford, \\
Andrew Wiles Building, Woodstock Road, Oxford, OX2 6GG, UK\\
Email: endre.suli@maths.ox.ac.uk
}

\begin{document}

\maketitle

\begin{abstract}
In this article we design and analyze a class of two-level non-overlapping additive Schwarz preconditioners for the solution of the linear system of equations stemming from discontinuous Galerkin discretizations of second-order elliptic partial differential equations on polytopic meshes. The preconditioner is based on a coarse space and a non-overlapping partition of the computational domain where local solvers are applied in parallel. In particular, the coarse space can potentially be chosen to be non-embedded with respect to the finer space; indeed it can be obtained from the fine grid by employing agglomeration and edge coarsening techniques. We investigate the dependence of the condition number of the preconditioned system with respect to the diffusion coefficient and the discretization parameters, i.e., the mesh size and the polynomial degree of the fine and coarse spaces. Numerical examples are presented which confirm the theoretical bounds.
\end{abstract}

\keywords{Domain decomposition, polytopic grids, discontinuous Galerkin methods.}


\section{Introduction}\label{intro}
The process of defining a computational grid characterized by standard triangular/tetrahedral or quadrilateral/hexahedral-shaped elements is one of the potential bottlenecks when traditional finite element methods are employed for the numerical approximation of problems characterized by strong complexity of the physical domain, such as, for example, in geophysical applications, fluid-structure interaction, or crack propagation problems. In order to overcome this issue, during the last decade a wide strand of literature has focused on the design of numerical methods that support the use of computational meshes composed of general po\-ly\-gonal and polyhedral elements. In the conforming setting we mention, for example, the Composite Finite Element Method \cite{HaSa1972,AnGiHo2013}, the Mimetic Finite Difference
Method \cite{HyShSt1997,BrLiSi2005,BrLiSh2005,BeLiMa2014,AnFoScVeNi2016}, the Polygonal Finite Element Method \cite{SukTab2004}, the Extended Finite Element Method \cite{TabSuk2008,FrBe2010}, the Virtual Element Method 
\cite{BeBrCaMaMaRu2012,BeBrMaRu2016_2,AnBeMoVe2014}, and the Hybrid High-Order
Method \cite{DiErLa2014,DiEr2015,DiEr2015_2}. A major issue in the design of conforming methods on such general polytopic meshes is the definition of a suitable space of continuous piecewise polynomial functions; in this context, this is far from being a trivial task, particularly for high-order approximations. An alternative strand of literature has focused on the non-conforming setting, where the ease of defining spaces of piecewise polynomial functions is naturally associated with the flexibility provided by polytopic meshes. Here, we mention, for example, Hybridizable Discontinuous Galerkin Methods \cite{CoDoGu2008,CoGoLa2009}, non-conforming %
{Virtual Element Methods} \cite{AyusoLipnikovManzini_2016,AnMaVe2017,CangianiManziniSutton_2017}, Gradient Schemes \cite{DrEyHe2016}, and Discontinuous Galerkin (DG) Methods~\cite{AntBreMar2009,BaBoCoDiPiTe2012,BaBoCoSu2014,AnGiHo2013,AntGiaHou2014,CaGeHo2014,CaDoGeHo2017,AnHoHuSaVe2017,AnCaCoDoGeGiHo2016,AnMa2017,AnBoMa2018,AntoniettiVeraniVergaraZonca_2019,AnFaRuVe2016}. In particular, DG methods represent a class of powerful non-conforming numerical schemes in which the use of numerical grids characterized by general polytopic elements couples very well with the possibility to build the underlying discrete space in the physical frame, thereby avoiding the need to map polynomial spaces from a reference/canonical element.

However, as was shown in~\cite{AnHo2011}, the condition number of the matrix in a system of linear equations stemming from DG methods may be prohibitively large; indeed, by writing $h$ to denote the mesh-size and $p$ the polynomial degree, the condition number of the DG approximation to Poisson's equation grows like $\mcal[O](\nicefrac{p^4}{h^2})$ as $h$ tends to zero and $p$ tends to infinity. For this reason, in recent years, the development of fast solvers and preconditioners for systems of linear equations stemming from (high-order) DG discretizations has been an active area of research. A variety of two-level and multigrid/multilevel techniques have been proposed, both in the geometric and algebraic settings, for the solution of DG discretizations; see, for example,~\cite{GoKa2003_2,DoLaVaZi2006,BrZh2005,BrCuGuSu2011,AnSaVe2015}. In particular, the availability of efficient geometric multilevel solvers is strongly related to the possibility of employing general-shaped polytopic grids; indeed, if polytopic grids can be employed, then the sequence of grids which are required within a multilevel iteration can be defined by agglomeration; see~\cite{AnHoHuSaVe2017,AnPe2018} for details. Besides multigrid, a recent strand in the literature has focused on Schwarz domain decomposition methods; see, for example,~\cite{ToWi2004}, for a general abstract overview of these methods. In the DG setting where standard triangular/tetrahedral or quadrilateral/hexahedral grids are employed, one of the first contributions in terms of domain decomposition solvers was presented for the solution of elliptic problems in~\cite{FeKa2001}, where bounds of order $\mcal[O](\nicefrac{H}{\delta})$ and $\mcal[O](\nicefrac{H}{h})$ were obtained for the condition number of the preconditioned system in the framework of overlapping and non-overlapping Schwarz methods, respectively; here, $H,\ h$, and $\delta$ represent the size of the coarse grid, the fine grid, and the amount of overlap, respectively. 
Dryja and Sarkis proposed in~\cite{DrSa2010} an additive Schwarz preconditioner for the solution of second-order elliptic problems with discontinuous coefficients. There, the authors showed that the condition number of the of the matrix of the preconditioned system is independent of the jumps of the coefficients across the substructure boundaries and outside a thin layer along the substructure boundaries. A further development of this algorithm, which is very well suited for parallel computation, can be found in~\cite{DrKr2016}. 
Concerning the setting of high-order DG methods, we mention the work in~\cite{AnHo2011}, where additive and multiplicative Schwarz preconditioners were introduced for efficiently solving systems of linear equations arising from the discretization of second-order symmetric elliptic boundary-value problems using $hp$-version DG methods; there, $hp$-spectral bounds of order $\mcal[O](\nicefrac{\sigma p^2 H}{h})$ were derived for a class of domain decomposition preconditioners for DG discretizations, where $\sigma$ is the coefficient of the interior penalty stabilization parameter, $p$ is the polynomial approximation degree, and $H,\ h$ is the size of the coarse and fine mesh, respectively. Recently, in~\cite{AnHoSm2016}, this condition number estimate was improved to yield the optimal rate of $\mcal[O](\nicefrac{\sigma p^2 H}{q h})$, where $q$ denotes the polynomial approximation degree employed within the coarse grid solver; cf. \cite{doi:10.1002/num.22063} for related work. We also mention the recent work presented in~\cite{KaCo2017}, where the influence of the penalty terms, as well as the choice of coarse mesh spaces, on the condition number of the matrix of the system of linear equations preconditioned with additive Schwarz methods were investigated.

The goal of this article is to design and analyze a class of two-level non-overlapping additive Schwarz preconditioners for $hp$-version DG discretizations of second-order elliptic problems on general polytopic grids. Given the DG discrete problem defined on a fine mesh of granularity $h$, the preconditioner is designed by introducing two additional partitions employed to define the local solver operators and the coarse space correction. On the one hand, the partition employed to build the local solvers is related to a suitable splitting of the DG space and hence it is assumed to be nested with respect to the fine polytopic mesh. On the other hand no conditions are imposed on the coarse partition, which can be non-nested with respect to the fine grid. In particular, we consider the \textit{massively parallel setting}, whereby the partition employed for the local solvers are finer than the grid employed within the coarse solver. Here, we investigate the dependence of the condition number with respect to both the diffusion coefficient and the discretization parameters of the fine and coarse spaces, as well as the granularity of the partition for the local solver. We stress that our analysis is carried out in a very general setting, and, in particular, for nested meshes, it allows the computational domain where the model problem is posed to be non-convex; cf.~\Cref{sec:DGMetPolGri} below.

The rest of the paper is organized as follows. In~\Cref{sec:DGMetPolGri} we introduce the DG method on polytopic grids for the numerical approximation of second-order elliptic problems. In~\Cref{sec:ASPCG} we formulate the additive Schwarz preconditioner analyzed in this article. In~\Cref{sec-analytical_background} we then outline some key analytical results that are required for the analysis that follows. \Cref{sec:PreResAS} is devoted to deriving some preliminary results required to obtain the desired bound on the condition number of the matrix of the preconditioned system stated in~\Cref{sec:ConNumEst}. Numerical experiments are presented in~\Cref{sec:NumRes} to confirm the theoretical bounds derived in this article.
%
%
%
\section{DG method on polytopic grids}\label{sec:DGMetPolGri}
In this article we consider the following second-order elliptic problem. Let 
$\Om \subset \mathbb{R}^d,\ d=2,3$, be a polygonal/polyhedral domain with boundary $\partial\Omega$ and let $f \in L^2(\Om)$ be a given function. We consider the following weak formulation: find $u \in V = H^1_0(\Om)$  such that
\begin{equation}\label{eq:Poisson_2}
\mathcal{A}(u,v) := \int_{\Omega} \rho \nabla u \cdot \nabla v \,\diff \xvec = \int_{\Omega} f v \,\diff \xvec \quad \forall\, v \in V.
\end{equation}
Here, $\rho \in L^{\infty}(\Om)$ denotes the diffusion coefficient, which we assume to be such that $0< \rho_0 \le \rho$; here, we can assume that $\rho_0 = 1$, since~\eqref{eq:Poisson_2} can always be scaled by $\nicefrac{1}{\rho_0}$. Throughout this article, we use the notation $x \lesssim y$ to signify that there exists a positive constant $C$, independent of the diffusion coefficient $\rho$ and the discretization parameters, such that $x \le Cy$. Similarly we write $x \gtrsim y$ in lieu of $x \ge Cy$, while $x\ \eqsim\ y$ is used if both $x \lesssim y$ and $x \gtrsim y$ hold.

Let $\mcal[T][h]$ be a tessellation of $\Om$ consisting of disjoint polytopic elements $\elem$ of diameter $h_{\elem}$ such that $\overline{\Om} = \cup \,\overline{ \elem}$. Here, we denote by $\mcal[F][h]$ the set of \emph{faces} $F$, which are defined as the $(d-1)$-dimensional planar facets of the elements $\kappa$ present in the mesh $\mcal[T][h]$. For $d=3$, we assume that each planar face of an element $\elem \in \mcal[T][h]$ can be subdivided into a set of co-planar $(d-1)$-dimensional simplices and we refer to this set as the set of \textit{faces}, cf.~\cite{CaDoGeHo2017}. Moreover, we write $\mcal[F][h]^B := \{ F \in \mcal[F][h]\text{: }F \subset \partial \Om \}$ to denote the set of boundary faces and $\mcal[F][h]^I := \mcal[F][h] \setminus \mcal[F][h]^B$ the set of interior faces. We set $h :=\max_{\elem\in\mcal[T][h]} h_{\elem}$ and, to ease the presentation, we assume that $h \eqsim h_{\elem}$ for all $\elem \in \mcal[T][h]$.
\begin{remark}
We adopt the hypothesis that the diffusion coefficient $\rho$ is piecewise constant on each polytopic element $\elem \in \mcal[T][\hfine]$ and write $\rho_{\elem} = \rho|_{\elem}$ to denote its restriction to $\elem$; we refer to~\cite{GeLa2010} for the more
general case when $\rho$ violates this assumption.
\end{remark}
We assume that $\mcal[T][h]$ satisfies the following assumptions; cf.~\cite{CaGeHo2014,CaDoGe2017,CaDoGeHo2017} for details.
\begin{assumption} \label{ass1}
For any $\elem \in \mcal[T][h]$ there exists a set of non-overlapping $d$-di\-mensional simplices $\elem_{\flat}^F \subset \elem$, for $F \subset \partial \elem$, such that for any face $F \subset \partial \elem$, we have that $\overline{F} = \partial \elem \cap \partial \elem_{\flat}^F$, $\bigcup_{F \subset \partial \elem} \overline{\elem_{\flat}^F} \subset \overline{\elem}$, and the diameter $h_{\elem}$ of $\elem$ can be bounded by
$ h_{\elem} \lesssim \nicefrac{d |\elem_{\flat}^F |}{|F|}$ for all $F \subset \partial \elem$,
where $|F|$ and $|\elem_{\flat}^F|$ denote the Hausdorff measure of $F$ and $\elem_{\flat}^F$, respectively.	 \end{assumption}

\begin{assumption}\label{ass2}
We assume that there exists a covering $\mcal[T][h]^{\#}:= \{ \mathcal{S}_{\elem} \}_{\elem}$ of $\mcal[T][h]$ consisting of shape-regular $d$-dimensional simplices $\mathcal{S}_{\elem}$, such that, for any $\elem \in \mcal[T][h]$, there is an $\mathcal{S}_{\elem} \in \mcal[T][h]^{\#}$ satisfying $\elem \subset \mathcal{S}_{\elem}$ and ${\rm diam}(\mcal[S][\elem]) \lesssim h_{\elem}$. We also assume that
\begin{equation}
\max_{\elem \in \mcal[T][h]} \text{card} \bigl\{ \elem' \in \mcal[T][h]: \elem' \cap \mathcal{S}_{\elem} \ne \emptyset , \mathcal{S}_{\elem} \in \mcal[T][h]^{\#} \text{ such that }\elem \subset \mathcal{S}_{\elem} \bigr\} \lesssim 1.
\end{equation}
\end{assumption}

We write $\langle \cdot \rangle$ to denote the harmonic average operator defined as follows: let $\eta$ be a sufficiently smooth function; then, for any $F \subset \partial \elem,\ F \in \mcal[F][h]$, we define
\begin{equation}
\langle \eta \rangle|_F := \begin{cases}
\displaystyle \frac{2\ \eta_{\elem^+}\ \eta_{\elem^-}}{\eta_{\elem^+}\ +\ \eta_{\elem^-}},\ & F\in\mcal[F][h]^I,\ F \subset \partial \elem^+ \cap \partial \elem^-,\\
\displaystyle \eta_{\elem},\ & F\in\mcal[F][h]^B,\ F \subset \partial \elem \cap\partial \Omega,
\end{cases}
\end{equation}
where $\eta_{\elem}$, $\kappa\in \mcal[T][h]$, denotes the trace of $\eta$ on $\partial \elem$.
Moreover, for sufficiently smooth vector- and scalar-valued functions $\boldsymbol{\tau}$ and $v$, respectively, we define the following jump and weighted average operators across $F \in \mcal[F][\hfine]$:
for $F \subset \partial \elem^+ \cap \partial \elem^-$, $F\in \mcal[F][h]^I$, we write
\begin{align}\label{eq:JumpAndAverage}
\jump{\boldsymbol{\tau}} &:= \boldsymbol{\tau}_{\elem^+} \cdot \mathbf{n}_{\elem^+} + \boldsymbol{\tau}_{\elem^-} \cdot \mathbf{n}_{\elem^-},\ &
\average{\boldsymbol{\tau}}_{\omega} &:= \omega \boldsymbol{\tau}_{\elem^+}  + (1-\omega) \boldsymbol{\tau}_{\elem^-},\\
\jump{v} &:= v_{\elem^+} \mathbf{n}_{\elem^+} + v_{\elem^-}\mathbf{n}_{\elem^-},
& \average{v}_{\omega} &:= \omega v_{\elem^+}  + (1-\omega) v_{\elem^-},
\end{align}
where $\mathbf{n}_{\elem^\pm}$ denotes the unit outward normal vector to $\elem^\pm$, respectively. For  $F \subset \partial \elem$, $F\in \mcal[F][h]^B$, we set $\average{ \boldsymbol{\tau}}_{\omega} :=\boldsymbol{\tau}_{\elem},\ \jump{v} := v_{\elem}\mathbf{n},$ cf. \cite{ArBrMa2001}. Here, $\omega \in [0,1]$ represents the weight employed for the definition of $\average{\cdot}_{\omega}$. Given $\mcal[T][h]^s \subseteq \mcal[T][h]$ and $\mcal[F][h]^s \subseteq \mcal[F][h]$, we write
	$\int_{\mcal[T][h]^s} \mathrm{d}x := \sum_{\elem\in \mcal[T][h]^s} \int_{\elem}\mathrm{d}x$,
	$| v |_{H^1(\mcal[T][\hfine]^s)}^2 := \sum_{\elem \in \mcal[T][{\hfine}]^s} \| \nabla v \|_{L^2(\elem)}^2$,
	$\int_{\mcal[F][h]^s} \mathrm{d}s := \sum_{F\in \mcal[F][h]^s} \int_F \mathrm{d}s$, and
	$\| v \|_{L^2(\mcal[F][\hfine]^s)}^2 := \sum_{F \in \mcal[F][\hfine]^s} \| v \|_{L^2(F)}^2$.
Finally, writing $\mcal[P][p](\elem)$ to denote the space of polynomials of total degree $p \ge 1$ on $\elem$, $\elem \in \mcal[T][\hfine]$, the DG space is given by
\begin{equation}\label{eq:Vh}
V_h :=\{v_h \in L^2(\Om):v_h|_\elem\in \mcal[P][p](\elem) ~ \forall\, \elem\in\mcal[T][h]\}.
\end{equation}
With this notation, we introduce the Symmetric Interior Penalty DG (SIPDG) discretization of \eqref{eq:Poisson_2}, cf.~\cite{Wh1978,Ar1982,8160990}: find $u_{\hfine} \in V_{\hfine}$ such that
\begin{equation}\label{eq:dG}
\Aa[h][u_h][v_h] = \int_\Om f v_h\mathrm{d}x\quad\forall\, v_h\in V_h,
\end{equation}
where $\mathcal{A}_h: V_h \times V_h \rightarrow \mathbb{R}$ is the bilinear form of the DG method defined as
\begin{align}\label{eq:Ah}
\begin{aligned}
\Aa[h][u_h][v_h] &:= \int_{\Om} \rho \Bigl[ \nabla_h u_h \cdot  \nabla_h v_h\ +  \nabla_h u_h \cdot \mathcal{R}_{\rho}(\jump{v_h}) + \nabla_h v_h \cdot \mathcal{R}_{\rho}(\jump{u_h} ) \Bigr] \mathrm{d}\xvec \\
& ~~~+ \int_{\mcal[F][h]} \sigma_{h,\rho}\jump{u_h}\cdot\jump{v_h}\ \mathrm{d}s,
\end{aligned}
\end{align}
and $\nabla_h$ denotes the piecewise gradient operator on $\mcal[T][h]$. Here, $\mathcal{R}_{\rho}:[L^1(\mcal[F][h])]^d \rightarrow [V_h]^d$ denotes the lifting operator defined by
\begin{equation}\label{eq:lifting_R_2}
\int_{\Om} \mcal[R][\rho](\mathbf{q}) \cdot \boldsymbol{\eta}\ \diff \xvec \, := - \int_{\mcal[F][h]} \mathbf{q} \cdot \average{\boldsymbol{\eta}}_{\omega}\ \mathrm{d}s \quad \forall \, \boldsymbol{\eta} \in [V_h]^d,
\end{equation}
where, we take $\omega|_F := \frac{\rho_{\elem^-}}{\rho_{\elem^+}\ +\ \rho_{\elem^-}}$ on each internal face $F\in \mcal[F][h]^I$ shared by $\elem^{\pm}$.
In~\eqref{eq:Ah}, according to \cite{Dr2003,CaDoGeHo2017}, $\sigma_{h,\rho} \in L^\infty(\mcal[F][h])$ denotes the interior penalty stabilization function, which is defined by
\begin{equation}\label{eq:Sigma}
\sigma_{h,\rho}|_F := C_{\sigma} \harmonic{\rho_{\elem}} \nicefrac{ p^2 }{ \harmonic{ h_{\elem} } }\quad \forall\, F \in \mcal[F][h],
\end{equation}
with $C_{\sigma}>0$ independent of $\rho$, $p$, $|F|$, $|\elem|$ and $h_{\elem}$. It is well known that the condition number of the operator $\mcal[A][h]$ is potentially prohibitively large and depends on the size of the partition $\mcal[T][h]$ and the polynomial degree $p$ employed for the discretization; cf.~\cite{AnHo2011} for standard triangular/tetrahedral/hexahedral grids and~\cite{AnSaVe2015} for polytopic grids. Our goal is to introduce a massively parallel non-overlapping additive Schwarz preconditioner, which can be employed as a preconditioner to accelerate the convergence of iterative solvers such as the Conjugate Gradient method. 

\section{Non-overlapping additive Schwarz preconditioner}\label{sec:ASPCG}
The definition of the additive Schwarz preconditioner requires the introduction of two additional partitions (besides $\mcal[T][h]$): a partition $\mcal[T][\hlocal]$ composed of disjoint polytopic subdomains where local solvers are applied in parallel and a non-overlapping partition $\mcal[T][\hcoarse]$ employed for the coarse space correction. To this end, we introduce the following notation:
\begin{itemize}
\item $\mcal[T][\hlocal] := \{ \Om_1, \dots, \Om_{N_\hlocal} \}$ of size $\hlocal := \max_{1 \le i \le N_{\hlocal}}\{\text{diam}(\Om_i)\}$ such that each subdomain $\Om_i$ is the union of some elements $\elem \in \mcal[T][\hfine]$; we assume that $\hlocal \eqsim \text{diam}(\Om_i)$ for all $i=1,\dots,N_\hlocal$. We also assume that a colouring property holds, i.e., there exists a positive integer $N_{\mathbb{S}}$ such that
\begin{equation}\label{eq:Ns}
\max_{i=1,\dots,N_{\hlocal}} \text{card}\{ \Om_j:\ \partial \Om_i \cap \partial \Om_j \ne \emptyset \} \le N_{\mathbb{S}},
\end{equation}
i.e., $N_{\mathbb{S}}$ represents the maximum number of neighbours that any subdomain $\Om_i \in \mcal[T][\hlocal]$ may possess.
\item $\mcal[T][\hcoarse] := \{ \mcal[D]_1, \dots, \mcal[D]_{N_\hcoarse} \}$ of size $\hcoarse := \max_{1 \le j \le N_{\hcoarse}}\{\text{diam}(\mcal[D]_j)\}$ such that $\hcoarse \eqsim \text{diam}(\mcal[D]_j)$ for all $j=1,\dots,N_{\hcoarse}$.
\end{itemize}
We remark that  the grids $\mcal[T][\hcoarse]$ and $\mcal[T][\hfine]$ are possibly non-nested, cf.~\Cref{fig:TauhTauHTauHH}.

\begin{figure}[t]
\centering
\begin{tabular}{cc}
\includegraphics[width=0.3\textwidth]{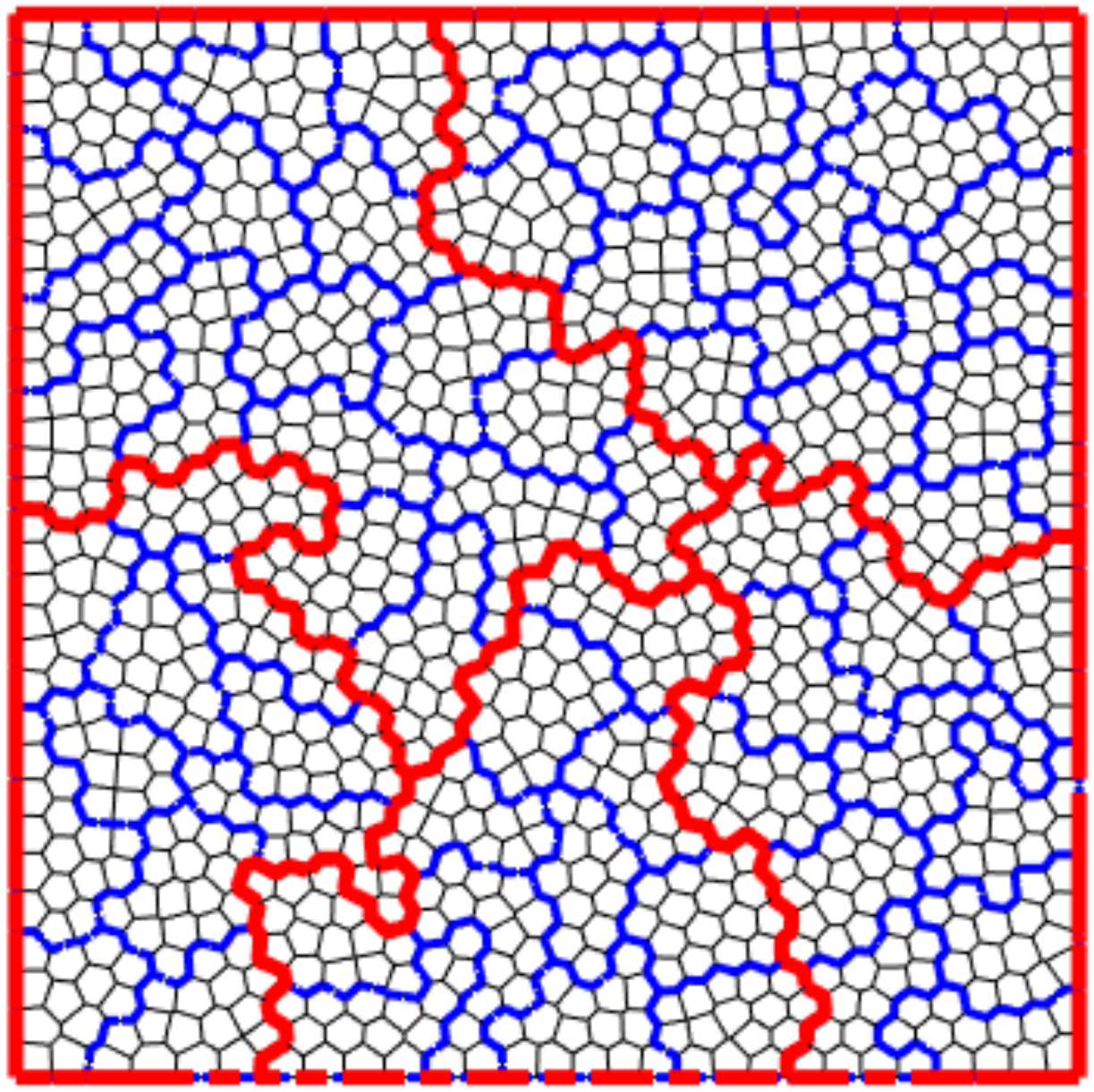}  &
\includegraphics[width=0.3\textwidth]{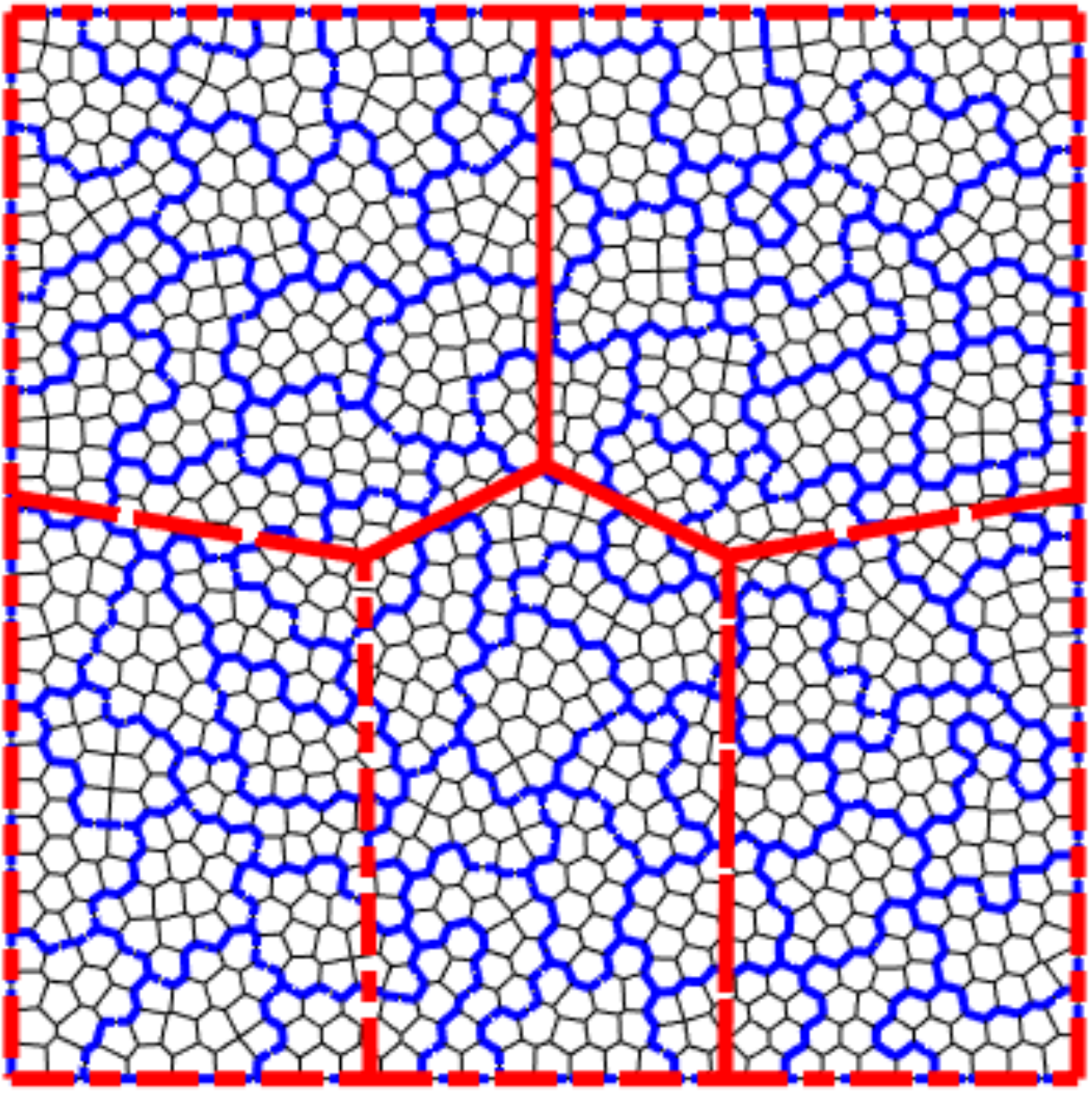}
\end{tabular}
\caption{Example of polygonal $\mcal[T][h]$ (black), $\mcal[T][\hlocal]$ (blue) and $\mcal[T][\hcoarse]$ (red), when the coarse and fine grids are nested, i.e., $\mcal[T][\hfine] \subseteq \mcal[T][\hcoarse]$, (left) and non-nested, i.e., $\mcal[T][\hfine] \nsubseteq \mcal[T][\hcoarse]$ (right).}
\label{fig:TauhTauHTauHH}
\end{figure}

\begin{remark} \label{rmrk:assump_nested_meshes}
Given that $\mcal[T][\hlocal]$ is defined by agglomeration of fine-grid-elements $\elem \in \mcal[T][\hfine]$, we write $\mcal[T][\hfine] \subseteq \mcal[T][\hlocal]$, since, for all $\elem \in \mcal[T][\hfine]$, there exists a $\mcal[K] \in \mcal[T][\hlocal]$ such that $\elem \subseteq \mcal[K]$. However, we point out that no further assumptions are needed on the relationship between $\mcal[T][\hlocal]$ and $\mcal[T][\hcoarse]$ for the definition of our method. Classical additive Schwarz methods have typically been defined based on the assumption that $\mcal[T][\hfine] \subseteq \mcal[T][\hcoarse] \subseteq \mcal[T][\hlocal]$. In this article we take a different approach: firstly, we {\em assume} that the granularity of $\mcal[T][\hlocal]$ is finer than that of $\mcal[T][\hcoarse]$; indeed, we are particularly interested in the massively parallel case whereby $\mcal[T][\hlocal] = \mcal[T][\hfine]$, cf.~\cite{DrKr2016}. Secondly, we also permit the use of non-nested coarse and fine partitions, i.e., when $\mcal[T][\hfine] \nsubseteq \mcal[T][\hcoarse]$.
\end{remark}

\noindent The main ingredients of the additive Schwarz method are defined as follows.
\medskip

\textbf{Local Solvers}. Consider the subdomain partition $\mcal[T][\hlocal]$ with cardinality $N_\hlocal$. Then, for each subdomain $\Om_i \in\mcal[T][\hlocal]$ we define a local space $V_i$ as the restriction of the DG finite element space $V_\hfine$ to $\Om_i$, i.e., for $i=1,\dots,N_\hlocal$,
\begin{equation}
V_i := \{ v_i \in L^2(\Om_i):\ v_i|_{\elem} \in \mcal[P][p](\elem)\ \forall\, \elem \in \mcal[T][h],\ \elem \subseteq \Om_i \} \equiv V_{\hfine}|_{\Om_i}.
\end{equation}
The associated local bilinear form on $V_i\times V_i$ is defined by
\begin{equation}
\mathcal{A}_i: V_i \times V_i \rightarrow \mathbb{R}, \quad \mathcal{A}_i(u_i,v_i) := \Aa[\hfine][R_i^\top u_i][R_i^\top v_i] \quad \forall\, u_i,v_i \in V_i,
\end{equation}
where $R_i^\top:V_i \rightarrow V_\hfine$ denotes the classical extension-by-zero operator from the local space $V_i$ to the global space $V_\hfine$. The restriction operator $R_i:V_{\hfine}\rightarrow V_i$, $i=1,\dots,N_\hlocal$, is defined as the transpose of $R_i^\top$ with respect to the $L^2(\Om_i)$ inner product.
\medskip

\textbf{Coarse Solver.} For $1 \le q \le p$, the coarse solver is defined on the partition $\mcal[T][\hcoarse]$. To this end, let $V_0$ be the DG finite element space defined on $\mcal[T][\hcoarse]$ given by
\begin{equation}
V_0 \equiv V_{\hcoarse} := \{v_{\hcoarse} \in L^2(\Om): v_{\hcoarse}|_{\mcal[D]_j} \in \mcal[P][q](\mcal[D]_j),\ j=1,\dots,N_{\hcoarse}\}.
\end{equation}
Further, let $R_0^\top$ be the $L^2$-prolongation operator from $V_0$ to $V_\hfine$, defined as:
\begin{equation}\label{eq:R0T}
R_0^\top: v_0 \in V_0 \longmapsto R_0^\top v_0 \in V_\hfine: \quad \int_{\Om} R_0^\top v_0 w_\hfine \diff \xvec :=  \int_{\Om} v_0 w_\hfine \diff \xvec \quad \forall\, w_\hfine \in V_\hfine.
\end{equation}
In this way $R_0^\top$ is well defined also when $\mcal[T][\hcoarse]$ and $\mcal[T][\hfine]$ are non-nested. Then, the bilinear form associated to $V_0$ is defined by
\begin{equation}\label{eq:A0}
\mathcal{A}_0: V_0 \times V_0 \rightarrow \mathbb{R}, \quad \mathcal{A}_0(u_0,v_0) := \mathcal{A}_\hfine(R_0^\top u_0, R_0^\top v_0) \quad \forall\, u_0,v_0 \in V_0.
\end{equation}

\begin{remark}[Nested spaces]\label{rmrk:R_0T_Nest}
When $V_0 \equiv V_{\hcoarse} \subseteq V_\hfine$, i.e., when the coarse and fine grids $\mcal[T][\hcoarse]$ and $\mcal[T][\hfine]$, respectively, are nested, then the action of $R_0^\top$ on a coarse function coincides with the action of the natural injection operator. Indeed, by contradiction, if there exists a $\overline{v}_0 \in V_{0}$ such that $R_0^\top \overline{v}_0 \ne  \overline{v}_0$, then, by employing the definition of $R_0^\top$, we have
\begin{equation}
0 < \| R_0^\top \overline{v}_0 - \overline{v}_0 \|_{L^2(\Om)} = \min_{w_h \in V_h} \| w_h - \overline{v}_0 \|_{L^2(\Om)} \le \| \overline{v}_0 - \overline{v}_0 \|_{L^2(\Om)} = 0,
\end{equation}
which is a contradiction and hence $R_0^{\top} v_0 = v_0$ for all $v_0 \in V_0$ when $V_0 \subseteq V_{\hfine}$.
\end{remark}
\noindent Introducing the projection operators $P_i := R_i^\top \widetilde{P}_i: V_{\hfine} \rightarrow V_{\hfine},\ i=0,1,\dots,N_\hlocal,$ where
\begin{align}
&\widetilde{P}_i:V_\hfine \rightarrow V_i, \ \mathcal{A}_i(\widetilde{P}_i v_\hfine, w_i) := \Aa[h][v_\hfine][R_i^\top w_i] \quad \forall\, w_i \in V_i, \ i=1,\dots,N_\hlocal,\\
&\widetilde{P}_0:V_\hfine \rightarrow V_0, \ \mathcal{A}_0(\widetilde{P}_0 v_\hfine, w_0) := \Aa[h][v_\hfine][R_0^\top w_0] \quad  \forall\, w_0 \in V_0,
\end{align}
the additive Schwarz operator is defined by $P_{ad} := \sum_{i=0}^{N_\hlocal} P_i$. For an upper bound on the condition number of $P_{as}$, we refer to \Cref{sec:ConNumEst} below.



\section{Analytical background} \label{sec-analytical_background}

Before we embark on developing the preliminary results needed to analyze the condition number of the additive Schwarz operator introduced in the previous section, we first state two key theorems, which are essential in the forthcoming analysis, and which may be of relevance in other more general settings. To this end, we write $D\subset {\mathbb R}^d$, $d\geq 2$, to denote a bounded, open, simply connected domain, with boundary $\partial D$; in the proceeding analysis, $D$ will be selected to be an element ${\mathcal D}_j$, $j=1,\ldots,N_H$, from the coarse mesh $\mcal[T][\hcoarse]$. Following \cite{St1970}, cf. also \cite{sauter_1996,sauter_warnke_1999}, we introduce the definition of a \textit{special Lipschitz domain}, and the notion of a \textit{domain with a minimally smooth boundary}.

\begin{definition}(\cite[Section 3.2]{St1970},\cite[Definition 1]{sauter_1996})
	Let $\phi:{\mathbb R}^{d-1}\rightarrow {\mathbb R}$ be a function that satisfies the Lipschitz condition
	\begin{equation}
	|\phi(\xvec) - \phi(\yvec)| \leq M |\xvec-\yvec| \qquad \forall\, \xvec,\yvec \in {\mathbb R}^{d-1}.
	\label{lipschitz_fn}
	\end{equation}
	The smallest $M$ for which \eqref{lipschitz_fn} holds is denoted by $C_\phi$.
	Based on this function, we define the special Lipschitz domain it determines to be the set of points lying above the hypersurface $y=\phi(\xvec)$ in ${\mathbb R}^{d}$, i.e.,
	$$
	\omega = \left\{ \xvec \in {\mathbb R}^{d}: x_d > \phi(x_1,x_2,\ldots,x_{d-1})\right\}\!.
	$$
	The Lipschitz constant of the domain $\omega$ is defined by $C_\omega:=C_\phi$.
\end{definition}
Equipped with this definition, we now introduce the concept of a \textit{minimally smooth boundary}.
\begin{definition}(\cite[Section 3.3]{St1970},\cite[Definition 2]{sauter_1996}) \label{minimally_smooth_boundary}
The boundary $\partial D$ of $D$ is said to be minimally smooth if there exists an $\epsilon>0$, a positive integer $N$, an $M>0$, and a sequence $U_1,U_2,\ldots$ of open sets such that:
\begin{enumerate}
	\item If $\xvec\in \partial D$, then $B(\xvec,\epsilon) \subset U_i$ for some $i$, where $B(\xvec,\epsilon)$ denotes the ball with centre $\xvec$ and radius $\epsilon$;
	\item No point $\xvec\in {\mathbb R}^{d}$ is contained in more than $N$ of the $U_i$'s;
	\item For each $i$ there exists a special Lipschitz domain $\omega_i$ with $C_{\omega_i} \leq M$ such that
	$$
	U_i\cap D =U_i \cap \omega_i.
	$$
\end{enumerate}
\end{definition}
Based on the previous definition, we now introduce the following classical extension operator.
\begin{theorem}(\cite[Theorem 5]{St1970}, \cite[Theorem 3]{sauter_1996}) \label{thm-extension}
Let $D$ be a domain with mi\-ni\-mally smooth boundary. Then, there exists a linear extension
operator $\frak{E}:H^s(D) \rightarrow H^s({\mathbb R}^d)$, $s \in {\mathbb N}_0$, such that
$\frak{E}v|_{D}=v$ and
$$
\| \frak{E} v \|_{H^s({\mathbb R}^d)} \leq C_{\frak{E}} \| v \|_{H^s(D)},
$$
where $C_{\frak{E}}$ is a positive constant depending only on $s$ and the constants $\epsilon$, $N$, and $M$ defined in Definition~\ref{minimally_smooth_boundary}, which characterize the boundary $\partial D$.
\end{theorem}
\begin{remark}
We highlight that, crucially, the constant $C_{\frak{E}}$ appearing in Theorem~\ref{thm-extension} is independent of the measure of the underlying domain $D$.
\end{remark}

Secondly, we now study the regularity of the following Neumann boundary-value problem: find $z$ such that
\begin{equation}\label{eq:ProblemZ}
- \Delta z  = f \quad\ \text{in } D, \qquad
\nabla z \cdot \mathbf{n} = 0 \quad\ \text{on } \partial D,
\end{equation}
with $f \in L^2_0(D) := \{ v \in L^2(D)\ :\ \int_{D} v\ \diff \xvec = 0 \}$.
\begin{theorem} \label{thm_elliptic_regularity}
	Let $D$ be a bounded, open, convex (and therefore Lipschitz) domain. Then, there exists a unique solution $z \in H^2(D) \cap L^2_0(D)$ to the homogeneous Neumann problem \eqref{eq:ProblemZ}. Moreover, the following stability bound holds:
	\begin{equation}
		\|z\|_{H^2(D)} \leq \sqrt{6} \,\big(\| f\|_{L^2(D)} + \frac{1}{\pi}\mbox{\rm diam}(D) \| \nabla z \|_{L^2(D)}\big).
	\label{eqn_elliptic_reg}
	\end{equation}
\end{theorem}

\begin{proof}
	We defer the proof to Appendix~\ref{proof_elliptic_regularity}.
\end{proof}

\begin{remark}
An analogous bound to \eqref{eqn_elliptic_reg}, with constant equal to unity, has been derived in \cite[Theorem 4.3.1.4]{Gr1985} for convex polygonal domains in ${\mathbb R}^2$; there the term
$\frac{1}{\pi}\mbox{\rm diam}(D) \|\nabla z\|_{L^2(\Omega)}$ on the right-hand side of \eqref{eqn_elliptic_reg} is replaced by $\| z \|_{H^1(D)}$.
\end{remark}

\section{Preliminary results}\label{sec:PreResAS}
We first present some preliminary results, which will be employed within the analysis contained in~\Cref{sec:ConNumEst}. For the sake of simplicity of the presentation, we assume that the grids $\mcal[T][\hfine]$ and $\mcal[T][\hcoarse]$ are nested; the extension of the theoretical analysis to the general case $\mcal[T][\hfine] \nsubseteq \mcal[T][\hcoarse]$ is deferred to~\Cref{app:NonNestProof}. 
Here, we introduce the following energy norm:
\begin{equation}\label{eq:DGnorm_2}
\normDG[w][h,\rho][2]:=\int_{\mcal[T][h]} \rho |\nabla_h w|^2\ \diff \xvec + \int_{\mcal[F][h]} \sigma_{h,\rho}|\jump{w}|^2\ \mathrm{d}s.
\end{equation}
The well-posedness of problem~\eqref{eq:dG} with respect to the norm~\eqref{eq:DGnorm_2} is then established in the following lemma, cf.~\cite{CaDoGeHo2017}.
\medskip
\begin{lemma}\label{lem:contcoerc_2}
Suppose that $\mcal[T][h]$ satisfies~\Cref{ass1}; then,
\begin{alignat}{3}
\Aa[h][u_h][v_h]&\lesssim \normDG[u_h][h,\rho]\normDG[v_h][h,\rho]\quad &&\forall\, u_h,v_h \in V_h, \label{eq:ContAh_2}\\
\Aa[h][u_h][u_h]&\gtrsim \normDG[u_h][h,\rho][2]\quad &&\forall\, u_h \in V_h. \label{eq:CoercAh_2}
\end{alignat}
The second bound holds provided that $C_{\sigma}$ appearing in~\eqref{eq:Sigma} is sufficiently large.
\end{lemma}
We also recall the following trace inequality on polytopic domains, introduced in~\cite{CaDoGe2017}.
\begin{lemma}\label{lem:inversepoly}
Suppose that $\mcal[T][h]$ satisfies~\Cref{ass1}; then, the following inequality holds:
\begin{equation}
\normL[v][2][\partial \elem][2]\lesssim \nicefrac{p^2}{h_{\elem}} \, \normL[v][2][\elem][2] \quad \forall\, v \in \mcal[P][p](\elem)\ \forall\, \elem \in \mcal[T][\hfine].
\end{equation}
\end{lemma}
Next we also recall a result regarding the approximation operator presented in \cite[Theorem 5.2]{CaGeHo2014} and \cite[Lemma 5.5]{CaDoGe2017}, to which we refer for details. However, for the purposes of this article, we consider a slight generalization: given a connected subdomain $D\subseteq\Omega$, we assume that $D$ is formed from the union of a subset of elements $\kappa\in \mcal[T][h]$; we denote the collection of such elements by $\mcal[T][h,D]$, i.e., $\bar{D} := \cup_{\kappa\in \mcal[T][h,D]} \bar{\kappa}$. With this definition, the approximant relies on the properties of the extension operator ${\mathfrak E}: H^s(D) \rightarrow H^s(\mathbb{R}^d),\ s \in \mathbb{N}_0$, introduced in Theorem~\ref{thm-extension}. Hence, following \cite[Theorem 5.2]{CaGeHo2014} and \cite[Lemma 5.5]{CaDoGe2017}, we deduce the following lemma.
\begin{lemma}\label{lem:interpDG_loc}
 Suppose that~\Cref{ass2} is satisfied, and let $v \in L^2(D)$ be such that, for some $k \ge 0$, $v|_{\elem} \in H^k(\elem)$ and ${\mathfrak E}v|_{\mcal[S][\elem]} \in H^k(\mcal[S][\elem])$ for $\elem \in \mcal[T][h,D]$, with $\mcal[S][\elem] \in \mcal[T]^{\#}_h$ as defined in~\Cref{ass2}. Then, there exists a projection operator $\Pi_{h,\elem}: L^2(D)\rightarrow \mcal[P][p](\elem)$ such that
\begin{align}
\normH[v-\Pi_{h,\elem} v][l][\elem] & \lesssim  \frac{h^{s-l}}{p^{k-l}} \normH[{\mathfrak E} v][k][\mcal[S][\elem]], \quad \ 0 \le l \le k.\\
\| v-\Pi_{h,\elem} v \|_{L^2(\partial \elem)} & \lesssim \frac{h^{s - \nicefrac{1}{2}}}{p^{k - \nicefrac{1}{2}}}  \normH[{\mathfrak E} v][k][\mcal[S][\elem]], \quad \ k \ge 1,
\end{align}
where $s:=\min\{p+1,k\}$ and $p \geq 1.$
\end{lemma}

\begin{remark}[Global approximant]\label{rmrk:GloApp}
Given that~\Cref{lem:interpDG_loc} holds for all $\elem \in \mcal[T][h,D]$ we can define the global approximation operator $\Pi_{h}: L^2(D)\rightarrow V_h|_D$ such that $\Pi_h|_{\elem} = \Pi_{h,\elem}$. If $v \in H^k(D)$, $k \ge 0$, then, by noting~\Cref{ass2}, together with Theorem~\ref{thm-extension}, the following bound holds:
\begin{equation}\label{eq:interpDG}
\normH[v-\Pi_h v][l][\mcal[T][h,D]] \lesssim  \frac{h^{s-l}}{p^{k-l}} \normH[v][k][D], \quad 0 \le l \le k.
\end{equation}
\end{remark}

A key ingredient in our analysis is the conforming approximant defined in~\cite{AnHoSm2016}. In particular, to ensure that the preconditioner is scalable in the presence of jumps in the diffusion coefficient, here we define the conforming approximant in a slightly different manner, in order to obtain an approximation of discontinuous discrete functions $v_h \in V_h$ on each local domain $\mcal[D][j] \in \mcal[T][\hcoarse],\ j=1,\dots,N_{\hcoarse}$. To this end, we adopt the following assumption on the coarse mesh $\mcal[T][\hcoarse]$.
\begin{assumption}\label{ass:Dj1}
We assume that $\mcal[D][j] \in \mcal[T][\hcoarse]$ is a convex polytope with Lipschitz boundary $\partial \mcal[D][j]$, for any $j=1,\dots,N_{\hcoarse}$, and that $|\mcal[D][j]| \eqsim \hcoarse^d$, $j=1,\dots,N_{\hcoarse}$, where $| \mcal[D][j] |$ represents the Hausdorff measure of $\mcal[D][j]$.
\end{assumption}
%
For the sake of the analysis, we define the following local grids generated from $\mcal[T][h]$ and $\mcal[T][\hcoarse]$:
\begin{align}\label{eq:Thj}
\begin{aligned}
  \mcal[T][h,j] := \{ \elem \in \mcal[T][\hfine]: \elem \subset \mcal[D][j],\ \mcal[D][j] \in \mcal[T][\hcoarse]\},
\, \ \mcal[F][h,j]^I := \{ F \in \mcal[F][h]: F \subset \mcal[D][j],\ \mcal[D][j] \in \mcal[T][\hcoarse]\}, 
\\
\mcal[F][h,j]^B := \{F\in \mcal[F][h]: F \subset \partial \mcal[D][j],\ \mcal[D][j] \in \mcal[T][\hcoarse]\},\ \
\mcal[F][h,j] := \mcal[F][h,j]^I \cup \mcal[F][h,j]^B, ~\hspace{2.2cm}
\end{aligned}
\end{align}
for $j=1,\dots,N_{\hcoarse}$.
\begin{remark}
Note that since the grids $\mcal[T][\hfine]$ and $\mcal[T][\hcoarse]$ are nested, i.e., $\mcal[T][h,j] \subseteq \mcal[T][\hfine]$, $j=1,\dots,N_{\hcoarse}$, $\mcal[T][h,j]$ also satisfies Assumptions~\ref{ass1} and~\ref{ass2}, for all $j=1,\dots,N_{\hcoarse}$.
\end{remark}
The local conforming approximant is then defined as follows.
\begin{definition}\label{def:H_of_v_h}
Let $\mcal[D][j] \in \mcal[T][\hcoarse]$ satisfy~\Cref{ass:Dj1}, $j=1,\dots,N_{\hcoarse}$, and let the discrete gradient operator of $v_h \in V_h$ inside $\mcal[D][j]$ be defined by the equality $\mathcal{G}_{h,j}(v_h) = (\nabla_h v_h + \mcal[R][1](\mcal[J][j](v_{\hfine}))) \idj$, $j=1,\dots,N_{\hcoarse}$, where $\idj$ is the characteristic function on $\mcal[D][j]$, while $\mcal[J][j]: V_{\hfine} \rightarrow [L^1(\mcal[F][h])]^d$ is defined by
\begin{equation}\label{eq:Jj}
\mcal[J][j](v_{\hfine})|_{F} := \jump{v_{\hfine}}|_F \ \ \text{if }F\in \mcal[F][h,j]^I, \qquad
\mcal[J][j](v_{\hfine})|_{F} := 0 \ \ \text{otherwise,}
\end{equation}
where $\mcal[F][h,j]^I$ is the set defined in~\eqref{eq:Thj}. Here, $\mcal[R][1]: [L^1(\mcal[F][h])]^d \rightarrow [V_h]^d$ is the lifting operator with $\rho=1$ and $\omega = \nicefrac{1}{2}$ in its definition given in~\eqref{eq:lifting_R_2}. Then, $\widetilde{v}_{h,j}$ is defined as the solution of the following problem: find $\widetilde{v}_{h,j} \in \widetilde{V}_j = H^1(\mcal[D][j]) \cap \L^2_0(\mcal[D][j])$ such that
\begin{equation}\label{eq:HdefInChapter4InChapter4}
\int_{\mcal[D][j]} \nabla \widetilde{v}_{h,j} \cdot \nabla w\ \diff \xvec = \int_{\mcal[D][j]} \mathcal{G}_{h,j}(v_h) \cdot \nabla w\ \diff \xvec \quad \forall\, w \in \widetilde{V}_j.
\end{equation}
\end{definition}

\begin{remark}[Poincar\'e's inequality]\label{rmrk:Poincare} Since $\widetilde{v}_{h,j} \in \widetilde{V}_j$ and $\mcal[D][j]$ satisfies~\Cref{ass:Dj1}, we note that, cf.~\cite[Corollary 3.4]{ZhQi2005}, $\| \widetilde{v}_{h,j} \|_{L^2(\mcal[D][j])} \le C_p \| \nabla \widetilde{v}_{h,j} \|_{L^2(\mcal[D][j])}$, with $C_p \lesssim {\rm diam}(\mcal[D][j])^{1+\nicefrac{d}{2}} | \mcal[D][j] |^{-\nicefrac{1}{2}}$. Thereby, invoking~\Cref{ass:Dj1} gives
\begin{equation}\label{eq:Poincare}
\| \widetilde{v}_{h,j} \|_{H^1(\mcal[D][j])}^2 \le (1 + C_p^2) \| \nabla \widetilde{v}_{h,j} \|_{L^2(\mcal[D][j])}^2\lesssim \| \nabla \widetilde{v}_{h,j} \|_{L^2(\mcal[D][j])}^2,
\end{equation}
where we also made use of the fact that ${\rm diam}(\mcal[D][j]) \eqsim \hcoarse \lesssim 1$.
\end{remark}

\noindent By proceeding as in~\cite{AnHoSm2016} we prove the following approximation result.
\begin{theorem}\label{th:Hbound}
Let $\mcal[D][j] \in \mcal[T][\hcoarse]$ satisfy~\Cref{ass:Dj1}, $j=1,\dots,N_{\hcoarse}$. Given $v_h \in V_{h}$, we take $\widetilde{v}_{h,j} \in \widetilde{V}_j$, $j=1,\dots,N_{\hcoarse}$, to be the conforming approximant given in~\Cref{def:H_of_v_h} and we define
$$\overline{v}_{h,j} := \widetilde{v}_{h,j} + \frac{1}{|\mcal[D][j]|} \int_{\mcal[D][j]} v_h \diff \xvec.$$
Then, the following approximation and stability bounds hold:
\begin{align}
\normL[v_h -  \overline{v}_{h,j} ][2][\mcal[D][j]] &\lesssim \frac{h}{p} \normL[\sigma_{h,1}^{\nicefrac{1}{2}} \jump{v_h}][2][\mcal[F][h,j]^I],\label{eq:HboundL2}\\
| \widetilde{v}_{h,j} |_{H^1(\Om)}^2 &\lesssim \| \nabla_h v_h \|_{L^2(\mcal[D][j])}^2 + \normL[\sigma_{h,1}^{\nicefrac{1}{2}} \jump{v_h}][2][\mcal[F][h,j]^I]^2,\label{eq:HboundH1}
\end{align}
for $j=1,\dots,N_{\hcoarse}$.
\end{theorem}

\begin{proof}
Given $\mcal[D][j] \in \mcal[T][\hcoarse]$, $j=1,\ldots,N_H$, let $z \in H^2(\mcal[D][j]) \cap L^2_0(\mcal[D][j])$ be the solution of problem~\eqref{eq:ProblemZ} posed on the domain $\mcal[D][j]$, i.e., $D=\mcal[D][j]$, with $f=v_h   - \overline{v}_{h,j}$. Then, by employing integration by parts we obtain 
\begin{align}
\| v_h   - \overline{v}_{h,j} \|_{L^2(\mcal[D][j])}^2 & = - \int_{\mcal[D][j]} (v_h - \overline{v}_{h,j} )\ \Delta z\ \diff \xvec \\
& = \int_{\mcal[D][j]} (\nabla_h v_h - \nabla \widetilde{v}_{h,j} ) \cdot \nabla z\ \diff \xvec\ -\ \int_{\mcal[F][h,j]^I} \nabla z \cdot \jump{v_h}\ \diff s\\
& = - \int_{\mcal[D][j]} \mcal[R][1](\mcal[J][j](v_{\hfine})) \cdot \nabla z\ \diff \xvec  \ \ \ -\int_{\mcal[F][h,j]^I} \nabla z \cdot \jump{v_h}\ \diff s,
\end{align}
where we have also used the facts that $\nabla \overline{v}_{h,j} = \nabla \widetilde{v}_{h,j}$, $\nabla z \cdot \mathbf{n}|_F = 0$ if $F \subset \partial \mcal[D][j]$, $\jump{\overline{v}_{h,j}}|_F = \mathbf{0}$ for all $F\in \mcal[F][h,j]^I$, since $\overline{v}_{h,j} \in H^1(\mcal[D][j])$, and that $\mcal[G][h,j](v_h) = (\nabla_h v_h + \mcal[R][1](\mcal[J][j](v_{\hfine}))\idj$ together with the definition of $\widetilde{v}_{h,j}$, cf.~\eqref{eq:HdefInChapter4InChapter4}. Using the definition of $\mcal[R][1]$ and $\mcal[J][j]$, cf.~\eqref{eq:lifting_R_2} and~\eqref{eq:Jj}, respectively, for any $z_h \in V_h|_{\mcal[D][j]}$, we have
\begin{align}
\| v_h  - \overline{v}_{h,j} \|_{L^2(\mcal[D][j])}^2
& = - \int_{\mcal[F][h,j]^I} \nabla z \cdot \jump{v_h}\ \diff s\ - \int_{\mcal[D][j]} \mcal[R][1](\mcal[J][j](v_{\hfine})) \cdot \nabla z\ \diff \xvec \\
& = -\ \int_{\mcal[F][h,j]^I} \nabla z \cdot \jump{v_h}\ \diff s \
 + \int_{\mcal[F][h,j]^I} \jump{v_h} \cdot \average{\nabla_h z_h}_{\nicefrac{1}{2}}\ \diff s \\
& \quad \ \! -\ \int_{\mcal[D][j]} \mcal[R][1](\mcal[J][j](v_{\hfine})) \cdot ( \nabla z - \nabla_h z_h )\ \diff \xvec\\
& = \int_{\mcal[F][h,j]^I} \jump{v_h} \cdot \average{ \nabla_h z_h - \nabla z}_{\nicefrac{1}{2}}\ \diff s \\
& \quad\  - \int_{\mcal[D][j]} \mcal[R][1](\mcal[J][j](v_{\hfine})) \cdot ( \nabla z - \nabla_h z_h )\ \diff \xvec;
\end{align}
here we have used that, since $z \in H^2(\mcal[D][j])$, $\average{\nabla z}_{\nicefrac{1}{2}}|_F = \nabla z|_F$, $F\in \mcal[F][h,j]^I$. Hence, we get
\begin{align}
\| v_h - \overline{v}_{h,j} \|_{L^2(\mcal[D][j])}^2 & \lesssim \| \mcal[R][1](\mcal[J][j](v_{\hfine})) \|_{L^2(\mcal[D][j])} \|  \nabla z - \nabla_h z_h \|_{L^2(\mcal[D][j])}\ \\
& \quad + \| \sigma_{h,1}^{\nicefrac{1}{2}} \jump{v_h} \|_{L^2(\mcal[F][h,j]^I)} \| \sigma_{h,1}^{-\nicefrac{1}{2}} \average{ \nabla z - \nabla_h z_h }_{\nicefrac{1}{2}}\ \|_{L^2(\mcal[F][h,j]^I)}. \label{eq:ThisEquationInThisProof}
\end{align}
The first term on the right-hand side of~\eqref{eq:ThisEquationInThisProof} can be written as follows:
\begin{align}\label{eq:HereUseDefRrho}
\begin{aligned}
\| \mcal[R][1](\mcal[J][j](v_{\hfine}) ) \|_{L^2(\mcal[D][j])}^2 & \le \int_{\Om} \mcal[R][1](\mcal[J][j](v_{\hfine})) \cdot \mcal[R][1](\mcal[J][j](v_{\hfine}))\ \diff \xvec \\
 & = - \int_{\mcal[F][h]} (\mcal[J][j](v_{\hfine}) )\cdot \average{ \mcal[R][1](\mcal[J][j](v_{\hfine})) }_{\nicefrac{1}{2}}\ \diff s \\
 & = - \int_{\mcal[F][h,j]^I} \jump{v_h} \cdot \average{ \mcal[R][1](\mcal[J][j](v_{\hfine})) }_{\nicefrac{1}{2}}\ \diff s,
\end{aligned}
\end{align}
where we have also used the definitions of $\mcal[R][1]$ and $\mcal[J][j]$, cf.~\eqref{eq:lifting_R_2} and~\eqref{eq:Jj}, respectively. Then, from~\eqref{eq:HereUseDefRrho} and the Cauchy--Schwarz inequality we obtain the following bound:
\begin{align}
\| \mcal[R][1]& (\mcal[J][j](v_{\hfine})) \|_{L^2(\mcal[D][j])}^2 \le \| \sigma_{h,1}^{\nicefrac{1}{2}} \jump{v_h} \|_{L^2(\mcal[F][h,j]^I)}\ \| \sigma_{h,1}^{-\nicefrac{1}{2}} \average{ \mcal[R][1](\mcal[J][j](v_{\hfine})) }_{\nicefrac{1}{2}} \|_{L^2(\mcal[F][h,j]^I)}. \label{eq:HereAfterCS}
\end{align}
The second term on the right-hand side of~\eqref{eq:HereAfterCS} can be bounded by invoking Lemma \ref{lem:inversepoly} as follows:
\begin{align}
\| \sigma_{h,1}^{-\nicefrac{1}{2}} \average{ \mcal[R][1](\mcal[J][j](v_{\hfine})) }_{\nicefrac{1}{2}} \|_{L^2(\mcal[F][h,j]^I)}^2 & \le \sum_{\elem \in \mcal[T][h,j]} \| \sigma_{h,1}^{-\nicefrac{1}{2}} \mcal[R][1](\mcal[J][j](v_{\hfine})) \|_{L^2(\partial \elem)}^2 \\
& = C_{\sigma}^{-1}  \sum_{\elem \in \mcal[T][h,j]}\frac{\harmonic{h_{\elem}}}{ p^2} \|  \mcal[R][1](\mcal[J][j](v_{\hfine})) \|_{L^2(\partial \elem)}^2 \\
& \lesssim \sum_{\elem \in \mcal[T][h,j]} \| \mcal[R][1](\mcal[J][j](v_{\hfine})) \|_{L^2(\elem)}^2,
\label{eq:HereAfterTrace}
\end{align}
where we have used that $\harmonic{h_{\elem}} \le 2 h_{\elem}$ for any $\elem \in \mcal[T][\hfine]$. By inserting~\eqref{eq:HereAfterTrace} into~\eqref{eq:HereAfterCS} we obtain
\begin{equation}\label{eq:HereRh1leSigmah1JumpVh}
\| \mcal[R][1](\mcal[J][j](v_{\hfine})) \|_{L^2(\mcal[D][j])} \lesssim \| \sigma_{h,1}^{\nicefrac{1}{2}} \jump{v_h} \|_{L^2(\mcal[F][h,j]^I)}.
\end{equation}
Hence, by selecting $z_h = \Pi_h z$ in~\eqref{eq:ThisEquationInThisProof} and employing~\eqref{eq:HereRh1leSigmah1JumpVh},~\Cref{lem:interpDG_loc}, cf., also, \Cref{rmrk:GloApp},  and Theorem~\ref{thm_elliptic_regularity}, we obtain
\begin{align}
~~~\| v_h - \overline{v}_{h,j} \|_{L^2(\mcal[D][j])}^2 &
\lesssim \nicefrac{h}{p} \,\| \sigma_{h,1}^{\nicefrac{1}{2}} \jump{v_h} \|_{L^2(\mcal[F][h,j]^I)}\ \| z \|_{H^2(\mcal[D][j])} \\
& \lesssim \nicefrac{h}{p} \, \| \sigma_{h,1}^{\nicefrac{1}{2}} \jump{v_h} \|_{L^2(\mcal[F][h,j]^I)} \ (\| v_h - \overline{v}_{h,j} \|_{L^2(\mcal[D][j])}\ +\ \| \nabla z \|_{L^2(\mcal[D][j])}).
\label{eq:BoundForTheForcingTermOfDual}
\end{align}
Here, we note that $z$ also solves
\begin{equation}
\int_{\mcal[D][j]} \nabla z \cdot \nabla w\ \diff \xvec = \int_{\mcal[D][j]} (v_{\hfine} - \overline{v}_{h,j}) w\ \diff \xvec \quad \forall\, w \in \widetilde{V}_j,
\end{equation}
from which, by choosing $w = z$, upon application of the Cauchy--Schwarz inequality and Poincar\'{e}'s inequality, cf. Remark~\ref{rmrk:Poincare}, we get that $\| \nabla z \|_{L^2(\mcal[D][j])} \lesssim \| v_{\hfine} - \overline{v}_{h,j} \|_{L^2(\mcal[D][j])}$. By inserting this bound into~\eqref{eq:BoundForTheForcingTermOfDual} we obtain~\eqref{eq:HboundL2}. In order to show~\eqref{eq:HboundH1} we first select $w = \widetilde{v}_{h,j} \in \widetilde{V}_j$ in~\eqref{eq:HdefInChapter4InChapter4}; then, using the Cauchy--Schwarz inequality we obtain:
\begin{equation}
| \widetilde{v}_{h,j} |_{H^1(\mcal[D][j])} \lesssim \| \mcal[G][h,j](v_h) \|_{L^2(\mcal[D][j])}.
\end{equation}
Then, from the definition of $\mcal[G][h,j]$ given in~\Cref{def:H_of_v_h} we have:
\begin{align}\label{eq:ThisEquationInThisProof_2}
\| \mcal[G][h,j](v_h) \|_{L^2(\mcal[D][j])}^2 ~\lesssim ~\| \nabla_h v_h \|_{L^2(\mcal[T][h,j])}^2 + \| \mcal[R][1](\mcal[J][j](v_{\hfine})) \|_{L^2(\mcal[D][j])}^2.
\end{align}
The bound~\eqref{eq:HboundH1} is then obtained by inserting~\eqref{eq:HereRh1leSigmah1JumpVh} into~\eqref{eq:ThisEquationInThisProof_2}.
\end{proof}

We are now ready to investigate the relationship between the spaces $V_{\hfine},\ V_{\hcoarse}$, and $V_{\hlocal}$ introduced above. The following result concerns the approximation of a function $v_h \in V_{\hfine}$ with a coarse function $v_{\hcoarse} \in V_{\hcoarse}$; this represents an extension of the analogous result presented in~\cite[Lemma 5.1]{AnHoSm2016}.

\begin{lemma}\label{lem:vh_R0TvH}
For any $v_h \in V_h$ there exists a coarse function $v_{\hcoarse} \in V_{\hcoarse}$ such that
\begin{align}
& \| v_h - R_0^\top v_{\hcoarse} \|_{L^2(\mcal[D][j])} \lesssim \nicefrac{\hcoarse}{q}\ ( \| \nabla_h v_h \|_{L^2(\mcal[T][h,j])}^2 + \| \sigma_{h,1}^{\nicefrac{1}{2}}\jump{v_h} \|_{L^2(\mcal[F][h,j]^I)}^2)^{\nicefrac{1}{2}}, \label{eq:L2}\\
& | v_h - R_0^\top v_{\hcoarse} |_{H^1(\mcal[T][h,j])} \lesssim ( \| \nabla_h v_h \|_{L^2(\mcal[T][h,j])}^2 + \| \sigma_{h,1}^{\nicefrac{1}{2}}\jump{v_h} \|_{L^2(\mcal[F][h,j]^I)}^2)^{\nicefrac{1}{2}}, \label{eq:H1}
\end{align}
for $j=1,\dots,N_{\hcoarse}$, where $\mcal[T][h,j]$ and $\mcal[F][h,j]^I$ are as defined in~\eqref{eq:Thj}.
\end{lemma}

\begin{proof}
Let $v_h \in V_h$ and define $v_{\hcoarse}$ by $v_{\hcoarse}|_{\mcal[D][j]} := (\nicefrac{1}{|\mcal[D][j]|} \int_{\mcal[D][j]} v_h \diff \xvec) + (\Pi_{\hcoarse} (\widetilde{v}_{h,j}))|_{\mcal[D][j]}$, $j=1,\dots,N_{\hcoarse}$, where $\widetilde{v}_{h,j}$ is as defined in~\Cref{def:H_of_v_h} and $\Pi_{\hcoarse}$ denotes the global variant of the $hp$-approximant introduced in~\Cref{lem:interpDG_loc}, cf. also \Cref{rmrk:GloApp}, defined on the coarse space $V_{\hcoarse}$. Then, by noting~\Cref{rmrk:R_0T_Nest}, the application of the triangle inequality gives
\begin{align}
\| v_h - R_0^\top v_{\hcoarse}\|_{L^2(\mcal[D][j])} &=   \| v_h - v_{\hcoarse}\|_{L^2(\mcal[D][j])} \\
  & \lesssim \| v_h - \overline{v}_{h,j}\|_{L^2(\mcal[D][j])} \!+ \| \overline{v}_{h,j} - v_{\hcoarse}\|_{L^2(\mcal[D][j])} \!\\
& \lesssim \| v_h - \overline{v}_{h,j}\|_{L^2(\mcal[D][j])} \!+ \| \widetilde{v}_{h,j} - \Pi_{\hcoarse}(\widetilde{v}_{h,j})\|_{L^2(\mcal[D][j])},\!
\end{align}
with $\overline{v}_{h,j}$ as defined in~\Cref{th:Hbound}.
Employing~\Cref{lem:interpDG_loc} together with~\Cref{ass2}, cf.~\Cref{rmrk:GloApp}, gives
\begin{equation}
\| v_h -  R_0^\top v_{\hcoarse}\|_{L^2(\mcal[D][j])} \lesssim \| v_h - \overline{v}_{h,j}\|_{L^2(\mcal[D][j])}+ \nicefrac{\hcoarse}{q} \, \| \widetilde{v}_{h,j}\|_{H^1(\mcal[D][j])}.
\end{equation}
By applying Poincar\'e's inequality to $\widetilde{v}_{h,j} \in \widetilde{V}_j$, see also~\Cref{rmrk:Poincare}, and noting the bounds given in~\Cref{th:Hbound}, we immediately deduce inequality~\eqref{eq:L2} by observing that $\hfine \le \hcoarse$ and $q \le p$. In order to obtain~\eqref{eq:H1} we proceed as follows:
\begin{align}
| v_h - R_0^\top v_{\hcoarse} |_{H^1(\mcal[T][h,j])} & \lesssim |v_h|_{H^1(\mcal[T][h,j])} + |R_0^\top v_{\hcoarse}|_{H^1(\mcal[T][h,j])}\\
& =  |v_h|_{H^1(\mcal[T][h,j])} + |v_{\hcoarse}|_{H^1(\mcal[T][h,j])}.\label{eq:vh_vH_H1}
\end{align}
Thanks to the triangle inequality and observing that $v_{\hcoarse}|_{\mcal[D][j]} \in \mcal[P][q](\mcal[D][j]) \subset H^1(\mcal[D][j])$ we have that
\begin{align}\label{eq:vHH1lesssimTildevhjH1} 
\begin{aligned}
| v_{\hcoarse} |_{H^1(\mcal[T][\hfine,j])} & \lesssim | \Pi_{\hcoarse}(\widetilde{v}_{h,j}) - \widetilde{v}_{h,j} |_{H^1(\mcal[D][j])} + | \widetilde{v}_{h,j} |_{H^1(\mcal[D][j])} \\
& \lesssim | \Pi_{\hcoarse}(\widetilde{v}_{h,j}) - \widetilde{v}_{h,j} |_{H^1(\mcal[T][\hcoarse])} + | \widetilde{v}_{h,j} |_{H^1(\Om)}  \lesssim | \widetilde{v}_{h,j} |_{H^1(\Om)},
\end{aligned}
\end{align}
where we have used the bound stated in~\Cref{rmrk:GloApp} and Poincar\'e's inequality, cf. also~\Cref{rmrk:Poincare}. Inserting~\eqref{eq:vHH1lesssimTildevhjH1} into~\eqref{eq:vh_vH_H1} and noting~\Cref{th:Hbound} gives~\eqref{eq:H1}.
\end{proof}

\begin{remark}\label{lem:vh_R0TvH_Global}
By summing over all $\mcal[D][j] \in \mcal[T][\hcoarse]$, $j=1,\dots,N_{\hcoarse}$, the local bounds of~\Cref{lem:vh_R0TvH} give rise to the following global estimates:
\begin{align}
 \| v_h - R_0^\top v_{\hcoarse} \|_{L^2(\Om)} \lesssim \nicefrac{\hcoarse}{q} \, \normDG[v_h][h,1],
 \qquad | v_h - R_0^\top v_{\hcoarse} |_{H^1(\mcal[T][h])} \lesssim \normDG[v_h][h,1], \label{eq:GloH1}
\end{align}
which are in agreement with the analogous results developed in~\cite{AnHoSm2016}.
\end{remark}

Before proceeding with the analysis of $P_{ad}$ we also need the following result regarding the properties of the subdomain decomposition introduced in~\Cref{sec:ASPCG}.

\begin{lemma}\label{lem:decVi}
Given $v_h \in V_h$, there exists a unique decomposition, i.e., $v_h = \sum_{i=1}^{N_{\hlocal}} R_i^\top v_i$, with $v_i \in V_i\ i=1,\dots,N_\hlocal$, such that
\begin{equation}\label{eq:AhvvEqSumAiviviSumRiviRivi}
\mcal[A][h](v_h,v_h) = \sum_{i=1}^{N_\hlocal} \mcal[A][i](v_i,v_i) + \sum_{i,j=1, i\ne j}^{N_\hlocal} \mcal[A][h](R_i^\top v_i, R_j^\top v_j),
\end{equation}
and
\begin{align}
\Bigl| \sum_{i,j=1, i\ne j}^{N_\hlocal} \mcal[A][h](R_i^\top v_i, R_j^\top v_j) \Bigr| & \lesssim \normL[\sqrt{\rho}\ \nabla_h v_h][2][\mcal[T][h]]^2 + \sum_{i=1}^{N_\hlocal} \| \sigma_{h,\rho}^{\nicefrac{1}{2}} v_h \|_{L^2(\partial \Om_i)}^2.
\end{align}
\end{lemma}

\begin{proof}
Given $v_h \in V_h$ set $v_i := R_i v_h,\ i=1,\dots,N_\hlocal$; then,
$\mcal[A][h](R_i^\top v_i, R_j^\top v_j) = 0$ if $\partial \Om_i \cap \partial \Om_j = \emptyset$.
For $i,j=1,\dots,N_\hlocal$, we have
\begin{equation}
\Bigl|  \sum_{i,j=1, i\ne j}^{N_\hlocal} \mcal[A][i](R_i^{\top} v_i,R_j^{\top} v_j) \Bigr|  \lesssim  \sum_{i,j=1, i\ne j}^{N_\hlocal} | \mcal[A][i](R_i^{\top} v_i, R_j^{\top} v_j) | \quad \forall\, v_i \in V_i, v_j \in V_j.
\end{equation}
Now let $i \ne j$ be such that $\partial \Om_i \cap \partial \Om_j \ne \emptyset$ and write $\check{v}_i = R_i^\top v_i$ and $\check{v}_j = R_j^\top v_j$; then,
\begin{equation}\label{eq:Atildevi}
\Aa[h][\check{v}_i][\check{v}_j] = \int_{\Om} \Bigl[\rho \nabla \check{v}_i \cdot \mathcal{R}_{\rho}(\jump{\check{v}_j} ) + \rho \nabla \check{v}_j \cdot \mathcal{R}_{\rho}(\jump{\check{v}_i}) \Bigr]\diff \xvec + \int_{\mcal[F][h]}\sigma_{h,\rho}\jump{\check{v}_i}\cdot\jump{\check{v}_j}\diff s.
\end{equation}
By recalling the definition of $\mcal[R][\rho]$ given in~\eqref{eq:lifting_R_2}, the first term on the right-hand side of~\eqref{eq:Atildevi} can be written as
\begin{align}
\int_{\Om}  \rho  \nabla_h \check{v}_i \cdot \mathcal{R}_{\rho}(\jump{\check{v}_j} ) \diff \xvec & =  - \int_{\mcal[F][h]} \average{\rho \nabla_h \check{v}_i}_{\omega} \cdot \jump{\check{v}_j} \diff s,
\end{align}
since $\rho$ is piecewise constant and $\check{v}_j \in R_i^{\top} V_i \subseteq V_{\hfine}$. By observing that $\jump{\check{v}_j}|_F = \mathbf{0}$ for all $F\in \mcal[F][h]$ such that $F \cap \overline{\Om}_j = \emptyset$, and that $\average{\rho \nabla_h \check{v}_i}_{\omega}|_F = \mathbf{0}$ when $F \cap \overline{\Om}_i = \emptyset$, we have
\begin{align}
\int_{\Om} \rho  \nabla_h \check{v}_i  \cdot \mathcal{R}_{\rho}(\jump{\check{v}_j} ) \diff \xvec  & =  - \int_{\partial \Om_i \cap \partial \Om_j} \jump{\check{v}_j} \cdot \average{\rho \nabla_h \check{v}_i}_{\omega} \diff s \\
& \le \| \sigma_{h,\rho}^{\nicefrac{1}{2}} \jump{\check{v}_j} \|_{L^2(\partial \Om_i \cap \partial \Om_j)}^2\ +\ \| \sigma_{h,\rho}^{-\nicefrac{1}{2}} \average{\rho  \nabla_h \check{v}_i}_{\omega} \|_{L^2(\partial \Om_i \cap \partial \Om_j)}^2 \\
& \le \| \sigma_{h,\rho}^{\nicefrac{1}{2}} \jump{\check{v}_j} \|_{L^2(\partial \Om_j)}^2\ +\ \| \sigma_{h,\rho}^{-\nicefrac{1}{2}} \average{\rho  \nabla_h \check{v}_i}_{\omega} \|_{L^2(\partial \Om_i)}^2. \label{eq:AfterCS}
\end{align}
Here, the second term on the right-hand side of~\eqref{eq:AfterCS} can be bounded by applying~\Cref{lem:inversepoly} and noting that $\harmonic{\rho_{\elem}} \le 2 \rho_{\elem}$ and $\harmonic{h_{\elem}} \le 2 h_{\elem}$, as follows:
\begin{align} \label{eq:AfterTrace}
\begin{aligned}
\| \sigma_{h,\rho}^{-\nicefrac{1}{2}} \average{\rho  \nabla_h \check{v}_i}_{\omega} \|_{L^2(\partial \Om_i)}^2
& \le \sum_{\elem \subset \Om_i} \| \sigma_{h,\rho}^{-\nicefrac{1}{2}} \harmonic{\rho_{\elem}} \nabla_h \check{v}_i \|_{L^2(\partial \elem)}^2 \\
& = C_{\sigma}^{-1}  \sum_{\elem \subset \Om_i} \frac{\harmonic{h_{\elem}}}{\harmonic{\rho_{\elem}} p^2} \harmonic{\rho_{\elem}}^2 \|  \nabla_h \check{v}_i \|_{L^2(\partial \elem)}^2 \\
& \lesssim \sum_{\elem \subset \Om_i} \| \sqrt{\rho_{\elem}} \nabla_h \check{v}_i) \|_{L^2(\elem)}^2
 = \| \sqrt{\rho} \nabla_h \check{v}_i \|_{L^2(\Om_i)}^2.
\end{aligned}
\end{align}
Inserting~\eqref{eq:AfterTrace} into~\eqref{eq:AfterCS} gives
\begin{equation}\label{eq:BoundForTerm1}
\int_{\Om} \rho \nabla \check{v}_i \cdot \mathcal{R}_{\rho}(\jump{\check{v}_j} ) \diff \xvec \lesssim \| \sqrt{\rho}\ \nabla_h \check{v}_i \|_{L^2(\Om_i)}^2 + \| \sigma_{h,\rho}^{\nicefrac{1}{2}} \jump{\check{v}_j}\|_{L^2(\partial \Om_j)}^2.
\end{equation}
Similarly, we have that
\begin{align}\label{eq:term2}
\int_{\Om} \rho \nabla \check{v}_j \cdot \mathcal{R}_{\rho}(\jump{\check{v}_i} ) \diff \xvec \lesssim \| \sqrt{\rho}\ \nabla_h \check{v}_j \|_{L^2(\Om_j)}^2 +  \| \sigma_{h,\rho}^{\nicefrac{1}{2}} \jump{\check{v}_i}\|_{L^2(\partial \Om_i)}^2
\end{align}
and
\begin{equation}\label{eq:term3}
\int_{\mcal[F][h]} \sigma_{h,\rho}\jump{\check{v}_i}\cdot\jump{\check{v}_j} \diff s \lesssim \| \sigma_{h,\rho}^{\nicefrac{1}{2}} \jump{\check{v}_i}\|_{L^2(\partial \Om_i)}^2 + \| \sigma_{h,\rho}^{\nicefrac{1}{2}} \jump{\check{v}_j}\|_{L^2(\partial \Om_j)}^2.
\end{equation}
Substituting~\eqref{eq:BoundForTerm1},~\eqref{eq:term2} and~\eqref{eq:term3} into~\eqref{eq:Atildevi} we obtain
\begin{align}
\Aa[h][\check{v}_i][\check{v}_j] & \lesssim \| \sqrt{\rho}\ \nabla_h \check{v}_i \|_{L^2(\Om_i)}^2 + \|\sqrt{\rho}\ \nabla_h \check{v}_j \|_{L^2(\Om_j)}^2 \\
& \quad + \| \sigma_{h,\rho}^{\nicefrac{1}{2}} \jump{\check{v}_i}\|_{L^2(\partial \Om_i)}^2 + \| \sigma_{h,\rho}^{\nicefrac{1}{2}} \jump{\check{v}_j}\|_{L^2(\partial \Om_j)}^2.
\end{align}
The result follows by summing over $i,j=1,\dots,N_\hlocal,\ i \ne j$, and exploiting~\eqref{eq:Ns}.
\end{proof}

For the forthcoming analysis we also require an extension of the trace-inverse inequality introduced by Feng and Karakashian in~\cite{FeKa2001}; cf. also Smears \cite[Lemma 5]{Sm2018}, to which we refer for the proof.
\begin{lemma}[Trace inverse inequality]\label{lem:Trace}
Let $\mcal[T][h]$ and $\mcal[T][\hlocal]$ be a pair of nested polytopic grids. We assume that $\mcal[T][\hlocal]$ is obtained by agglomeration of elements of $\mcal[T][\hfine]$ and that both $\mcal[T][\hfine]$ and $\mcal[T][\hlocal]$ satisfy~\Cref{ass1}. Moreover, we assume that for each $\Om_i \in \mcal[T][\hlocal],\ i=1,\dots,N_\hlocal,$ there exists an $\mathbf{x}_{0,i} \in \Om_i$ such that $(\mathbf{x} - \mathbf{x}_{0,i}) \cdot \mathbf{n}_i \gtrsim \hlocal$ for all $\mathbf{x} \in \partial \Om_i$, where $\mathbf{n}_i$ is the unit outward normal vector to $\partial \Om_i$. Then, for any $v_h \in V_h$, writing $\mcal[F][h](\Om_i) := \{ F \in \mcal[F][h]\text{ such that }F \subset \Om_i, F \not\subset \partial \Om_i \}$, the following bound holds:
\begin{align}
\| v_h \|_{L^2(\partial \Om_i)}^2
& \lesssim \| \nabla_h v_h \|_{L^2(\Om_i)} \| v_h \|_{L^2(\Om_i)}
+ \nicefrac{1}{\hlocal} \| v_h \|_{L^2(\Om_i)}^2 \\
& \quad + \| \sigma_{h,1}^{\nicefrac{1}{2}} \jump{v_h} \|_{L^2(\mcal[F][h](\Om_i))} \| v_h \|_{L^2(\Om_i)}.
\end{align}
\end{lemma}

\section{Condition number estimates}\label{sec:ConNumEst}
In this section we derive an upper bound on the condition number of $P_{ad}$ by following the analysis presented in~\cite{ToWi2004}; see also~\cite{Li1987}. To this end, we show that the following three assumptions are satisfied.

\begin{assumption}[Local stability]\label{ass:LocSta}
There exists an $\alpha \in (0,2)$ such that
\begin{equation}
\Aa[\hfine][R_i^\top v_i][R_i^\top v_i] \le \alpha \Aa[i][v_i][v_i] \quad \forall\, v_i \in V_i,\ i=0,1,\dots,N_H.
\end{equation}
\end{assumption}
We point out that in our setting~\Cref{ass:LocSta} immediately follows with $\alpha = 1$ from the definition of $\Aa[i]$ given in~\Cref{sec:ASPCG}, cf.~\cite{AnHo2011}.

\begin{assumption}[Strengthened Cauchy--Schwarz inequality]\label{ass:StrCauSchIne}
There exist constants $\epsilon_{ij} \in [0,1]$, for $1 \le i,j \le N_{\hlocal}$, such that
\begin{equation}
| \Aa[h][R_i^\top v_i][R_j^\top v_j] | \le  \epsilon_{ij} \Aa[h][R_i^\top v_i][R_i^\top v_i]^{\nicefrac{1}{2}} \Aa[h][R_j^\top v_j][R_j^\top v_j]^{\nicefrac{1}{2}}
\end{equation}
for all $v_i \in V_i$, $v_j \in V_j$. Define $\Theta(\mathbf{\mcal[E]})$ to be the spectral radius of
$(\mathbf{\mcal[E]})_{ij} =  \{\epsilon_{ij}\}_{i,j=1,\ldots,N_{\hlocal}}$.
\end{assumption}
\Cref{ass:StrCauSchIne} immediately follows since each subdomain $\Om_i \in \mcal[T][\hlocal]$, $i=1,\dots,N_{\hlocal}$, can possess only a finite number of neighbours, cf.~\eqref{eq:Ns}. In particular, by observing that if $\partial \Om_i \cap \partial \Om_j = \emptyset$, then $\Aa[h][R_i^\top v_i][R_j^\top v_j] = 0$ for all $v_i \in V_i$,  $v_j \in V_j$, we deduce that
$\epsilon_{ij} = 0$ if $\partial \Om_i \cap \partial \Om_j = \emptyset$, $\epsilon_{ij}$ $= 1$, otherwise. Then $\Theta(\mathbf{\mcal[E]})$ is uniformly bounded by $(N_{\mathbb{S}}+1)$, where $N_{\mathbb{S}}$ is the maximum number of neighbours that each subdomain may possess, cf.~\eqref{eq:Ns}. This result ensures that a stable (in the sense of the energy norm) decomposition can be found for the local spaces and the coarse one.

\begin{assumption}[Stable decomposition]\label{ass:StaDec}
Each $v_h \in V_h$ admits a decomposition of the form $v_h = \sum_{i=0}^{N_{\hlocal}} R_i^\top v_i$, $v_i \in V_i$, $i=1,\dots,N_{\hlocal}$, and $v_0 \in V_0$, such that
\begin{equation}
\sum_{i=0}^{N_{\hlocal}} \Aa[i][u_i][u_i] \le C_{\sharp}^2 \Aa[h][u_h][u_h].
\end{equation}
\end{assumption}

Following~\cite[Theorem 2.7]{ToWi2004} the upper bound on the condition number of $P_{ad}$ is stated in the following theorem.
\begin{theorem}\label{thm:TosWil}
Supposing that~\Cref{ass:LocSta}--\Cref{ass:StaDec} hold, the condition number $K(P_{ad})$ of the additive Schwarz operator $P_{ad}$ is bounded as follows:
\begin{equation}
K(P_{ad}) \lesssim C_{\sharp}^2 \alpha (\Theta(\mathbf{\mcal[E]}) + 1),
\end{equation}
where $\alpha$, $\mathbf{\mcal[E]}$, and $C_{\sharp}$ are as defined in Assumptions~\ref{ass:LocSta},~\ref{ass:StrCauSchIne} and~\ref{ass:StaDec}, respectively.
\end{theorem}

Next we prove that \Cref{ass:StaDec} holds.
\begin{theorem}\label{thm:TheoremValidityAssumption3}
Let $v_h \in V_h$, and assume that the grid $\mcal[T][\hfine]$ satisfies Assumptions~\ref{ass1} and~\ref{ass2}. We also assume that $\mcal[T][\hlocal]$ is obtained by agglomeration of elements of $\mcal[T][\hfine]$, $\mcal[T][\hcoarse]$ is obtained by agglomeration of elements of $\mcal[T][\hlocal]$, and that both $\mcal[T][\hlocal]$ and $\mcal[T][\hcoarse]$ satisfy Assumptions~\ref{ass1} and~\ref{ass2}. Then,~\Cref{ass:StaDec} holds with
\begin{equation}
C_{\sharp}^2 \eqsim \Bigl[ \max_{j=1,\dots,N_{\hcoarse}} \Bigl( \frac{\overline{\rho}_j}{\underline{\rho}_j} \Bigr) \Bigr] \Bigl(\frac{p^{2}}{q} \frac{\hcoarse}{\hfine} + \frac{p^2}{q^2} \frac{\hcoarse^2}{\hfine \hlocal} \Bigr),
\end{equation}
where $\underline{\rho}_j = \min_{\xvec \in \mcal[D][j]} (\rho(\xvec))$ and $\overline{\rho}_j = \max_{\xvec \in \mcal[D][j]} (\rho(\xvec))$.
\end{theorem}

\begin{proof}
Given $v_h \in V_h$, we select $v_0 = v_\hcoarse$, where $v_\hcoarse \in V_\hcoarse$ is defined as in the proof of~\Cref{lem:vh_R0TvH}. Then, by employing~\Cref{lem:decVi}, $v_h - R_0^\top v_0$ can be uniquely decomposed as
$
v_h - R_0^\top v_0 = \sum_{i=1}^{N_\hlocal} R_i^\top v_i,
$
where $v_i = R_i (v_h - R_0^\top v_0)$, $i=1,\dots,N_\hlocal$, and
\begin{equation}\label{eq:A_vh_R0Tv0}
\Aa[h][v_h - R_0^\top v_0][ v_h - R_0^\top v_0] = \sum_{i=1}^{N_\hlocal} \Aa[i][v_i][v_i] + \sum_{i,j=1, i\ne j}^{N_\hlocal} \Aa[h][R_i^\top v_i][R_j^\top v_j].
\end{equation}
Adding $\Aa[0][v_0][v_0]$ to both sides of~\eqref{eq:A_vh_R0Tv0} we obtain the following inequality:
\begin{align}\label{eq:Sum_Ai_vi}
\begin{aligned}
\Bigl| \sum_{i=0}^{N_\hlocal} \Aa[i][v_i][v_i] \Bigr| & \le \Bigl| \Aa[h][v_h - R_0^\top v_0][v_h - R_0^\top v_0] \Bigr|   +    \Bigl| \Aa[0][v_0][v_0]  \Bigr| \\
& \quad + \Bigl| \sum_{i,j=1, i\ne j}^{N_\hlocal} \Aa[h][R_i^\top v_i][R_j^\top v_j] \Bigr| \\
& \equiv \RomanNumeralCaps{1} + \RomanNumeralCaps{2} + \RomanNumeralCaps{3}.
\end{aligned}
\end{align}
From the definition of $\Aa[0]$, cf.~\eqref{eq:A0}, we have that
\begin{align}\label{eq:B}
\begin{aligned}
\RomanNumeralCaps{2} & \le |\Aa[h][R_0^{\top} v_0 - v_h][R_0^{\top} v_0]| + |\Aa[h][v_h][R_0^{\top} v_0]| \\
& \le |\Aa[h][R_0^{\top} v_0 - v_h][R_0^{\top} v_0 - v_h]| + 2 |\Aa[h][R_0^{\top} v_0 - v_h][v_h]| + |\Aa[h][v_h][v_h]|.
\end{aligned}
\end{align}
Recalling the continuity of $\mcal[A][h]$, cf.~\Cref{lem:contcoerc_2}, and applying the triangle inequality and Young's inequality gives
\begin{align}
|\Aa[h][R_0^{\top} v_0 - v_h][v_h]| & \lesssim \normDG[v_h - R_0^{\top} v_0][h,\rho] \normDG[v_h][h,\rho]
 \lesssim \normDG[v_h - R_0^\top v_0][h,\rho]^2 + \normDG[v_h][h,\rho]^2.
\end{align}
Then, by inserting the above bound into~\eqref{eq:B} and using the continuity and coercivity of $\mcal[A][h]$, cf.~\Cref{lem:contcoerc_2}, we deduce that
\begin{equation}\label{eq:AplusB}
\RomanNumeralCaps{1} + \RomanNumeralCaps{2} \lesssim \normDG[v_h - R_0^\top v_0][h,\rho]^2 + \Aa[h][v_h][v_h].
\end{equation}
In particular, we observe that from the definition of $\| \cdot \|_{h,\rho}$ we have
\begin{align}\label{eq:TheNormDGofIt}
\normDG[v_h - R_0^\top v_0][h,\rho]^2  = \| \sqrt{\rho} \nabla_h (v_h - R_0^{\top}v_0) \|_{L^2(\mcal[T][h])}^2 + \| \sigma_{h,\rho}^{\nicefrac{1}{2}} \jump{v_h - R_0^{\top} v_0} \|_{L^2(\mcal[F][h])}^2.
\end{align}
Writing $\mcal[F][\hcoarse]$ to denote the set of faces of $\mcal[T][\hcoarse]$, and observing that $\mcal[F][\hcoarse] \subseteq \mcal[F][\hfine]$ since $	\mcal[T][\hcoarse] \subseteq \mcal[T][\hfine]$, the second term on the right-hand side
of~\eqref{eq:TheNormDGofIt} can be bounded as follows:
\begin{align}
\| \sigma_{h,\rho}^{\nicefrac{1}{2}} \jump{v_h \! -\! R_0^{\top} v_0} \|_{L^2(\mcal[F][h])}^2
& = \| \sigma_{h,\rho}^{\nicefrac{1}{2}} \jump{v_h - R_0^{\top} v_0} \|_{L^2(\mcal[F][h]\setminus \mcal[F][\hcoarse])}^2
 \!+ \| \sigma_{h,\rho}^{\nicefrac{1}{2}} \jump{v_h - R_0^{\top} v_0} \|_{L^2(\mcal[F][\hcoarse])}^2 \\
& =  \| \sigma_{h,\rho}^{\nicefrac{1}{2}} \jump{v_h} \|_{L^2(\mcal[F][h]\setminus \mcal[F][\hcoarse])}^2
+ \| \sigma_{h,\rho}^{\nicefrac{1}{2}} \jump{v_h - R_0^{\top} v_0} \|_{L^2(\mcal[F][\hcoarse])}^2\\
& \le \normDG[v_h][h,\rho]^2\ +\ \sum_{j=1}^{N_{\hcoarse}} \| \sigma_{h,\rho}^{\nicefrac{1}{2}} (v_h - R_0^{\top} v_0) \|_{L^2(\partial \mcal[D][j])}^2 \\
& \le \normDG[v_h][h,\rho]^2\ +\ \sum_{i=1}^{N_{\hlocal}} \| \sigma_{h,\rho}^{\nicefrac{1}{2}} (v_h - R_0^{\top} v_0) \|_{L^2(\partial \Om_i)}^2, \label{eq:BoundForDGNormOfIT}
\end{align}
where we have used that $\jump{R_0^{\top} v_0} = \mathbf{0}$ on each face $F \in \mcal[F][h] \setminus \mcal[F][\hcoarse]$ and, in the last step, the fact that $\mcal[T][\hlocal] \subseteq \mcal[T][\hcoarse]$; cf. \Cref{rmrk:assump_nested_meshes}.
Hence, inserting~\eqref{eq:BoundForDGNormOfIT} into~\eqref{eq:TheNormDGofIt} and employing~\Cref{lem:contcoerc_2}, inequality~\eqref{eq:AplusB} becomes
\begin{equation}
\RomanNumeralCaps{1} + \RomanNumeralCaps{2} \lesssim \| \sqrt{\rho}\ \nabla_h(v_h - R_0^\top v_0)\|_{L^2(\mcal[T][h])}^2
 +\ \sum_{i=1}^{N_{\hlocal}} \| \sigma_{h,\rho}^{\nicefrac{1}{2}} (v_h - R_0^{\top} v_0) \|_{L^2(\partial \Om_i)}^2 + \Bigr| \Aa[h][v_h][v_h] \Bigr|.
\end{equation}
From~\Cref{lem:decVi} we get
\begin{equation}\label{eq:BoundOfTermC}
 \RomanNumeralCaps{3} \lesssim \| \sqrt{\rho}\ \nabla_h(v_h - R_0^\top v_0)\|_{L^2(\mcal[T][h])}^2\ +\ \sum_{i=1}^{N_\hlocal} \| \sigma_{h,\rho}^{\nicefrac{1}{2}} (v_h - R_0^\top v_0)\|_{L^2(\partial \Om_i)}^2.
\end{equation}
Thereby,~\eqref{eq:Sum_Ai_vi} can be bounded as follows
\begin{align}\label{eq:Sum_Ai_vi_2}
\begin{aligned}
\Bigl| \sum_{i=0}^{N_\hlocal} \Aa[i][v_i][v_i] \Bigr| & \lesssim \, \Bigr| \Aa[h][v_h][v_h] \Bigr| + \| \sqrt{\rho}\ \nabla_h(v_h - R_0^\top v_0) \|_{L^2(\mcal[T][h])}^2 \\
& \quad + \sum_{i=1}^{N_\hlocal} \| \sigma_{h,\rho}^{\nicefrac{1}{2}} (v_h - R_0^\top v_0)\|_{L^2(\partial \Om_i)}^2 \\
& \equiv \RomanNumeralCaps{4} + \RomanNumeralCaps{5} + \RomanNumeralCaps{6}.
\end{aligned}
\end{align}
Thanks to~\Cref{lem:vh_R0TvH} we have that
\begin{align}\label{eq:circ2}
\begin{aligned}
\RomanNumeralCaps{5} & = \sum_{j=1}^{N_{\hcoarse}}\| \sqrt{\rho} \nabla_h(v_h - R_0^\top v_0 )\|_{L^2(\mcal[D][j])}^2 \lesssim \sum_{j=1}^{N_{\hcoarse}} \overline{\rho}_j \| \nabla_h(v_h - R_0^\top v_0 )\|_{L^2(\mcal[D][j])}^2 \\
& \lesssim \sum_{j=1}^{N_{\hcoarse}} \overline{\rho}_j  \Bigl[\| \nabla_h v_h \|_{L^2(\mcal[T][h,j])}^2 + \| \sigma_{h,1}^{\nicefrac{1}{2}} \jump{v_h} \|_{L^2(\mcal[F][h,j]^I)}^2 \Bigr] \\
& \lesssim \sum_{j=1}^{N_{\hcoarse}} \nicefrac{\overline{\rho}_j}{\underline{\rho}_j} \Bigl[\| \sqrt{\rho}\ \nabla_h v_h \|_{L^2(\mcal[T][h,j])}^2 + \| \sigma_{h,\rho}^{\nicefrac{1}{2}} \jump{v_h} \|_{L^2(\mcal[F][h,j]^I)}^2 \Bigr]\\
& \lesssim \max_{j=1,\dots,N_{\hcoarse}} \Bigl( \nicefrac{\overline{\rho}_j}{\underline{\rho}_j} \Bigr)  \normDG[v_h][h,\rho]^2 \lesssim \max_{j=1,\dots,N_{\hcoarse}} \Bigl( \nicefrac{\overline{\rho}_j}{\underline{\rho}_j} \Bigr) \Aa[h][v_h][v_h],
\end{aligned}
\end{align}
where we have used the coercivity bound from~\Cref{lem:contcoerc_2} in the last inequality. The bound on term $\RomanNumeralCaps{6}$ can be deduced by using the inverse trace inequality of~\Cref{lem:Trace}. To this end, we first observe that
\begin{equation}\label{eq:3lesssim}
\RomanNumeralCaps{6} \lesssim \sum_{i=1}^{N_\hlocal} \nicefrac{p^2 \max_{\{ \elem \subset \Om_i\}} \rho_{\elem} }{h} \| v_h - R_0^\top v_0 \|_{L^2(\partial \Om_i )}^2,
\end{equation}
where we have also employed the definition of $\sigma_{h,\rho}$ and the fact that $\harmonic{\rho_{\elem}}|_{F} \le 2 \rho_{\elem^{\pm}}$ for any $F \subset \partial \Om_i$, $F \subset \partial \elem^{\pm}$, for some $\elem^{\pm} \in \mcal[T][\hfine]$, which implies that $\harmonic{\rho_{\elem}}|_F \le 2 \max_{\{ \elem \subset \Om_i\}} \rho_{\elem}$ for all $F \in \mcal[F][\hfine]$ such that $F \subset \partial \Om_i$. Then, by applying~\Cref{lem:Trace} to each $\Om_i \in \mcal[T][\hlocal]$, $i=1,\dots,N_{\hlocal}$, from~\eqref{eq:3lesssim} we obtain the following bound:
\begin{align}
\RomanNumeralCaps{6} \lesssim \sum_{i=1}^{N_{\hlocal}} &\nicefrac{p^2 \max_{\{ \elem \subset \Om_i\}} \rho_{\elem} }{h} \Bigl[  \| \nabla_h (v_h - R_0^\top v_0) \|_{L^2(\Om_i)} \| v_h - R_0^\top v_0 \|_{L^2(\Om_i)} \\
& + \nicefrac{1}{\hlocal} \, \| v_h - R_0^\top v_0 \|_{L^2(\Om_i)}^2\\
& + \Bigl( \sum_{F \in \mcal[F][h](\Om_i)} \| \sigma_{h,1}^{\nicefrac{1}{2}} \jump{v_h - R_0^\top v_0} \|_{L^2(F)}^2 \Bigr)^{\nicefrac{1}{2}} \| v_h - R_0^\top v_0 \|_{L^2(\Om_i)} \Bigr].
\end{align}
Since $\mcal[T][\hlocal] \subseteq \mcal[T][\hcoarse]$, we denote by $\mcal[I][j] := \{ k: 1\le k \le N_{\hlocal},\ \Om_k \in \mcal[T][\hlocal]\text{ and } \Om_k \subset \mcal[D][j]\}$ the set of indices that correspond to the subdomains inside $\mcal[D][j] \in \mcal[T][\hcoarse]$, for all $j=1,\dots,N_{\hcoarse}$. Hence, $\mcal[I][j] \cap \mcal[I][k] = \emptyset$ for any $j \ne k,\ 1\le j,k \le N_{\hcoarse}$, and $\cup_{j=1}^{N_{\hcoarse}} \mcal[I][j] = \{1,\dots,N_{\hlocal}\}$. Then,
\begin{align}\label{eq:ThisThirdTermHere}
\begin{aligned}
\RomanNumeralCaps{6} & \lesssim \sum_{j=1}^{N_{\hcoarse}} \sum_{i \in \mcal[I][j]} \nicefrac{p^2 \max_{\{ \elem \subset \Om_i\}} \rho_{\elem} }{h} \Bigl[  \| \nabla_h (v_h - R_0^\top v_0) \|_{L^2(\Om_i)} \| v_h - R_0^\top v_0 \|_{L^2(\Om_i)} \\
& \qquad + \nicefrac{1}{\hlocal} \,\| v_h - R_0^\top v_0 \|_{L^2(\Om_i)}^2\\
& \qquad + \Bigl( \sum_{F \in \mcal[F][h](\Om_i)} \| \sigma_{h,1}^{\nicefrac{1}{2}} \jump{v_h - R_0^\top v_0} \|_{L^2(F)}^2 \Bigr)^{\nicefrac{1}{2}} \| v_h - R_0^\top v_0 \|_{L^2(\Om_i)} \Bigr] \\
& \lesssim \sum_{j=1}^{N_{\hcoarse}} \nicefrac{p^2 \overline{\rho}_j}{h} \Bigl[ \sum_{i \in \mcal[I][j]}  \| \nabla_h (v_h - R_0^\top v_0) \|_{L^2(\Om_i)} \| v_h - R_0^\top v_0 \|_{L^2(\Om_i)}  \\
& \qquad + \nicefrac{1}{\hlocal} \, \sum_{i \in \mcal[I][j]}  \| v_h - R_0^\top v_0 \|_{L^2(\Om_i)}^2 \\
& \qquad + \sum_{i \in \mcal[I][j]}  \Bigl( \sum_{F \in \mcal[F][h](\Om_i)} \| \sigma_{h,1}^{\nicefrac{1}{2}} \jump{v_h - R_0^\top v_0} \|_{L^2(F)}^2 \Bigr)^{\nicefrac{1}{2}} \| v_h - R_0^\top v_0 \|_{L^2(\Om_i)} \Bigr].
\end{aligned}
\end{align}
We now proceed by bounding each term present in the bracket in \eqref{eq:ThisThirdTermHere}; to this end, using the Cauchy--Schwarz inequality for sums, we get
\begin{align}
\sum_{i \in \mcal[I][j]} \| \nabla_h (v_h - & R_0^\top v_0) \|_{L^2(\Om_i)} \| v_h - R_0^\top v_0 \|_{L^2(\Om_i)} \\
& \le \Bigl( \sum_{i \in \mcal[I][j]} \| \nabla_h (v_h - R_0^\top v_0) \|_{L^2(\Om_i)}^2 \Bigr)^{\nicefrac{1}{2}} \Bigl( \sum_{i \in \mcal[I][j]} \| v_h - R_0^\top v_0 \|_{L^2(\Om_i)}^2 \Bigr)^{\nicefrac{1}{2}} \\
& = | v_h - R_0^\top v_0 |_{H^1(\mcal[T][h,j])} \| v_h - R_0^\top v_0 \|_{L^2(\mcal[D][j])}. \label{eq:BoundFirstTermHere}
\end{align}
Similarly, by noting that $\mcal[F][h](\Om_i)$ is the set of faces $F \in \mcal[F][\hfine]$ strictly contained in $\Om_i$, and therefore $\cup_{i \in \mcal[I][j]} \mcal[F][h](\Om_i) \subset  \mcal[F][h,j]^I$, we deduce that
\begin{align}
\sum_{i \in \mcal[I][j]} \Bigl(  \sum_{F \in \mcal[F][h](\Om_i)} & \| \sigma_{h,1}^{\nicefrac{1}{2}} \jump{v_h - R_0^\top v_0} \|_{L^2(F)}^2 \Bigr)^{\nicefrac{1}{2}} \| v_h - R_0^\top v_0 \|_{L^2(\Om_i)} \\
& \le \Bigl( \sum_{i \in \mcal[I][j]} \| \sigma_{h,1}^{\nicefrac{1}{2}} \jump{v_h - R_0^\top v_0} \|_{L^2(\mcal[F][h](\Om_i))}^2 \Bigr)^{\nicefrac{1}{2}} \Bigl( \sum_{i \in \mcal[I][j]} \| v_h - R_0^\top v_0 \|_{L^2(\Om_i)}^2 \Bigr)^{\nicefrac{1}{2}} \\
& \le \Bigl( \sum_{F \in \mcal[F][h,j]^I} \| \sigma_{h,1}^{\nicefrac{1}{2}} \jump{v_h - R_0^\top v_0} \|_{L^2(F)}^2 \Bigr)^{\nicefrac{1}{2}} \| v_h - R_0^\top v_0 \|_{L^2(\mcal[D][j])}.
\end{align}
Noting that
$\sum_{i \in \mcal[I][j]}  \| v_h - R_0^\top v_0 \|_{L^2(\Om_i)}^2 = \| v_h - R_0^\top v_0 \|_{L^2(\mcal[D][j])}^2$ gives
\begin{align}
\RomanNumeralCaps{6} \lesssim \sum_{j=1}^{N_{\hcoarse}} & \nicefrac{p^2 \overline{\rho}_j}{h} \Bigl[ |v_h - R_0^\top v_0|_{H^1(\mcal[T][h,j])} \| v_h - R_0^\top v_0 \|_{L^2(\mcal[D][j])}
 + \nicefrac{1}{\hlocal} \, \| v_h - R_0^\top v_0 \|_{L^2(\mcal[D][j])}^2\\
& + \Bigl( \sum_{F \in \mcal[F][h,j]^I} \| \sigma_{h,1}^{\nicefrac{1}{2}} \jump{v_h - R_0^\top v_0} \|_{L^2(F)}^2 \Bigr)^{\nicefrac{1}{2}} \| v_h - R_0^\top v_0 \|_{L^2(\mcal[D][j])} \Bigr]. \label{eq:AnotherThirdTermHere}
\end{align}
The last term on the right-hand side of~\eqref{eq:AnotherThirdTermHere} can rewritten as
\begin{align}
\Bigl( \sum_{F \in \mcal[F][h,j]^I}  \| \sigma_{h,1}^{\nicefrac{1}{2}} \jump{v_h - R_0^\top v_0} & \|_{L^2(F)}^2 \Bigr)^{\nicefrac{1}{2}} \| v_h - R_0^\top v_0 \|_{L^2(\mcal[D][j])} \\
& = \Bigl( \sum_{F \in \mcal[F][h,j]^I} \| \sigma_{h,1}^{\nicefrac{1}{2}} \jump{v_h} \|_{L^2(F)}^2 \Bigr)^{\nicefrac{1}{2}} \| v_h - R_0^\top v_0 \|_{L^2(\mcal[D][j])} \\
& = \| \sigma_{h,1}^{\nicefrac{1}{2}} \jump{v_h} \|_{L^2(\mcal[F][h,j]^I)}   \| v_h - R_0^\top v_0 \|_{L^2(\mcal[D][j])};
\end{align}
 here we observe that $\jump{R_0^{\top} v_0}|_F = \mathbf{0}$ on each $F \in \mcal[F][h,j]^I$, since $\mcal[T][\hfine]$ and $\mcal[T][\hcoarse]$ are nested. 
Then, by employing the above estimate together with ~\Cref{lem:vh_R0TvH}, we deduce that
\begin{align}\label{eq:circ3}
\begin{aligned}
\RomanNumeralCaps{6} & \lesssim \sum_{j=1}^{N_{\hcoarse}} \Bigl[ \frac{p^2 \overline{\rho}_j}{h} \Bigl( \frac{\hcoarse}{q} + \frac{1}{q^2} \frac{\hcoarse^2}{ \hlocal} \Bigr) \Bigl(\| \nabla_h v_h \|_{L^2(\mcal[T][h,j])}^2 + \| \sigma_{h,1}^{\nicefrac{1}{2}} \jump{v_h} \|_{L^2(\mcal[F][h,j]^I)}^2 \Bigr) \Bigr] \\
& \lesssim \sum_{j=1}^{N_{\hcoarse}} \frac{\overline{\rho}_j}{\underline{\rho}_j} \Bigl( \frac{p^{2}}{q} \frac{\hcoarse}{h} + \frac{p^2}{q^2} \frac{\hcoarse^2}{h \hlocal} \Bigr) \Bigl(\| \sqrt{\rho}\ \nabla_h v_h \|_{L^2(\mcal[T][h,j])}^2 + \| \sigma_{h,\rho}^{\nicefrac{1}{2}} \jump{v_h} \|_{L^2(\mcal[F][h,j]^I)}^2\Bigr) \\
& \lesssim \max_{j=1,\dots,N_{\hcoarse}} \Bigl( \frac{\overline{\rho}_j}{\underline{\rho}_j} \Bigr) \Bigl(\frac{p^{2}}{q} \frac{\hcoarse}{h} + \frac{p^2}{q^2} \frac{\hcoarse^2}{h \hlocal} \Bigr) \Aa[h][v_h][v_h],
\end{aligned}
\end{align}
where we have also made use of the coercivity bound of~\Cref{lem:contcoerc_2} in the last inequality. Inserting the estimates~\eqref{eq:circ2} and~\eqref{eq:circ3} into~\eqref{eq:Sum_Ai_vi_2} we obtain the desired result.
\end{proof}


\begin{remark}\label{rmrk:Cond}
According to the statement of~\Cref{thm:TheoremValidityAssumption3}, given that Assumptions~\ref{ass:LocSta} and~\ref{ass:StrCauSchIne} hold, using~\Cref{thm:TosWil} we deduce that
\begin{equation}\label{eq:cond}
K(P_{ad}) \lesssim \max_{1\le j \le N_{\hcoarse}} \Bigl( \frac{\overline{\rho}_j}{\underline{\rho}_j} \Bigr) \Bigl(\frac{p^{2}}{q} \frac{\hcoarse}{h} + \frac{p^2}{q^2} \frac{\hcoarse^2}{h \hlocal} \Bigr)(N_{\mathbb{S}}+1).
\end{equation}
In particular, in the lowest order case, i.e., when $p=q=1$, we have
$K_h(P_{ad}) \lesssim \nicefrac{\hcoarse^2}{h \hlocal}$,
which is in agreement with the corresponding bound derived in~\cite{DrKr2016}. On the other hand if the size of the coarse subdomain and fine meshes are fixed, we deduce that
$K_p(P_{ad}) \lesssim \nicefrac{p^{2}}{q}$.
Moreover, we also observe that if the diffusion coefficient $\rho$ is constant on each subdomain $\mcal[D][j],\ j=1,\dots,N_{\hcoarse}$, then the condition number is independent of the jump in $\rho$.
\end{remark}

\begin{remark}\label{rmrk:Cond2}
We remark that~\Cref{ass:Dj1} is needed in order to obtain the local estimates of~\Cref{lem:vh_R0TvH}, which allow to bound $K(P_{ad})$ with $\max_{1\le j \le N_{\hcoarse}} ( \nicefrac{\overline{\rho}_j}{\underline{\rho}_j})$, cf.~\eqref{eq:cond}. However, if the diffusion coefficient $\rho$ is constant on $\Omega$, we point out that the analysis can be simplified by employing the global estimates of~\Cref{lem:vh_R0TvH_Global} without making~\Cref{ass:Dj1}. We refer to~\cite{AnHoSm2016} for further details.
\end{remark}

\section{Numerical results}\label{sec:NumRes}
In this section, we present a series of numerical experiments to demonstrate the sharpness of the condition number bounds stated in Remarks~\ref{rmrk:Cond} and \ref{rmrk:Cond2}. Throughout this section we solve~\eqref{eq:dG} by using the additive Schwarz Preconditioned Conjugate Gradient (ASPCG) method; here, we report the number of iterations needed to reduce the Euclidean norm of the relative residual vector below a tolerance of $10^{-8}$, based on starting from the trivial initial guess. Furthermore, we estimate the condition number $K(P_{ad})$ by using the extreme-eigenvalue-estimate based on the ASPCG iterations. For the sake of simplicity of the presentation, we only consider the massively parallel case, i.e., when $\mcal[T][\hlocal] = \mcal[T][\hfine]$. Furthermore, here we select the penalty parameter $C_{\sigma} = 10$.

\subsection{Example 1}\label{sec:TeCa6}

\begin{figure}[t]
\centering
\begin{tabular}{cc}
\includegraphics[width=0.35\textwidth]{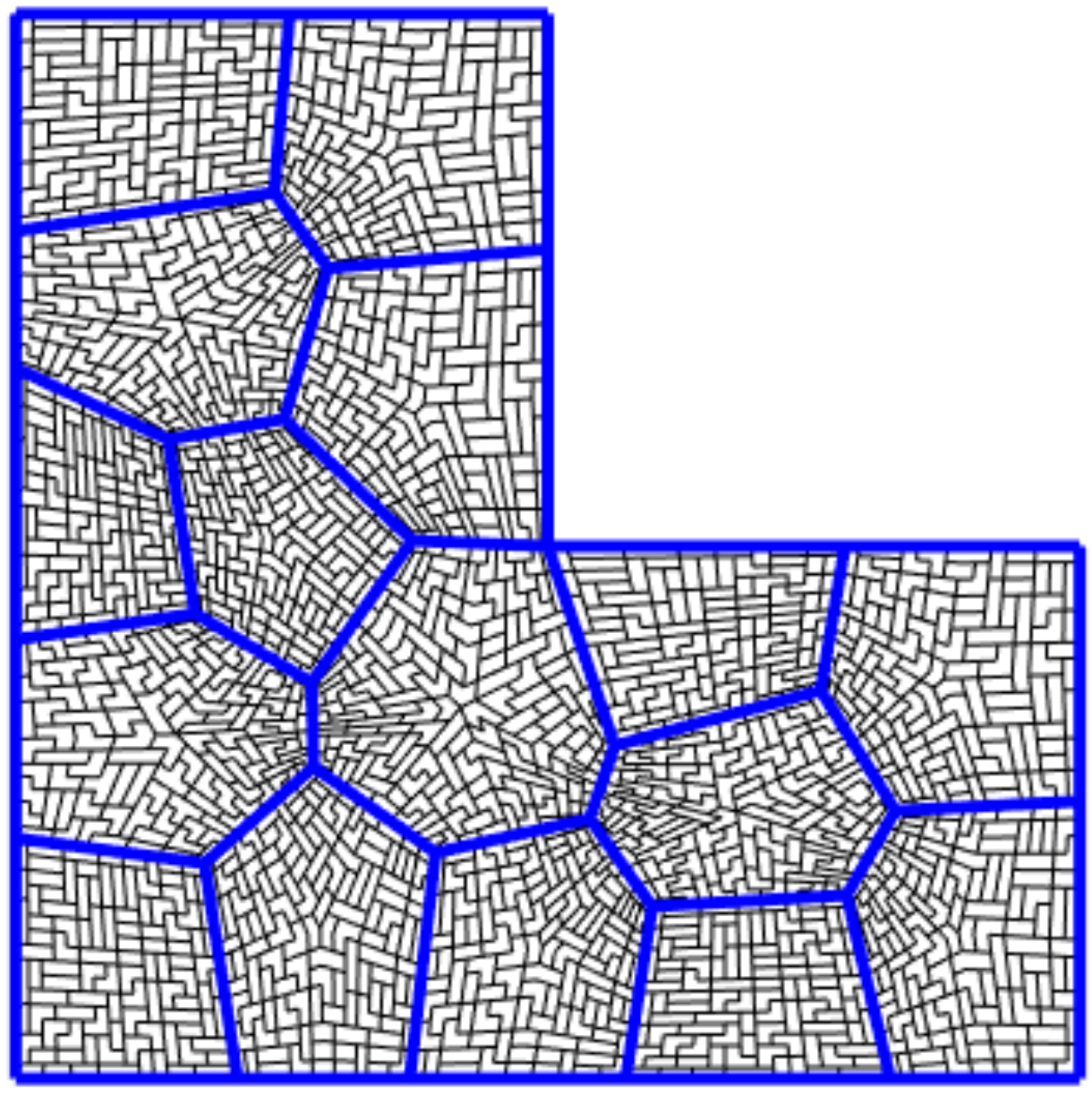}  &
\includegraphics[width=0.35\textwidth]{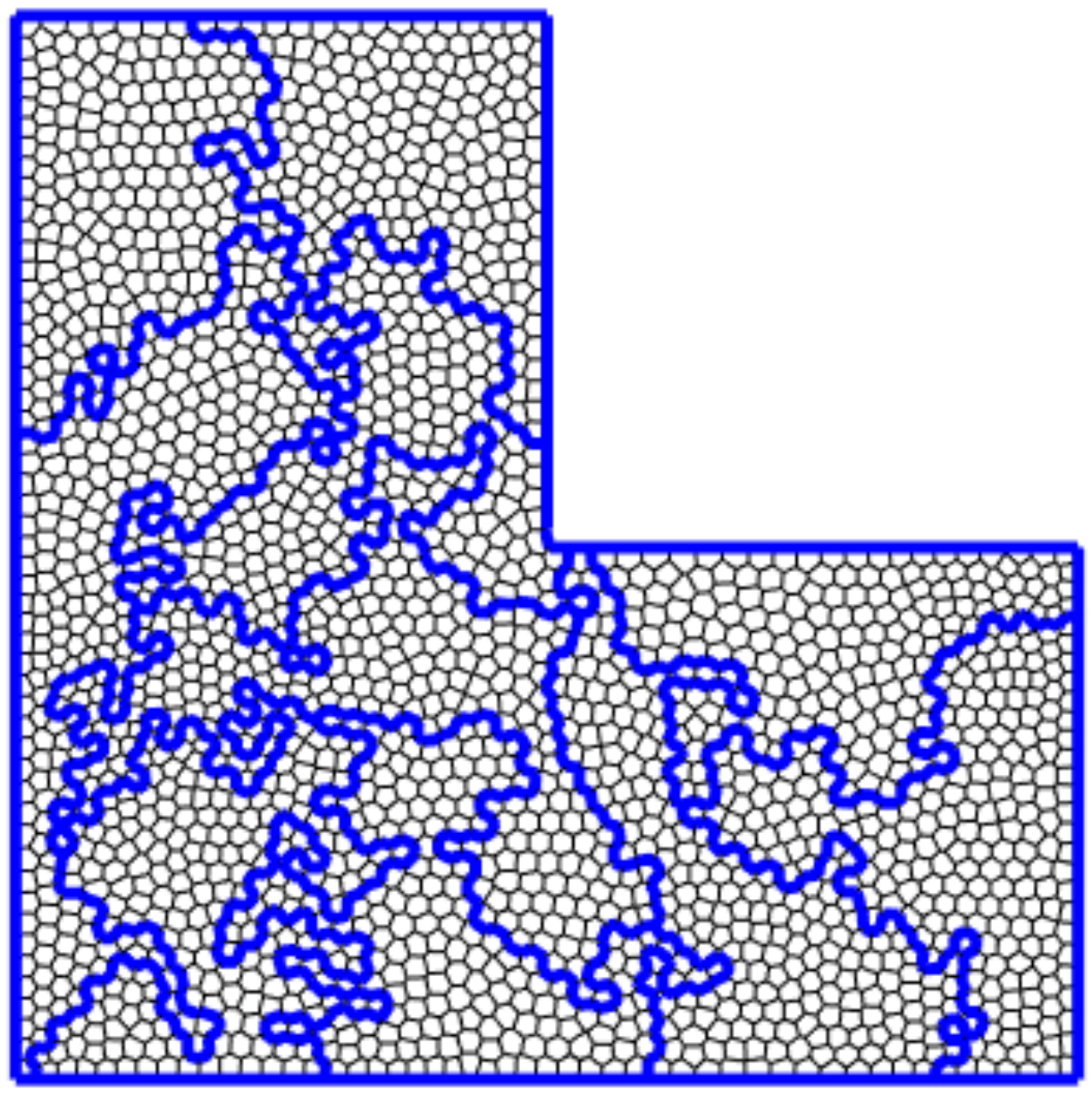}
\end{tabular}
\caption{Example 1. Nested polygonal grids $\mcal[T][h]$ (thin) and $\mcal[T][\hcoarse]$ (thick) on an L-shaped domain, when the elements of $\mcal[T][\hcoarse]$ are convex (left) and non-convex (right).}\label{fig:LshapedDomain}
\end{figure}

In this first example we investigate the dependence of the condition number of $P_{ad}$ on the diffusion coefficient $\rho$. Based on \Cref{rmrk:Cond}, we expect the condition number of the preconditioned system to be dependent on the choice of the coarse grid $\mcal[T][\hcoarse]$. More precisely, if $\mcal[T][\hcoarse]$ is chosen to be aligned with the discontinuities of $\rho$, then $K(P_{ad})$ should be independent of the jump in the coefficient $\rho$; otherwise it will depend on the maximum ratio between the maximum and the minimum value of $\rho$ present inside the subdomains $\mcal[D][j] \in \mcal[T][\hcoarse]$. To verify this behavior, we consider two experiments based on fine/coarse grids $\mcal[T][h]/\mcal[T][\hcoarse]$, respectively, where \mcal[T][\hcoarse] is a Voronoi polygonal grid on the L-shaped domain $\Om$ depicted in Figure~\ref{fig:LshapedDomain} with $16$ polygonal elements and $\mcal[T][\hfine]$ is obtained by successive refinement of elements of $\mcal[T][\hcoarse]$. Here, we observe that the elements of $\mcal[T][\hcoarse]$ are convex and satisfy~\Cref{ass:Dj1}, cf.~\Cref{fig:LshapedDomain} (left). Moreover, we choose the polynomial degrees to be either $p=q=1$ or $p=q=2$. In the first experiment we fix $\rho|_{ \mcal[D][j] } = \rho_{o} = 1$ on the elements $\mcal[D][j]\in \mcal[T][\hcoarse]$ with odd index $j$ and set $\rho|_{ \mcal[D][j] } = \rho_{e} \in \{10^1,10^2,\dots,10^6\}$, in each test case, on the polygonal subdomains with even index $j$. The results shown in the first two lines of~\Cref{tab:RhoAlignedAndNotAligned} confirm the independence with respect to the jumps of $\rho$ when those jumps are aligned with the subdomains of $\mcal[T][\hcoarse]$. In the second experiment we proceed similarly, but here we take different values of $\rho$ on odd and even polygonal elements $\elem \in \mcal[T][h]$: in this way $\mcal[T][\hcoarse]$ is not aligned with the discontinuities of $\rho$, and hence the ratio between the maximum and the minimum value of $\rho$ inside the polygonal subdomains $\mcal[D][j] \in \mcal[T][\hcoarse]$ is given by $\rho_{e}$. As expected from the theory, the results presented in the last two lines of~\Cref{tab:RhoAlignedAndNotAligned} show that the condition number of $P_{ad}$ grows linearly with $\rho_{e}$.
Finally, we repeat the same set of experiments by first selecting $\mcal[T][\hfine]$ as a Voronoi polygonal grid consisting of $2000$ elements and, subsequently, define the coarse grid $\mcal[T][\hcoarse]$ by successive agglomerations of elements of $\mcal[T][\hfine]$. The agglomeration is undertaken based on employing \emph{Metis}, cf.~\cite{KaKu2009}. As shown in~\Cref{fig:LshapedDomain} (right), the elements of $\mcal[T][\hcoarse]$ are clearly non-convex in this example. Although~\Cref{ass:Dj1} is not satisfied in this case, the results of~\Cref{tab:RhoAlignedAndNotAlignedNonCOnvex} illustrate analogous behavior to that observed in the previous setting when the coarse elements were convex.

\begin{table}[t!]
\centering
\footnotesize
\begin{tabular}{l l||llllllll}
 & & \multicolumn{5}{l}{$\rho_{e}\ \to$} \\
\hhline{~~------}
 &     & $10$  & $10^2$  & $10^3$  & $10^4$ & $10^5$ & $10^6$  \\
\hline
\hline
Aligned & $p=1$   & $2.31\cdot 10^2$  & $2.33\cdot 10^2$  & $2.34\cdot 10^2$  & $2.34\cdot 10^2$ & $2.34\cdot 10^2$ & $2.34\cdot 10^2$  \\
    & $p=2$   & $6.12\cdot 10^2$  & $6.14\cdot 10^2$  & $6.14\cdot 10^2$  & $6.12\cdot 10^2$ & $6.11\cdot 10^2$ & $6.10\cdot 10^2$  \\
\hline
\hline
Not Aligned & $p=1$  & $2.92\cdot 10^2$  & $1.04\cdot 10^3$  & $7.73\cdot 10^3$  & $7.42\cdot 10^4$ & $7.39\cdot 10^5$ & $7.38\cdot 10^6$  \\
 & $p=2$  & $6.54\cdot 10^2$  & $2.26\cdot 10^3$  & $1.71\cdot 10^4$  & $1.65\cdot 10^5$ & $1.64\cdot 10^6$ & $1.64\cdot 10^7$  \\
\end{tabular}

\vspace*{0.3cm}

\caption{Example 1. Condition number $K(P_{ad})$ as a function of the maximum jump in $\rho$ when the polygonal elements of $\mcal[T][\hcoarse]$ are convex and $\mcal[T][\hcoarse]$ is aligned (top) and not aligned (bottom) with the discontinuities of $\rho$.}
\label{tab:RhoAlignedAndNotAligned}
\end{table}

\begin{table}[t!]
\centering
\footnotesize
\begin{tabular}{l l||llllllll}
 & & \multicolumn{5}{l}{$\rho_{e}\ \to$} \\
\hhline{~~------}
 &     & $10$  & $10^2$  & $10^3$  & $10^4$ & $10^5$ & $10^6$  \\
\hline
\hline
Aligned & $p=1$   & $9.65\cdot 10^2$  & $1.15\cdot 10^3$  & $1.20\cdot 10^3$  & $1.21\cdot 10^3$ & $1.21\cdot 10^3$ & $1.21\cdot 10^3$  \\
    & $p=2$   & $2.47\cdot 10^3$  & $2.86\cdot 10^3$  & $3.18\cdot 10^3$  & $3.25\cdot 10^3$ & $3.26\cdot 10^3$ & $3.26\cdot 10^3$  \\
\hline
\hline
Not Aligned & $p=1$  & $8.02\cdot 10^2$  & $2.09\cdot 10^3$  & $1.68\cdot 10^4$  & $1.65\cdot 10^5$ & $1.64\cdot 10^6$ & $1.64\cdot 10^7$  \\
 & $p=2$  & $2.36\cdot 10^3$  & $6.10\cdot 10^3$  & $4.74\cdot 10^4$  & $4.62\cdot 10^5$ & $4.61\cdot 10^6$ & $4.61\cdot 10^7$  \\
\end{tabular}

\vspace*{0.3cm}

\caption{Example 1. Condition number $K(P_{ad})$ as a function of the maximum jump in $\rho$ when the polygonal elements of $\mcal[T][\hcoarse]$ are non-convex and $\mcal[T][\hcoarse]$ is aligned (top) and not aligned (bottom) with the discontinuities of $\rho$.}
\label{tab:RhoAlignedAndNotAlignedNonCOnvex}
\end{table}


\subsection{Example 2}\label{sec:TeCa3}

\begin{figure}[t!]
\centering
\begin{tabular}{cccc}
 $\mcal[T][\hfine]$ & $\hcoarse = 2 h$ & $\hcoarse = 4 h$ & $\hcoarse = 8 h$\\
\subfloat{
	  \raisebox{-.45\height}{\includegraphics[width=0.2\textwidth]{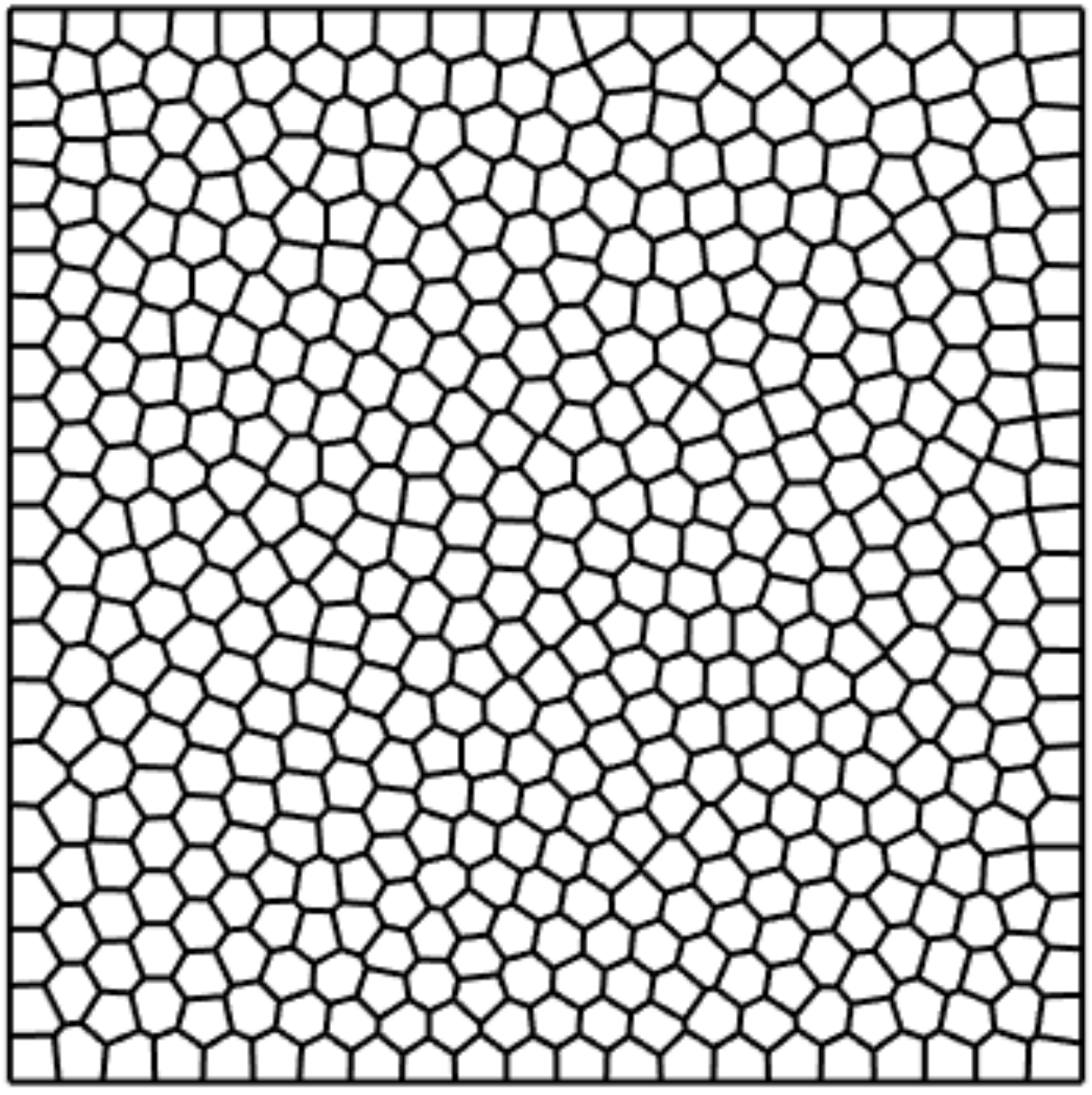}}}  & \subfloat{
          \raisebox{-.45\height}{\includegraphics[width=0.2\textwidth]{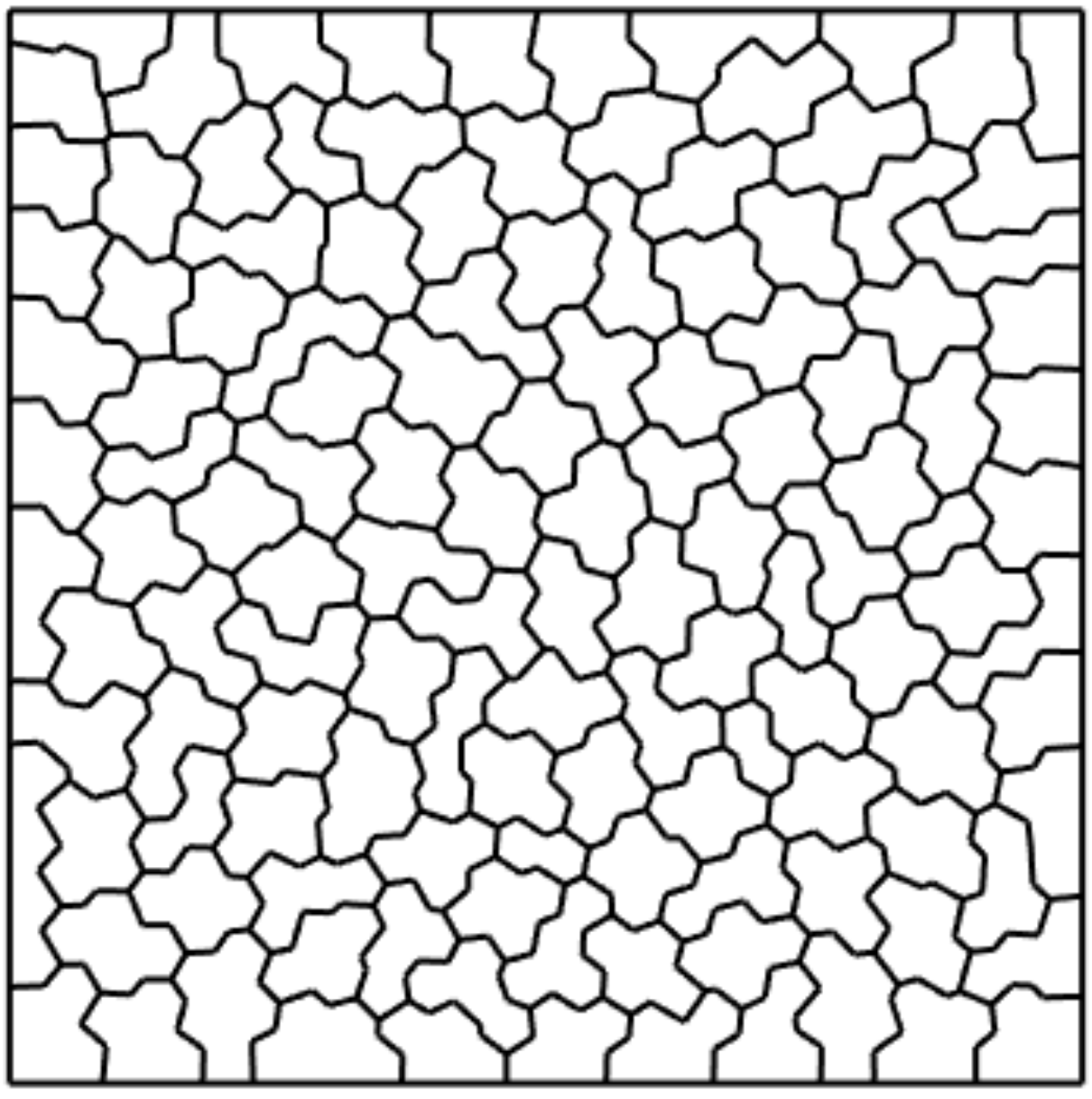}}} & \subfloat{
          \raisebox{-.45\height}{\includegraphics[width=0.2\textwidth]{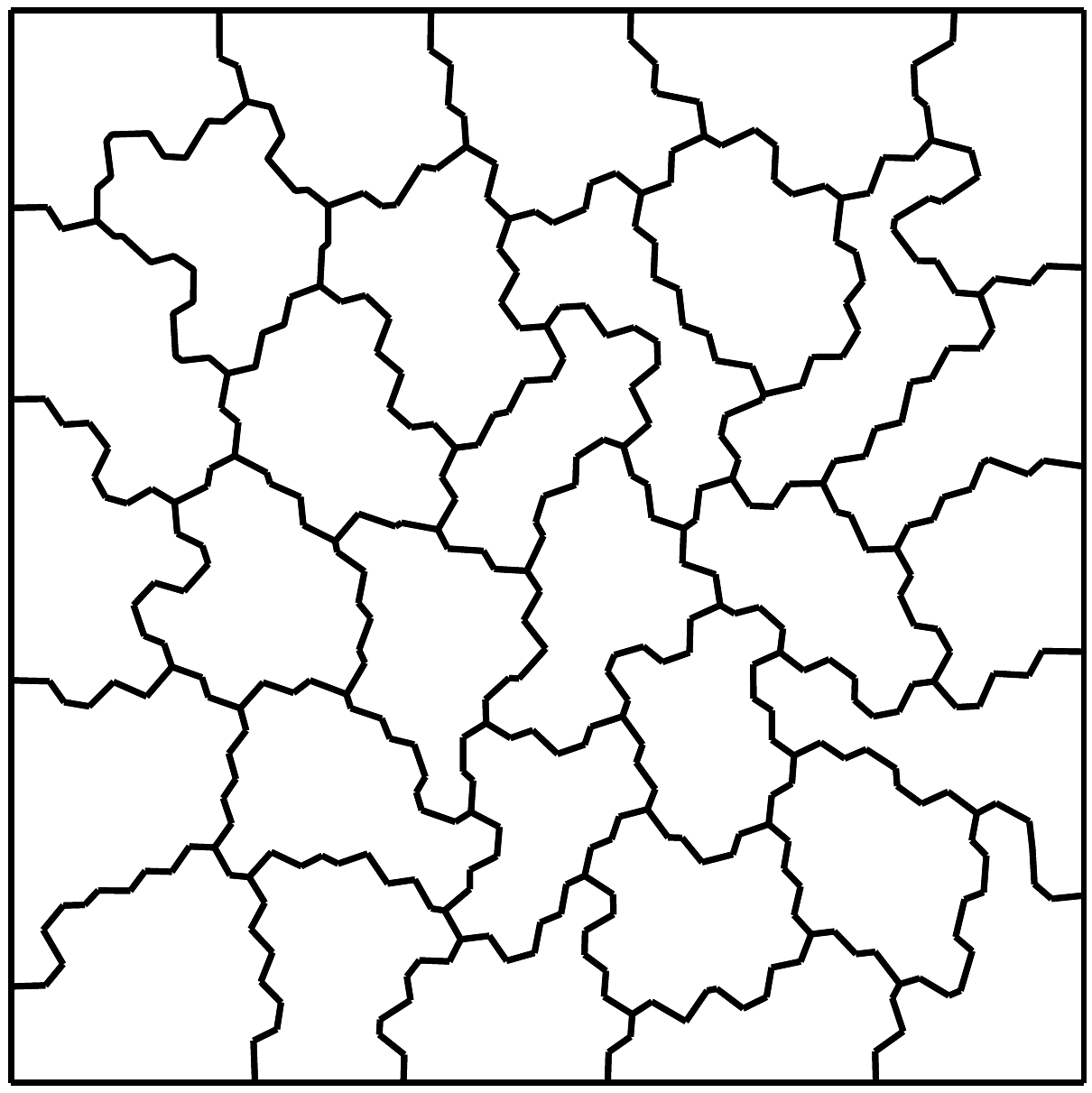}}} & \subfloat{
          \raisebox{-.45\height}{\includegraphics[width=0.2\textwidth]{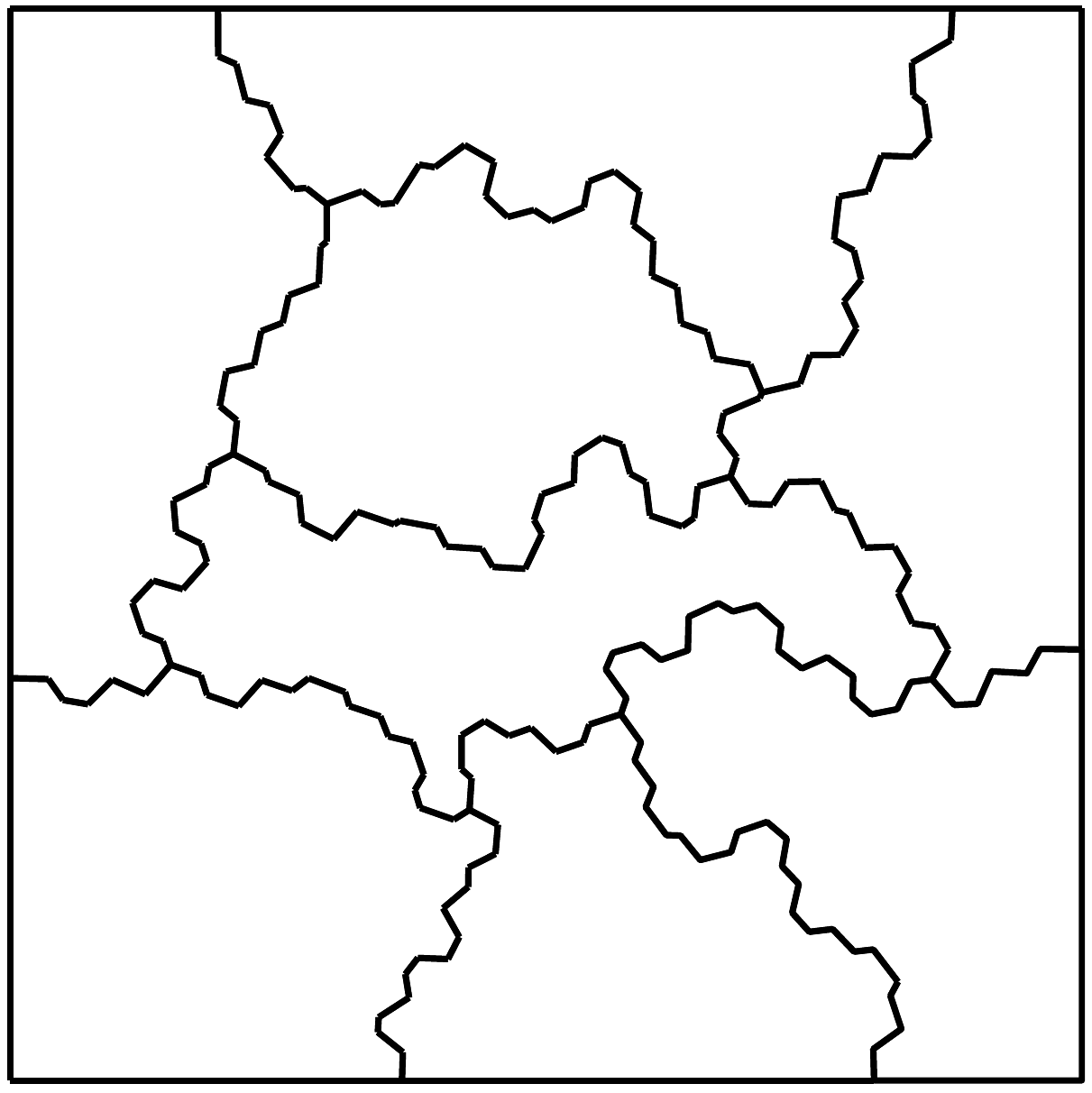}}}\\
\end{tabular}
\caption{Example 2. Example of a sequence of nested polygonal grids.}
\label{fig:PolyGrids}
\end{figure}

In this example, we investigate the performance of the proposed ASPCG algorithm on a set of
Voronoi polygonal fine grids $\mcal[T][h]$, where $\Om=(0,1)^2$ and $\rho=1$. For each grid $\mcal[T][h]$ we construct a sequence of nested polygonal grids $\mcal[T][\hcoarse]$ obtained by successive levels of agglomeration, cf.~\cite{AnHoHuSaVe2017}. 
For each fine Voronoi grid of size $h$ the agglomeration process has been performed in order to ensure that the size of the coarser partitions is approximately $\hcoarse = 2h,\ 4h, \dots$, cf. \Cref{fig:PolyGrids}, for example. In~\Cref{tab:ASPCG_P1_poly,tab:ASPCG_P3_poly} we report the condition number and the iteration counts for the proposed ASPCG algorithm for $p=q=1$ and $p=q=3$, respectively. Here, we clearly observe that the condition number and the iteration counts grow quadratically and linearly, respectively, as $\hfine$ tends to zero for fixed $\hcoarse$. Moreover, if the ratio of $\hfine$ and $\hcoarse$ is kept fixed, then we observe that the condition number and iteration counts are approximately constant; cf. the diagonals and subdiagonals of~\Cref{tab:ASPCG_P1_poly,tab:ASPCG_P3_poly}. This behavior is in agreement with the theoretical bound stated in \Cref{rmrk:Cond}, where $K(P_{ad}) = \mcal[O]( \nicefrac{\hcoarse^2}{h^2} ).$


\bgroup
\def\arraystretch{1.8}

\begin{table}[t!]
\centering
\footnotesize
\begin{tabular}{l||lllll}
& \multicolumn{5}{l}{$\mcal[T][\hfine]\ \to$} \\
\hhline{~-----}
$\downarrow \mcal[T][\hcoarse]$  & $h=8\overline{h}$  & $h=4\overline{h}$  & $h=2\overline{h}$  & $h=\overline{h}$  \\
\hhline{------}
$\hcoarse = 16\overline{h}$  & 20.70 (45)  & 72.31 (86)  & 269.70 (163)  & 818.09 (289)   \\
$\hcoarse = 8\overline{h}$    & -   & 21.89 (46)  & 73.42 (86)  & 261.36 (163)   \\
$\hcoarse = 4\overline{h}$    & -   & -   & 20.91 (46)  & 83.77 (91)  \\
$\hcoarse = 2\overline{h}$    & -   & -   & -   & 23.08 (48) \\
\end{tabular}

\vspace*{0.3cm}

\caption{Example 2. Condition number (and iteration counts): nested polygonal grids, with $p=q=1$. Here, $\overline{h}$ is the diameter of a grid with $N_{\hfine}=4096$ elements.}
\label{tab:ASPCG_P1_poly}
\end{table}

\begin{table}[t!]
\centering
\footnotesize
\begin{tabular}{l||lllll}
& \multicolumn{5}{l}{$\mcal[T][\hfine]\ \to$} \\
\hhline{~-----}
$\downarrow \mcal[T][\hcoarse]$  & $h=8\overline{h}$  & $h=4\overline{h}$  & $h=2\overline{h}$  & $h=\overline{h}$  \\
\hhline{------}
$\hcoarse = 16\overline{h}$  & 88.63 (80)  & 291.77 (145)  & 1137.55 (276)  & 3241.64 (513)   \\
$\hcoarse = 8\overline{h}$    & -   & 102.96 (82)  & 278.15 (140)  & 949.21 (271)   \\
$\hcoarse = 4\overline{h}$    & -   & -   & 90.30 (79)  & 343.19 (148)  \\
$\hcoarse = 2\overline{h}$    & -   & -   & -   & 104.24 (82) \\
\end{tabular}

\vspace*{0.3cm}

\caption{Example 2. Condition number (and iteration counts): nested polygonal grids with, $p=q=3$. Here, $\overline{h}$ is the diameter of a grid with $N_{\hfine}=4096$ elements. }
\label{tab:ASPCG_P3_poly}
\end{table}

\egroup

\subsection{Example 3}\label{sec:TeCa7}
We now consider the performance of the ASPCG algorithm on tetrahedral meshes in three dimensions. To this end, we set $\Om = (0,1)^3$ and $\rho=1$; furthermore, the elements of the coarse mesh are general-shaped polyhedra obtained by successive agglomeration, cf. the previous example. The results for $p=q=1$ and $p=q=3$ are reported in~\Cref{tab:3DTetrahedra_CondAndIter_p1,tab:3DTetrahedra_CondAndIter_p3}, respectively. Here, we have also added a line with the condition number of the operator $A_{\hfine}:V_{\hfine} \times V_{\hfine} \rightarrow V_{\hfine}$ defined by $(A_{\hfine} u_{\hfine}, v_{\hfine})_{L^2(\Om)} = \Aa[\hfine][u_{\hfine}][v_{\hfine}]$ for all $u_{\hfine},v_{\hfine} \in V_{\hfine}$, and, in parentheses, the iteration counts of the Conjugate Gradient method for solving~\eqref{eq:dG} without preconditioning. Analogous behaviour of the condition number and iteration counts to those presented in the previous example are observed. In particular, we observe that the condition number is roughly constant on the diagonals and subdiagonals of the two tables, while, along each row, i.e., when $\mcal[T][\hcoarse]$ is fixed, the expected quadratic growth in $K(P_{ad})$ is observed. Similar considerations are also noted for the iteration counts.

\bgroup
\def\arraystretch{1.5}
{\setlength{\tabcolsep}{0.5em} 
\begin{table}[t!]
\centering
\footnotesize
\begin{tabular}{l || llllll}
& \multicolumn{6}{l}{$N_{\hfine}\ \to$} \\
\hhline{~------}
$\downarrow N_{\hcoarse}$  & 384  & 3072  & 24576  & 196608  & 1572864  & 12582912\\
\hhline{-------}
48  & 107 (85)  & 411 (156)  & 1497  (294)  & 6216 (580)  & 25791 (1089) & 94276  (2012) \\
384  & -   & 136 (95)  & 499 (169)  & 1878 (311)  & 7407  (584) & 28762 (1106)  \\
3072  & -   & -   & 146 (96)  & 480  (165)  & 1904  (306)  & 7861  (578) \\
24576  & -   & -   & -   & 144  (94)  & 491  (164) & 1973  (306) \\
196608  & -   & -   & -   & -   & 144 (94) & 496 (164) \\
1572864 & -   &  -  &  -  &  -  &  -  & 145 (94) \\
\hline
\hline
$K(A_{\hfine})$  & 734 (166)  & 2859 (289)  & 11407 (507)  & 45618 (933)  & 182199 (1649)  & 708509 (3012)\\
\end{tabular}

\vspace*{0.3cm}

\caption{Example 3. Condition number (and iteration counts): tetrahedral fine meshes and agglomerated polyhedral coarse grids, with $p=q=1$.}
\label{tab:3DTetrahedra_CondAndIter_p1}
\end{table}
}


\begin{table}[t!]
\centering
\footnotesize
\begin{tabular}{l || llll}
& \multicolumn{4}{l}{$N_{\hfine}\ \to$} \\
\hhline{~----}
$\downarrow N_{\hcoarse}$  & 384  & 3072  & 24576  & 196608  \\
\hhline{-----}
48  & 607 (174) & 2120 (309) & 6760 (515) &  26674 (924) \\
384  & -   & 655 (179) & 2334 (314)  &  7507 (536)  \\
3072  & -   & -   & 693 (182)  &  2295 (316)  \\
24576  & -   & -   & -   & 697 (182)   \\
\hline
\hline
$K(A_{\hfine})$  & 12247 (679)  & 40675 (1056)  & 154934 (1722)  & 602050 (2991)\\
\end{tabular}

\vspace*{0.3cm}

\caption{Example 3. Condition number (and iteration counts): tetrahedral fine meshes and agglomerated polyhedral coarse grids, with $p=q=3$.}

\label{tab:3DTetrahedra_CondAndIter_p3}
\end{table}
\egroup

\subsection{Example 4}\label{sec:TeCa4}

\begin{figure}[t!]
\centering
\begin{tabular}{cccc}
 $\mcal[T][\hfine]$ & $\hcoarse = 2 h$ & $\hcoarse = 4 h$ & $\hcoarse = 8 h$\\
\subfloat{
	  \raisebox{-.45\height}{\includegraphics[width=0.2\textwidth]{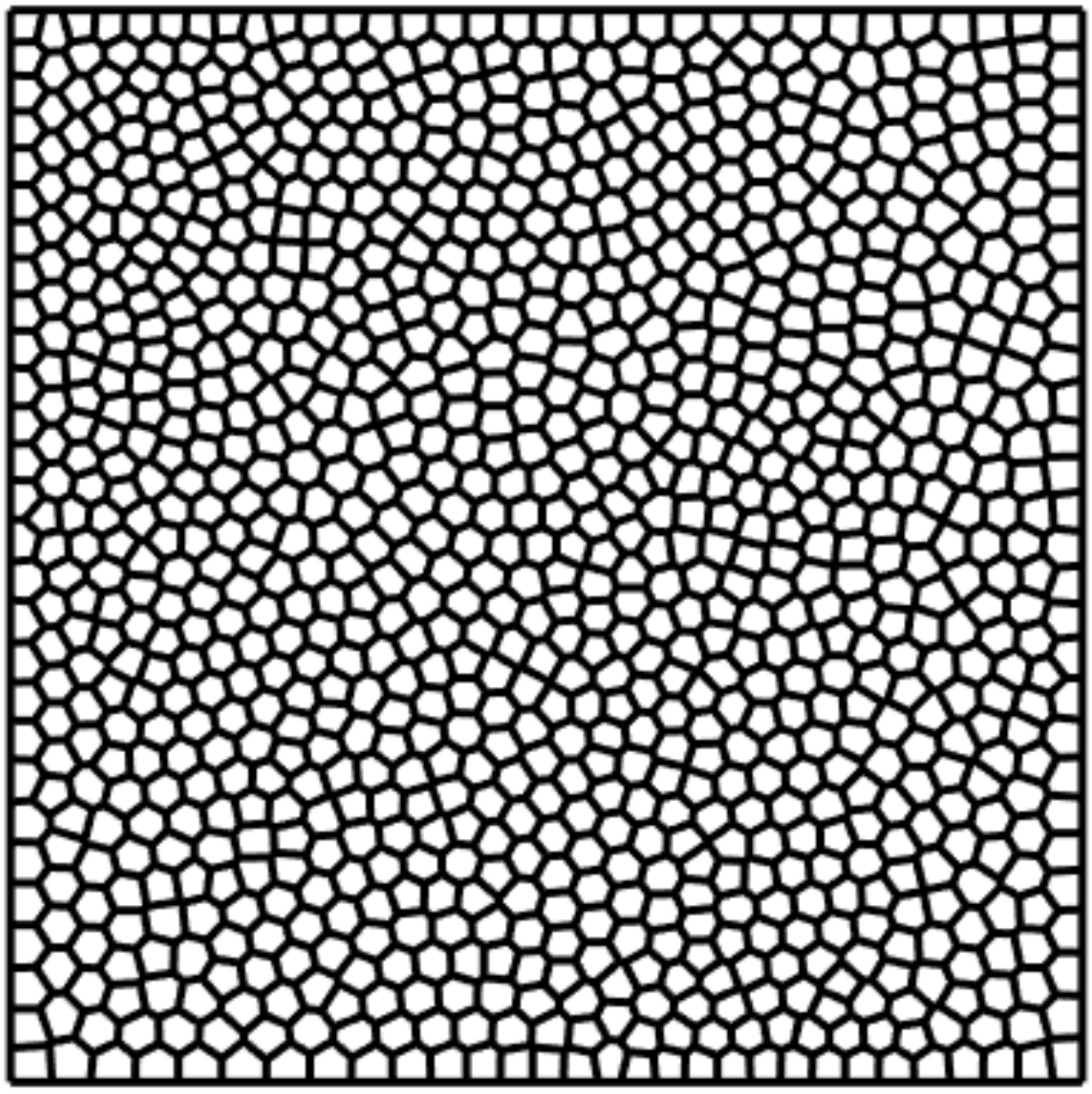}}}  & \subfloat{
          \raisebox{-.45\height}{\includegraphics[width=0.2\textwidth]{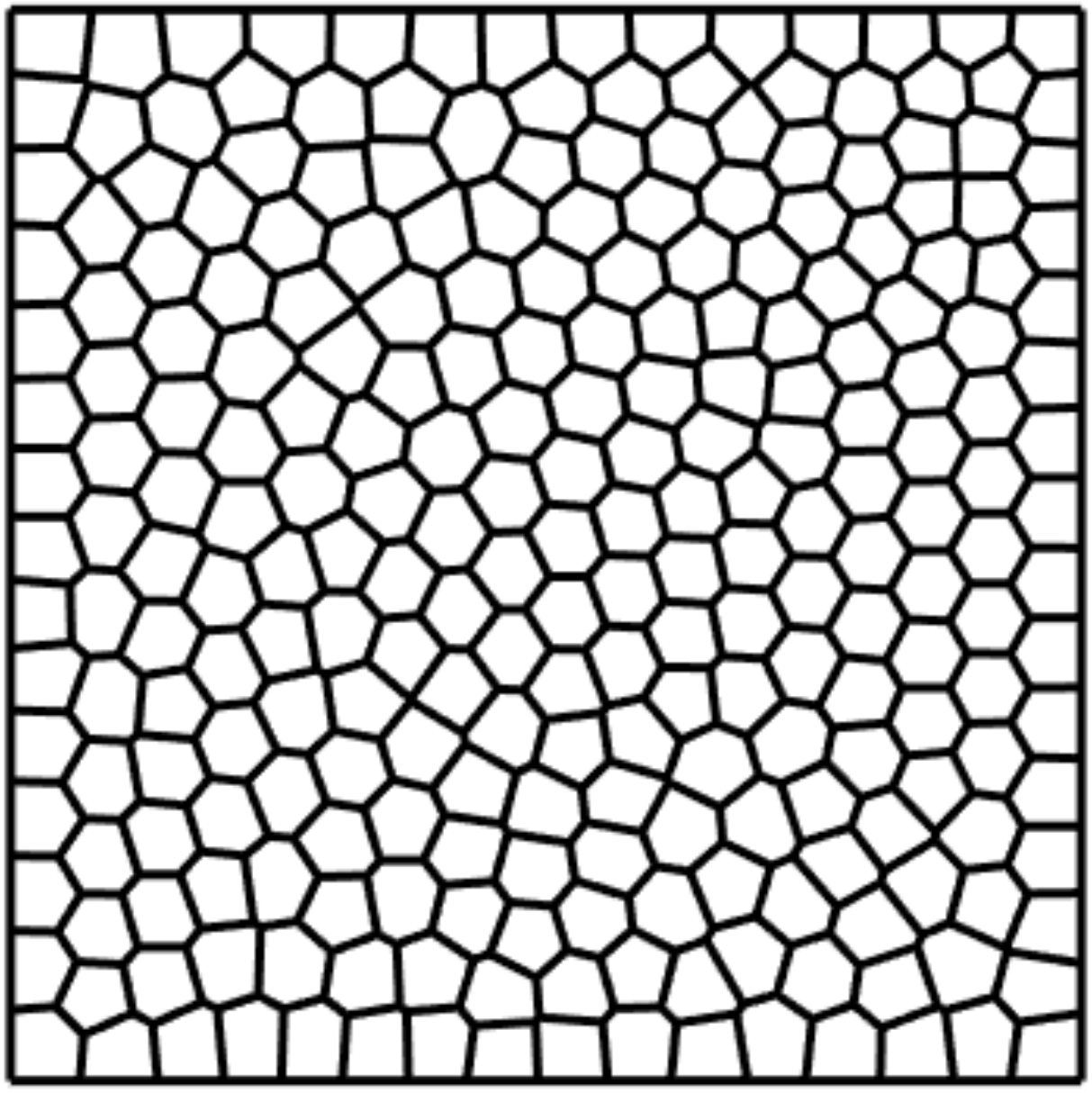}}} & \subfloat{
          \raisebox{-.45\height}{\includegraphics[width=0.2\textwidth]{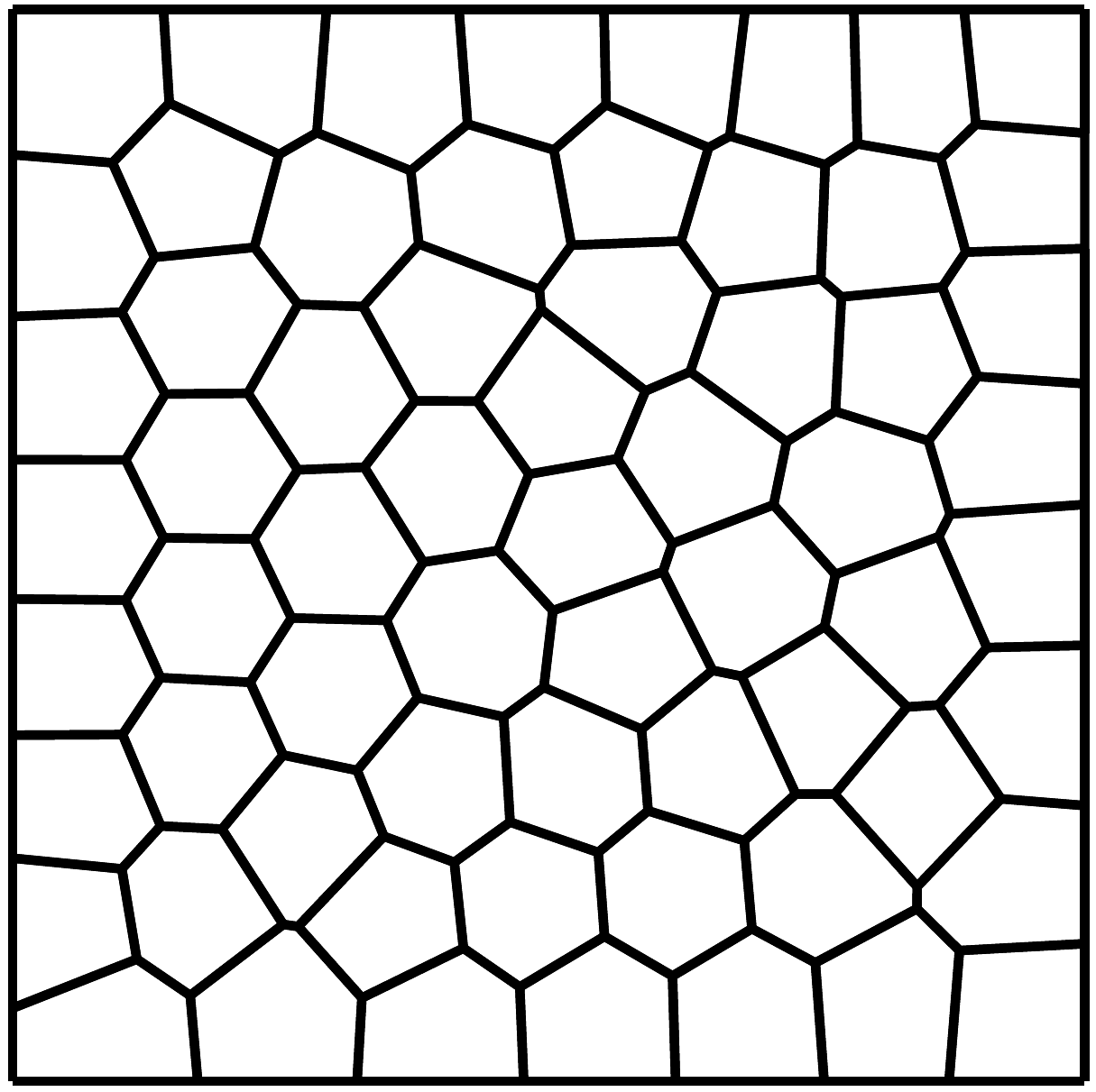}}} & \subfloat{
          \raisebox{-.45\height}{\includegraphics[width=0.2\textwidth]{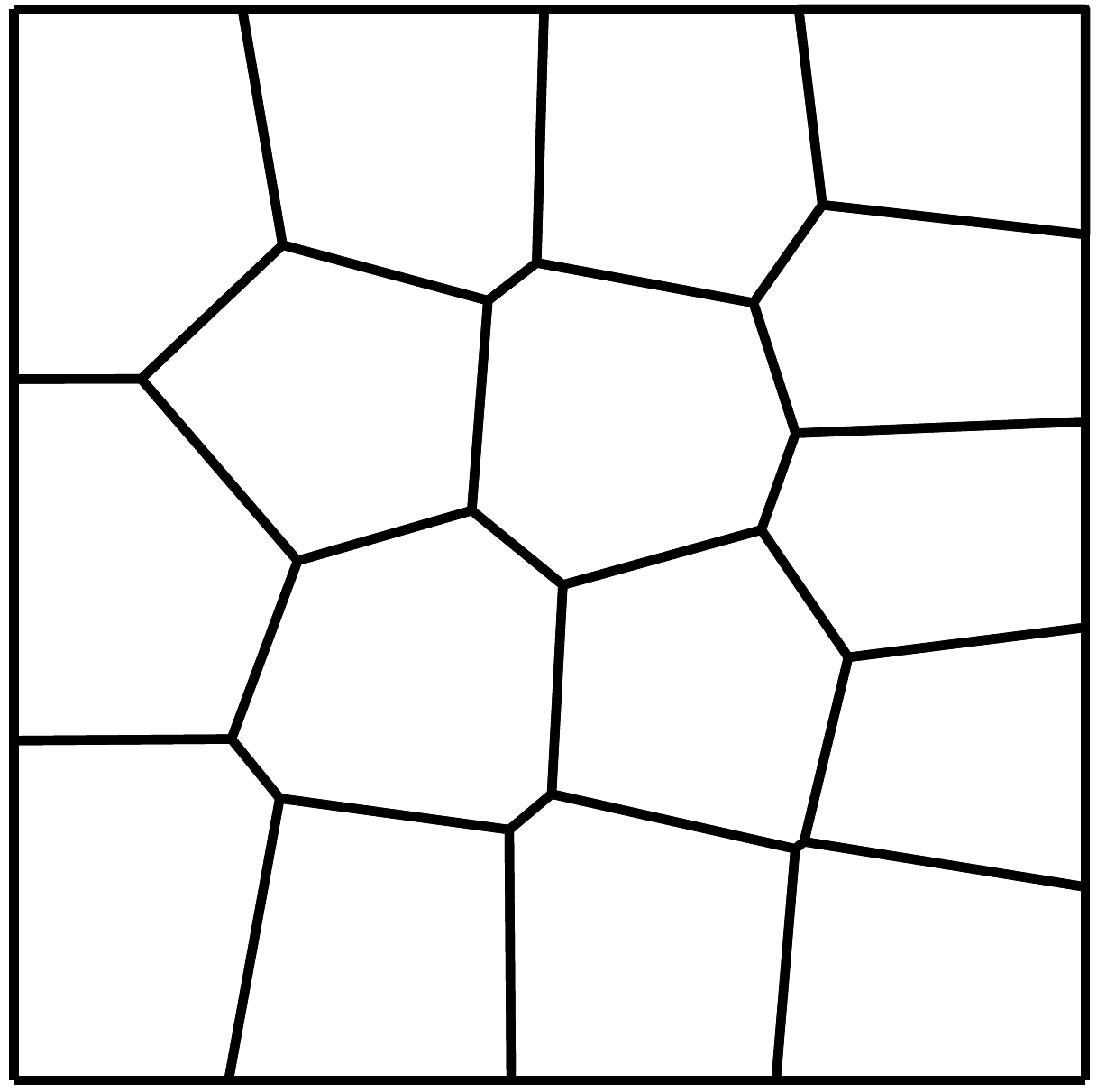}}}\\
\end{tabular}
\caption{Example 4. Sequence of non-nested Voronoi polygonal grids employed.}
\label{fig:PolyGridsNonNested}
\end{figure}

Given the definition of $R_0^\top$, the proposed ASPCG algorithm naturally admits the use of non-nested coarse spaces, i.e., when $V_{\hcoarse} \nsubseteq V_h$. In order to confirm the condition number bound stated in \Cref{rmrk:CondNonNested} when $V_{\hcoarse} \nsubseteq V_h$, we consider a set of independently generated Voronoi polygonal tessellations of $(0,1)^2$ of size $\hfine$ and $\hcoarse>h$, respectively; in this way, $\mcal[T][h]$ and $\mcal[T][\hcoarse]$ are non-nested, cf. \Cref{fig:PolyGridsNonNested}.
The results (for $\rho=1$) shown in~\Cref{tab:ASPCG_Cond_Nest0_hHfine_HHcoarse_P1,tab:ASPCG_Cond_Nest0_hHfine_HHcoarse_P3} for $p=q=1$ and $p=q=3$, respectively, illustrate analogous behavior to the results for the nested case presented in the previous examples; this is in agreement with the condition number bound stated in \Cref{rmrk:CondNonNested}.

\bgroup
\def\arraystretch{1.5}

\begin{table}[t!]
\centering
\footnotesize
\begin{tabular}{l||lllll}
& \multicolumn{5}{l}{$N_h\ \to$} \\
\hhline{~-----}
$\downarrow N_{\hcoarse}$  & 64  & 256  & 1024  & 4096  & 16384  \\
\hhline{------}
16  & 23.29 (38)  & 92.13 (77)  & 387.61 (159)  & 1624.26 (324)  & 6370.86 (657)  \\
64  & -   & 25.91 (39)  & 106.42 (84)  & 411.02 (167)  & 1774.19 (342)  \\
256  & -   & -   & 26.73 (41)  & 100.89 (82)  & 425.97 (169)  \\
1024  & -   & -   & -   & 31.81 (44)  & 118.61 (86)  \\
4096  & -   & -   & -   & -   & 30.56 (43)  \\
\end{tabular}

\vspace*{0.3cm}

\caption{Example 4. Condition number (and iteration counts): non-nested polygonal grids, with $p=q=1$.}
\label{tab:ASPCG_Cond_Nest0_hHfine_HHcoarse_P1}
\end{table}


\begin{table}[t!]
\centering
\footnotesize
\begin{tabular}{l||lllll}
& \multicolumn{5}{l}{$N_h\ \to$} \\
\hhline{~-----}
$\downarrow N_{\hcoarse}$  & 64  & 256  & 1024  & 4096  & 16384  \\
\hhline{------}
16  & 148.36 (83)  & 429.84 (143)  & 1602.53 (275)  & 5405.40 (529)  & 21263.66 (1058)  \\
64  & -   & 142.42 (76)  & 405.44 (135)  & 1525.94 (263)  & 5170.13 (498)  \\
256  & -   & -   & 157.47 (80)  & 452.41 (137)  & 1469.08 (249)  \\
1024  & -   & -   & -   & 147.97 (77)  & 402.98 (124)  \\
4096  & -   & -   & -   & -   & 135.77 (70)  \\
\end{tabular}

\vspace*{0.3cm}

\caption{Example 4. Condition number (and iteration counts): non-nested polygonal grids, with $p=q=3$.}
\label{tab:ASPCG_Cond_Nest0_hHfine_HHcoarse_P3}
\end{table}

\egroup

\subsection{Example 5}\label{sec:TeCa5}
In this final example we investigate the dependence of the condition number on the polynomial degree $p$ in both the nested and non-nested cases with $\rho=1$. For the nested case, we consider a total of four tests: two of them are characterized by quadrilateral fine grids with $N_h = 256$ and $N_h = 1024$ elements, while the two other tests are based on employing the polygonal fine grids depicted in~\Cref{fig:PolyGrids}, where the fine meshes have $N_h = 262$ and $N_h = 516$ polygonal elements. For each test the coarse mesh $\mcal[T][\hcoarse]$ is obtained by agglomeration of $\mcal[T][\hfine]$ in order to guarantee $\hcoarse \eqsim \nicefrac{h}{4}$. Analogous fine meshes are also considered in the non-nested setting; however, here the coarse mesh $\mcal[T][\hcoarse]$ is selected to be a Voronoi grid generated independently of $\mcal[T][h]$, cf. \Cref{fig:NonNestForPdep}. In \Cref{fig:p_dep} we plot the condition number $K(P_{ad})$ on each set of grids as the polynomial degree $p$ is increased. Here, we observe that, in the nested setting, i.e., when $V_{\hcoarse} \subseteq V_h$, for a fixed mesh size $K(P_{ad}) = \mcal[O](p)$ as $p$ increases; however, when  $V_{\hcoarse} \nsubseteq V_h$, then $K(P_{ad}) = \mcal[O](p^2)$ as $p$ increases. This behaviour is in agreement with the condition number bounds stated in \Cref{rmrk:Cond,rmrk:CondNonNested}, respectively.


\begin{figure}[t!]
    \centering
    \subfloat{{\includegraphics[width=4cm]{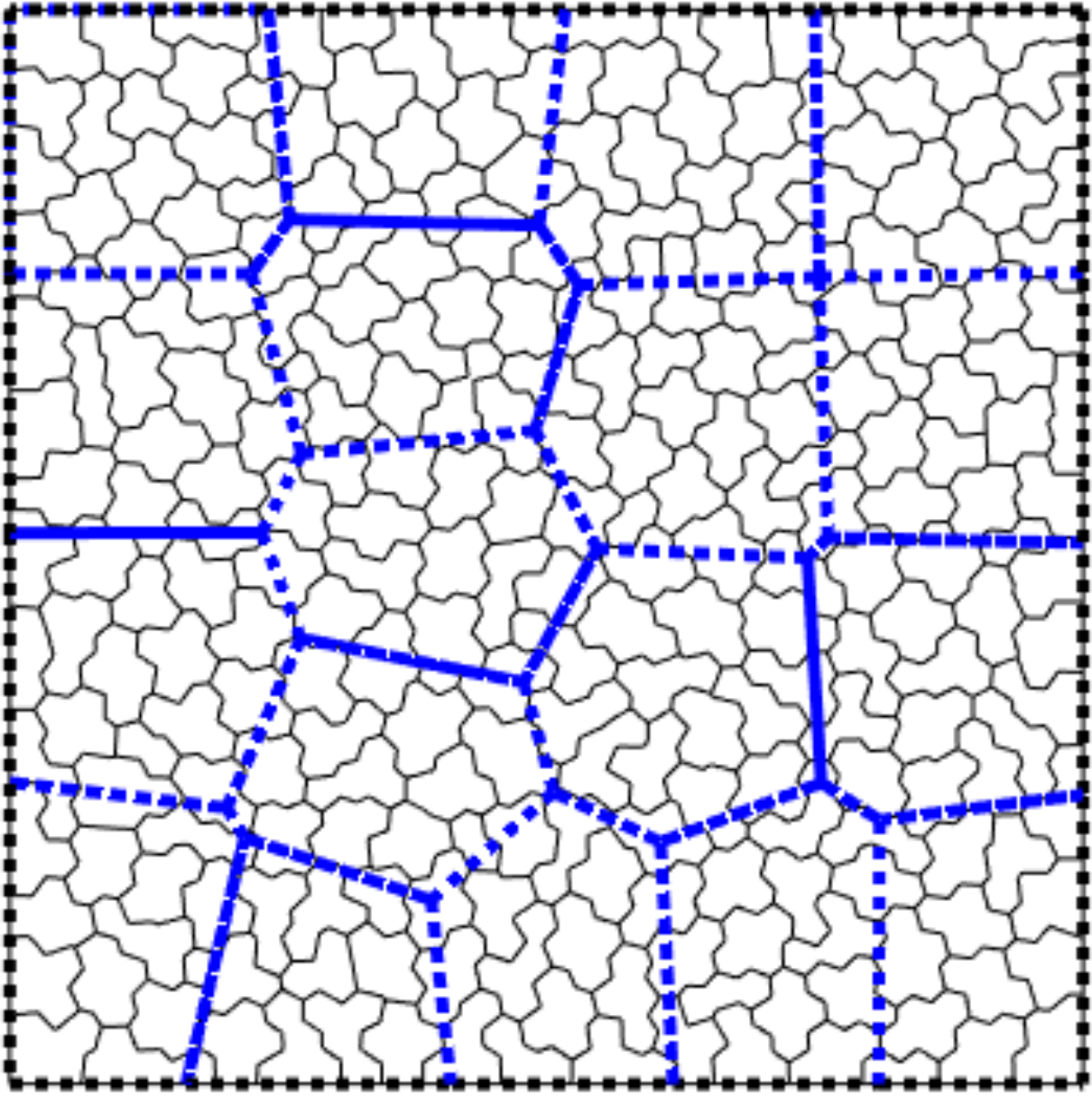} }}
    \
    \subfloat{{\includegraphics[width=4cm]{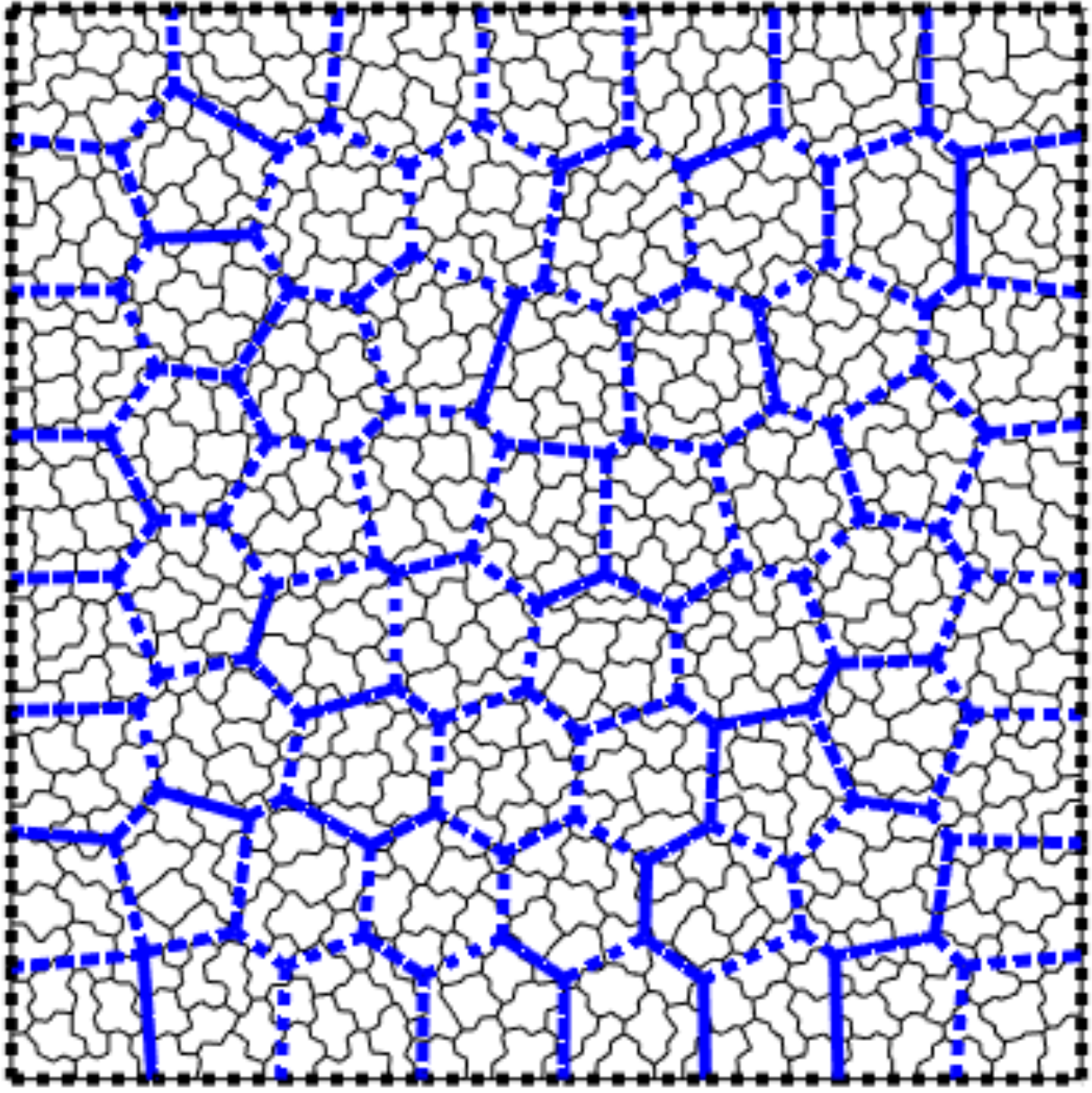} }}
    \caption{Example 5. Pairs of non-nested grids $\mcal[T][h]$ (solid) and $\mcal[T][\hcoarse]$ (dashed), respectively.}
    \label{fig:NonNestForPdep}
\end{figure}


\begin{figure}
\begin{tikzpicture}

\begin{axis}[%
width=6cm,
height=5cm,
scale only axis,
xmode=log,
xmin=0.75,   
xmax=10,
xminorticks=true,
xlabel={$p$},
ymode=log,
ymin=30,
ymax=5000,
yminorticks=true,
ylabel={$K(P_{ad})$},
title={},
legend style={draw=black,fill=white,legend cell align=left,font=\fontsize{6}{5}\selectfont},
legend pos=outer north east
]

\addplot [color=red,solid,line width=2.0pt, mark=square*,mark options={color=red}]
table{1 60.5186
2 209.6212
3 488.3035
4 813.926
5 1317.8976
6 1985.7064
7 2592.3762
8 3238.7525
9 3905.5234
};
\addlegendentry{Non-nested poly. $N_{\hfine}=262$};

\addplot [color=blue,solid,line width=2.0pt, mark=*,mark options={color=blue}]
table{1 34.3035
2 178.7451
3 418.7812
4 682.6494
5 1060.6766
6 1534.0321
7 2101.8662
8 2776.7924
9 3597.7841
};
\addlegendentry{Non-nested poly. $N_{\hfine}=516$};

\addplot [color=orange,solid,line width=2.0pt, mark=triangle*,mark options={color=orange}]
table{1 62.641
2 160.0706
3 241.8657
4 325.4846
5 379.2382
6 456.2464
7 551.1343
8 603.7709
9 661.4906
};
\addlegendentry{Nested poly. $N_{\hfine}=262$};

\addplot [color=black,solid,line width=2.0pt, mark=diamond*,mark options={color=black}]
table{1 64.9171
2 179.6452
3 288.2962
4 385.0758
5 483.7439
6 533.9126
7 625.9311
8 624.4185
9 675.1565
};
\addlegendentry{Nested poly. $N_{\hfine}=516$};

\addplot [color=gray,solid,line width=2.0pt, mark=pentagon*,mark options={color=gray}]
table{1 43.4183
2 115.4801
3 194.2126
4 252.5511
5 336.7532
6 411.7016
7 543.2502
8 496.9999
9 631.2337
};
\addlegendentry{Nested quad. $N_{\hfine}=256$};

\addplot [color=green,solid,line width=1.0pt, mark=asterisk,mark options={color=green}]
table{1 43.6818
2 115.4801
3 194.2126
4 252.5511
5 336.7532
6 411.7016
7 543.2502
8 501.3266
9 633.7908
};
\addlegendentry{Nested quad. $N_{\hfine}=1024$};

\addplot [color=black,solid,line width=2.0pt]
table{1   30
2   60
3   90
4   120
5   150
6   180
7   210
8   240
9   270
};
\addlegendentry{$\mcal[O](p)$};

\addplot [color=black,dashed,line width=2.0pt]
table{1   80
2   320
3   720
4   1280
5   2000
6   2880
7   3920
8   5120
9   6480
};
\addlegendentry{$\mcal[O](p^2)$};

\end{axis}
\end{tikzpicture}%

\caption{Example 5. Condition number $K(P_{ad})$ as function of $p$.}
\label{fig:p_dep}
\end{figure}
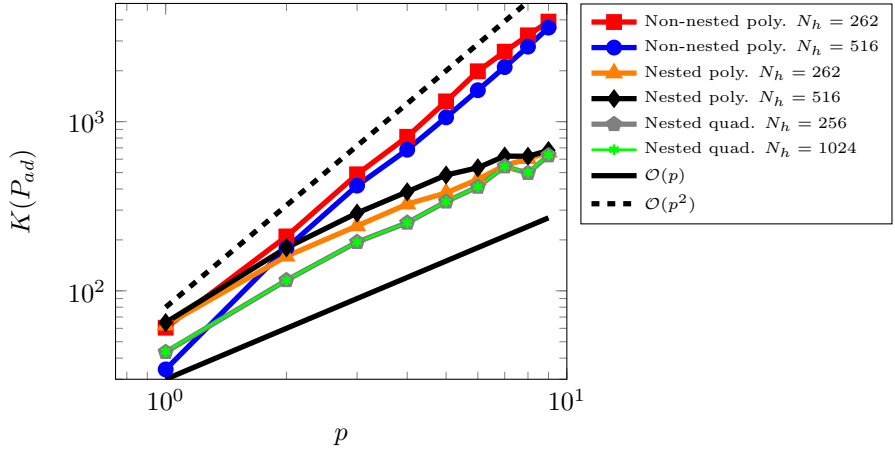

\appendix
\section{Proof of Theorem~\ref{thm_elliptic_regularity}} \label{proof_elliptic_regularity}
In this section we present the proof of Theorem~\ref{thm_elliptic_regularity}. To this end, suppose that $\Omega \subset \mathbb{R}^n$ is a bounded, open, convex domain with boundary $\partial\Omega$. Given $f \in L_0^2(\Omega)$, consider the homogeneous Neumann problem
\begin{equation}\label{eq:ProblemZ2}
- \Delta u  = f \quad\ \text{in } \Omega, \qquad
\nabla u \cdot \mathbf{n} = 0 \quad\ \text{on } \partial \Omega.
\end{equation}
The weak formulation of \eqref{eq:ProblemZ2} is: find $u \in H^1(\Omega)/\mathbb{R} := H^1(\Omega)\cap L^2_0(\Omega)$, such that
\[
a(u,v):=\int_\Omega \nabla u \cdot \nabla v \,\diff \xvec = \ell(v) := \int_{\Omega} f v \,\diff \xvec \qquad \forall\, v \in H^1(\Omega)/\mathbb{R},
\]
with $H^1(\Omega)/\mathbb{R}$ equipped with the norm $\|v\|_{H^1(\Omega)/\mathbb{R}} := \|\nabla v\|_{L^2(\Omega)}$. From \cite[Theorem 3.2.1.3]{Gr1985} it follows that $u\in H^2(\Omega)$; the proof of \eqref{eqn_elliptic_reg} now proceeds with the following steps.

\textbf{Step 1.} [$H^1(\Omega)$ bound] The existence of a unique weak solution to \eqref{eq:ProblemZ2} follows by the Lax--Milgram lemma applied to the bilinear form $a(\cdot,\cdot)$, which is bounded and coercive on $H^1(\Omega)/\mathbb{R}$, and noting that the linear functional $\ell(\cdot)$ is bounded on $H^1(\Omega)/\mathbb{R}$. Indeed,
\begin{equation}
\ell(v) = \int_\Omega f v \, \diff \xvec = \int_\Omega f \big[v - v_\Omega \big]\, \diff \xvec \leq \|f\|_{L^2(\Omega)} \| v - v_\Omega\|_{L^2(\Omega)},
\end{equation}
where $v_\Omega := \nicefrac{1}{|\Omega|} \int_\Omega v \,\diff \xvec$.
By Poincar\'{e}'s inequality,
\[ \| v - v_\Omega\|_{L^2(\Omega)} \leq C(\Omega) \|\nabla v\|_{L^2(\Omega)},\]
where $C(\Omega)$ is a positive constant. Recalling that $\Omega$ is a bounded, open, convex domain, it can be shown that $C(\Omega) \leq \nicefrac{1}{\pi} \, \mbox{diam}(\Omega)$, cf.~\cite{Payne_Weinberger}. Setting $v=u$ in the weak formulation above, we then deduce that
\[ \|\nabla u\|^2_{L^2(\Omega)} = \ell(u) \leq C(\Omega) \|f\|_{L^2(\Omega)} \|\nabla u\|_{L^2(\Omega)}\leq \frac{1}{\pi} \mbox{diam}(\Omega)\|f\|_{L^2(\Omega)}\|\nabla u\|_{L^2(\Omega)}.\]
Hence,
\[ \|\nabla u\|_{L^2(\Omega)} \leq \frac{1}{\pi} \mbox{diam}(\Omega)\|f\|_{L^2(\Omega)}.\]

\smallskip

\textbf{Step 2.} Suppose that $\Omega \subset \Omega_m$, where $\Omega_m$, $m=1,2,\dots$, is a sequence of bounded, open, convex $C^2$ domains, such that $\mbox{\rm dist}(\Omega,\Omega_m) \leq \nicefrac{1}{m}$.
The existence of such a sequence $\{\Omega_m\}_{m=1}^\infty$ follows from Eggleston's lemma, cf. \cite[Lemma 3.2.1.1]{Gr1985}. Hence, in particular $0 \leq \mbox{\rm diam}(\Omega_m) - \mbox{\rm diam}(\Omega) \rightarrow 0$ as $m \rightarrow \infty$.

\smallskip

\textbf{Step 3.} With $\Omega_m$, $m=1,2,\dots$, as in Step 2, we extend $f$ by $0$ from $\Omega$ to $\Omega_m$ for each $m=1,2,\dots$, and define
\[ f_m(\xvec):= \left\{ \begin{array}{cl} f(\xvec) & \mbox{for $\xvec \in \Omega$},\\ 0 & \mbox{for $\xvec \in \Omega_m \setminus \Omega$. }\end{array} \right. \]
Clearly, $f_m \in L^2_0(\Omega_m)$. We consider the following Neumann problem on $\Omega_m$:
\begin{equation}
- \Delta u_m  = f_m \quad\ \text{in } \Omega_m, \qquad
\nabla u_m \cdot \mathbf{n}_m = 0 \quad\ \text{on } \partial \Omega_m,
\end{equation}
where $\mathbf{n}_m$ is the unit outward normal vector to $\partial\Omega_m$. We have that the unique weak solution $u_m \in H^1(\Omega_m)/\mathbb{R}$ of the above Neumann problem satisfies the following elliptic regularity result: $u_m \in H^2(\Omega_m)$, cf. \cite[Theorem 3.2.1.3]{Gr1985}, and
\[\|u_m\|_{H^2(\Omega_m)} \leq \sqrt{6} \,\big(\| f_m \|_{L^2(\Omega_m)} + \| u_m \|_{L^2(\Omega_m)}\big),\qquad m=1,2,\dots,\]
cf. \cite[Theorem 3.1.2.3]{Gr1985}, with $\lambda=1$. Thereby, we deduce that
\begin{equation}\label{eq1}
\|u_m\|_{H^2(\Omega_m)} \leq \sqrt{6} \,\big(\| f\|_{L^2(\Omega)} + \| u_m \|_{L^2(\Omega_m)}\big),\qquad m=1,2,\dots.
\end{equation}
Since $\Omega \subset \Omega_m$, upon application of Poincar\'e's inequality on the right-hand side of \eqref{eq1},
followed by recalling the $H^1(\Omega)$ bound derived in Step 1, we get
\begin{eqnarray}
\|u_m\|_{H^2(\Omega)}
\leq \|u_m\|_{H^2(\Omega_m)}
&\leq& \sqrt{6} \,\big(\| f\|_{L^2(\Omega)} + \frac{1}{\pi}\mbox{\rm diam}(\Omega_m) \| \nabla u_m \|_{L^2(\Omega_m)}\big)\nonumber\\
&\leq& \sqrt{6} \,\big(\| f\|_{L^2(\Omega)} + \frac{1}{\pi^2}[\mbox{\rm diam}(\Omega_m)]^2 \|f_m\|_{L^2(\Omega_m)}\big)\nonumber\\
&=& \sqrt{6} \,\bigg(1 + \frac{1}{\pi^2}[\mbox{\rm diam}(\Omega_m)]^2\bigg) \|f\|_{L^2(\Omega)},
\label{eq2}
\end{eqnarray}
$m=1,2,\dots$. Thanks to Step 2, $\lim_{m \rightarrow \infty} \mbox{\rm diam}(\Omega_m) = \mbox{\rm diam}(\Omega)$. As $(\mbox{\rm diam}(\Omega_m))_{m=1}^\infty$ is a convergent sequence in $\mathbb{R}$ it is automatically a bounded sequence, and therefore, because of \eqref{eq2}, there exists a positive constant $C_0$, independent of $m$, such that
\begin{equation}\label{eq3}
\|u_m\|_{H^2(\Omega)} \leq \|u_m\|_{H^2(\Omega_m)} \leq C_0\qquad \mbox{for all $m=1,2,\dots$}.
\end{equation}
Thus, $(u_m)_{m=1}^\infty$ is a bounded sequence in $H^2(\Omega)$. Hence, there exists an element $u_\infty \in H^2(\Omega)$
and a weakly convergent subsequence $u_{m_k} \rightharpoonup u_\infty$ in $H^2(\Omega)$.
By weak lower semicontinuity of the norm function $\|\cdot\|_{H^2(\Omega)}$, we have that
\begin{equation}\label{eq4}
\|u_\infty\|_{H^2(\Omega)} \leq \liminf_{k \rightarrow \infty} \|u_{m_k}\|_{H^2(\Omega)}.
\end{equation}
Further, thanks to the compact Sobolev embedding $H^2(\Omega) \Subset H^1(\Omega)$ guaranteed by the Rellich--Kon\-drashov theorem, by extracting a further subsequence (not indicated), we have that $u_{m_k} \rightarrow u_\infty$ strongly
in $H^1(\Omega)$. Now,
\begin{align}\label{eq5}
\| \nabla u_{m_k}\|_{L^2(\Omega_{m_k})}^2 = \|\nabla u_{m_k}\|^2_{L^2(\Omega)} + \|\nabla u_{m_k}\|^2_{L^2(\Omega_{m_k} \setminus \Omega)}.
\end{align}
Focusing on the second term on the right-hand side of \eqref{eq5}, by H\"older's inequality with conjugate exponents $\alpha=p/2$ and $\alpha' = \frac{\alpha}{\alpha-1} = p/(p-2)$, $2<p< 2n/(n-2)$ (where $2n/(n-2)$ is the critical Sobolev index), and \eqref{eq3}, we have that
\begin{eqnarray}
\|\nabla u_{m_k}\|_{L^2(\Omega_{m_k} \setminus \Omega)}
&\leq& \|\nabla u_{m_k}\|_{L^p(\Omega_{m_k} \setminus \Omega)} \,|\Omega_{m_k} \setminus \Omega|^{\frac{p-2}{2p}}\nonumber\\
& \leq& \|\nabla u_{m_k}\|_{L^p(\Omega_{m_k})} \,|\Omega_{m_k} \setminus \Omega|^{\frac{p-2}{2p}}\nonumber\\
& \leq& C\big(|\Omega_{m_k}|,\mbox{\rm diam}(\Omega_{m_k}),n,p\big)\, \|u_{m_k}\|_{H^2(\Omega_{m_k})}\, |\Omega_{m_k} \setminus \Omega|^{\frac{p-2}{2p}}\nonumber\\
& \leq& C\big(n,p)\, \|u_{m_k}\|_{H^2(\Omega_{m_k})}\, |\Omega_{m_k} \setminus \Omega|^{\frac{p-2}{2p}}\nonumber\\
& \leq& C\big(n,p)\, C_0 \, |\Omega_{m_k} \setminus \Omega|^{\frac{p-2}{2p}} \rightarrow 0 \qquad \mbox{as $k \rightarrow \infty$}. \label{eq6}
\end{eqnarray}
Passing to the limit $k \rightarrow \infty$ in \eqref{eq5} we therefore have that
\begin{equation}
\lim_{k \rightarrow \infty}\| \nabla u_{m_k}\|_{L^2(\Omega_{m_k})}^2 = \lim_{k \rightarrow \infty}\|\nabla u_{m_k}\|^2_{L^2(\Omega)} = \|\nabla u_{\infty}\|^2_{L^2(\Omega)},
\end{equation}
where the last equality follows from the strong convergence $u_{m_k} \rightarrow u_\infty$ in $H^1(\Omega)$.

Recalling from \eqref{eq2} that
\begin{align}\label{eq7}
\|u_{m_k}\|_{H^2(\Omega)} \leq \sqrt{6} \,\big(\| f\|_{L^2(\Omega)} + \frac{1}{\pi}\mbox{\rm diam}(\Omega_{m_k}) \| \nabla u_{m_k} \|_{L^2(\Omega_{m_k})}\big),
\end{align}
passage to the limit $k \rightarrow \infty$ in inequality \eqref{eq7} using \eqref{eq4}, \eqref{eq6}, together with $\lim_{m\rightarrow \infty}\mbox{\rm diam}(\Omega_m) = \mbox{\rm diam}(\Omega)$, yields that
\begin{align}\label{eq8}
\|u_\infty\|_{H^2(\Omega)} \leq \sqrt{6} \,\big(\| f\|_{L^2(\Omega)} + \frac{1}{\pi}\mbox{\rm diam}(\Omega) \| \nabla u_\infty \|_{L^2(\Omega)}\big).
\end{align}

\smallskip

\textbf{Step 4.} [Identification of $u_\infty$] It remains to show that $u_\infty = u$, the weak solution of the original Neumann problem on $\Omega$. To this end, we consider the weak formulation of the Neumann problem satisfied by $u_m$:
\[ \int_{\Omega_m} \nabla u_m \cdot \nabla v \, \diff \xvec = \int_{\Omega_m} f_m v \, \diff \xvec \qquad \forall\, v \in H^1(\Omega_m)/\mathbb{R}.\]
Thanks to the definition of $f_m$, this weak formulation is equivalent to
\[ \int_{\Omega_m} \nabla u_m \cdot \nabla v \, \diff \xvec = \int_{\Omega} f v \, \diff \xvec \qquad \forall\, v \in H^1(\Omega_m)/\mathbb{R},\]
and therefore
\[ \int_{\Omega}\nabla u_m \cdot \nabla v \, \diff \xvec
+ \int_{\Omega_m \setminus \Omega} \nabla u_m \cdot \nabla v \, \diff \xvec =
\int_{\Omega} f v \, \diff \xvec \qquad \forall\, v \in H^1(\Omega_m)/\mathbb{R}, \]
whereby
\[ \int_{\Omega} \nabla u_{m_k} \cdot \nabla v \, \diff \xvec
+ \int_{\Omega_{m_k} \setminus \Omega} \nabla u_{m_k} \cdot \nabla v \, \diff \xvec = \int_\Omega f v \, \diff \xvec \qquad \forall\, v \in H^1(\Omega_{m_k})/\mathbb{R}. \]
Equivalently, because $\int_\Omega f(\xvec) \, \diff \xvec = 0$,
\[ \int_\Omega \nabla u_{m_k} \cdot \nabla v \, \diff \xvec
+ \int_{\Omega_{m_k} \setminus \Omega} \nabla u_{m_k} \cdot \nabla v \, \diff \xvec = \int_\Omega f v \, \diff \xvec\qquad \forall\, v \in H^1(\Omega_{m_k}). \]
Consider a fixed bounded domain $\Omega_0 \subset \mathbb{R}^n$ such that $\Omega_0 \Supset \Omega_m \supset \Omega$ for all
$m=1,2,\dots$. Then,
\[ \int_\Omega \nabla u_{m_k} \cdot \nabla v \, \diff \xvec
+ \int_{\Omega_{m_k} \setminus \Omega} \nabla u_{m_k} \cdot \nabla v \, \diff \xvec = \int_\Omega f v \, \diff \xvec \qquad \forall\, v \in H^1(\Omega_0). \]
By noting \eqref{eq6}, the strong convergence $u_{m_k} \rightarrow u_\infty$ in $H^1(\Omega)$ as $k \rightarrow \infty$, we have that
\[ \int_\Omega \nabla u_{\infty} \cdot \nabla v \, \diff \xvec = \int_\Omega f v \, \diff \xvec \qquad \forall\, v \in H^1(\Omega_0), \]
hence also
\[ \int_\Omega \nabla u_{\infty} \cdot \nabla v \, \diff \xvec = \int_\Omega f v \, \diff \xvec \qquad \forall\, v \in H^1(\Omega), \]
since any element of $v \in H^1(\Omega_0)$ can be viewed as the extension of a $v \in H^1(\Omega)$ to the superset $\Omega_0$.
Therefore, again since $\int_\Omega f(x)\, \mathrm{d}x = 0$, also
\[ \int_\Omega \nabla u_{\infty} \cdot \nabla v \, \diff \xvec = \int_\Omega f v \, \diff \xvec \qquad \forall\, v \in H^1(\Omega)/\mathbb{R}. \]
Thus we have shown that $u_\infty$ coincides with the unique weak solution $u$ of the homogeneous Neumann problem posed on $\Omega$.
Returning with this information to \eqref{eq8}, we have that the weak solution of the homogeneous Neumann problem on the
bounded, open, convex (and therefore Lipschitz) domain $\Omega$ satisfies
\begin{align}\label{eq9}
\|u\|_{H^2(\Omega)} &\leq \sqrt{6} \,\big(\| f\|_{L^2(\Omega)} + \frac{1}{\pi}\mbox{\rm diam}(\Omega) \| \nabla u \|_{L^2(\Omega)}\big),
\end{align}
as required.

\section{Condition number estimates for non-nested grids}\label{app:NonNestProof}
In this Appendix we provide a bound on the condition number of the additive Schwarz operator $P_{ad}$ introduced in~\Cref{sec:ASPCG} when the fine and coarse grids $\mcal[T][\hfine]$ and $\mcal[T][\hcoarse]$, respectively, are non-nested. For the sake of simplicity, here we assume that $\rho = 1$ on $\Om$ and consider the massively parallel case, i.e., when $\mcal[T][\hlocal] = \mcal[T][\hfine]$. For the purposes of the proceeding analysis, we also assume that $\Omega$ is convex; moreover, we make the following additional assumption on $\mcal[T][\hfine]$.
\begin{assumption} \label{ass3}(Coverability) For every polytopic element $\elem \in \mcal[T][h]$, there exists a set of $m_{\elem}$ overlapping shape-regular simplices $\mcal[K]_i,\ i=1,\dots,m_{\elem}$, such that
\begin{equation}
\mbox{\rm dist}(\elem,\partial \mcal[K]_i) \lesssim \nicefrac{{\rm diam}(\mcal[K]_i)}{p^2},\quad \text{and}\quad |\mcal[K]_i| \gtrsim | \elem |,
\end{equation}
for all $i=1,\dots,m_{\elem}$, see~\cite[Chapter 3]{CaDoGeHo2017}.
\end{assumption}
Given that Assumption~\ref{ass3} holds, we state the following inverse inequality, cf. \cite{AnHoHuSaVe2017}.
\begin{lemma}
\label{lem:inverse}
Suppose that $v_{\hfine} \in V_{\hfine}$; then, the following bound holds:
\begin{equation}
\normL[\nabla v_h][2][\elem][2]\lesssim  p^4 h_\elem^{-2}\normL[v_h][2][\elem][2] \qquad \forall\, \elem \in \mcal[T][h].
\end{equation}
\end{lemma}
\begin{proof}
We refer to~\cite{CaDoGeHo2017} for the proof of this result.
\end{proof}

We first provide a counterpart of~\Cref{lem:vh_R0TvH}, which holds in the non-nested case and allows us to prove the validity of~\Cref{ass:StaDec} also for non-nested spaces $V_{\hfine}$ and $V_{\hcoarse}$. The key aspect of our analysis is the construction of the conforming approximant introduced in~\Cref{th:Hbound}. In particular, we recall the following result.

\begin{theorem}\label{thm:TildeVhGlobalBounds}
Let $\mathcal{G}_h(v_h) := \nabla_h v_h + \mcal[R][1](\jump{v_h})$ be the discrete gradient operator of $v_h \in V_h$ defined as in~\Cref{def:H_of_v_h} and let $\widetilde{v}_{h} \in H^1_0(\Om)$ be such that
\begin{equation}
\int_{\Om} \nabla \widetilde{v}_{h} \cdot \nabla w\ \diff \xvec = \int_{\Om} \mathcal{G}_h(v_h) \cdot \nabla w\ \diff \xvec \quad \forall\, w \in H^1_0(\Om).
\end{equation}
Then, the following approximation and stability results hold:
\begin{align}
 \normL[v_h - \widetilde{v}_{h}][2][\Om] &\lesssim \frac{h}{p} \normL[\sigma_{h,1}^{\nicefrac{1}{2}} \jump{v_h}][2][\mcal[F][h]], \qquad | \widetilde{v}_{h} |_{H^1(\Om)} \lesssim \normDG[v_h][h,1]. \label{eq:HboundH1_2}
\end{align}
\end{theorem}

\begin{remark}
\Cref{thm:TildeVhGlobalBounds} provides global bounds for $v_{\hfine} \in V_{\hfine}$ in the $L^2$-norm. This result is a particular case of~\Cref{th:Hbound}, where local bounds on each coarse element $\mcal[D][j] \in \mcal[T][\hcoarse]$ are provided. We refer to~\cite{AnHoSm2016} for the proof of~\Cref{thm:TildeVhGlobalBounds}.
\end{remark}

On the basis of the previous result,~\Cref{lem:vh_R0TvH} can be generalized to non-nested spaces as follows.
\begin{lemma}\label{lem:vh_R0TvH_2}
For any $v_h \in V_h$ there exists a coarse function $v_{\hcoarse} \in V_{\hcoarse}$ such that
\begin{align}
& \| v_h - R_0^\top v_{\hcoarse} \|_{L^2(\Om)} \lesssim \frac{\hcoarse}{q} \normDG[v_h][h,1], \label{eq:L2_2}\\
& \normDG[v_h - R_0^\top v_{\hcoarse}][h,1] \lesssim \frac{p^2}{q} \frac{\hcoarse}{\hfine} \normDG[v_h][h,1]. \label{eq:H1_2}
\end{align}
\end{lemma}

\begin{proof}
Let $v_{\hfine} \in V_{\hfine}$ and let $v_{\hcoarse} \in V_{\hcoarse}$ be defined as $v_{\hcoarse} = \Pi_{\hcoarse} \widetilde{v}_{h},$ with $\widetilde{v}_{h}$ as defined in~\Cref{thm:TildeVhGlobalBounds} and where $\Pi_{\hcoarse}$ is the $hp$-approximant introduced in~\Cref{lem:interpDG_loc}. Then, by employing the triangle inequality we have
\begin{equation}
\| v_h - R_0^\top v_{\hcoarse}\|_{L^2(\Om)}  \lesssim \| v_h - \widetilde{v}_{h}\|_{L^2(\Om)} \!+ \| \widetilde{v}_{h} - \Pi^0\widetilde{v}_{h}\|_{L^2(\Om)}
 \!\!+ \| \Pi^0\widetilde{v}_{h} - R_0^\top( \Pi_{\hcoarse}\widetilde{v}_{h})\|_{L^2(\Om)},
\end{equation}
where $\Pi^0:L^2(\Om) \rightarrow V_{\hfine}$ is the $L^2$-projection operator onto $V_{\hfine}$.
From the definition of $R_0^\top$, we note that
$\Pi^0 v_{\hcoarse} = R_0^\top v_{\hcoarse}$ for all $v_{\hcoarse} \in V_{\hcoarse}$. Hence, exploiting~\Cref{lem:interpDG_loc} together with Assumption~\ref{ass2}, cf.~\Cref{rmrk:GloApp}, gives
\begin{align}
\| v_h -  R_0^\top v_{\hcoarse}\|_{L^2(\Om)} & \lesssim \| v_h - \widetilde{v}_{h}\|_{L^2(\Om)} + \| \widetilde{v}_{h} - \Pi^0\widetilde{v}_{h}\|_{L^2(\Om)}  + \| \Pi^0( \widetilde{v}_{h} - \Pi_{\hcoarse}\widetilde{v}_{h} )\|_{L^2(\Om)} \\
& \le \| v_h - \widetilde{v}_{h}\|_{L^2(\Om)} + \| \widetilde{v}_{h} - \Pi_{\hfine}\widetilde{v}_{h}\|_{L^2(\Om)}  + \| \widetilde{v}_{h} - \Pi_{\hcoarse}\widetilde{v}_{h}\|_{L^2(\Om)} \\
& \lesssim \| v_h - \widetilde{v}_{h}\|_{L^2(\Om)}+ \nicefrac{h}{p} \, \| \widetilde{v}_{h}\|_{H^1(\Om)} + \nicefrac{\hcoarse}{q} \, \| \widetilde{v}_{h}\|_{H^1(\Om)};
\end{align}
here we have used that $\| \Pi^0v \|_{L^2(\Om)} \le \| v \|_{L^2(\Om)}$ for all $v \in L^2(\Om)$ and $\| v - \Pi^0 v \|_{L^2(\Om)} \le \| v - w \|_{L^2(\Om)}$ for all $w \in L^2(\Om)$. By applying Poincar\'e's inequality to $\widetilde{v}_{h} \in H^1_0(\Om)$ and noting~\Cref{thm:TildeVhGlobalBounds}, inequality~\eqref{eq:L2_2} immediately follows by observing that $\hfine \le \hcoarse$ and $q \le p$. In order to obtain~\eqref{eq:H1_2} we proceed as follows:
\begin{align}\label{eq:VhMinusR0TVHDG}
 \normDG[v_h - R_0^\top v_{\hcoarse}][h,1]^2 = \| \nabla_h (v_h - R_0^\top v_{\hcoarse}) \|_{L^2(\mcal[T][h])}^2 + \| \sigma_{h,1}^{\nicefrac{1}{2}} \jump{ v_h - R_0^\top v_{\hcoarse} }\|_{L^2(\mcal[F][\hfine])}^2.
\end{align}
We bound the first term on the right-hand side of~\eqref{eq:VhMinusR0TVHDG} by means of \Cref{lem:inverse} and~\eqref{eq:L2_2} as follows:
\begin{align}
\| \nabla_h (v_h - R_0^\top v_{\hcoarse}) \|_{L^2(\mcal[T][h])}^2  & = \sum_{\elem \in \mcal[T][\hfine]} \| \nabla_h (v_h - R_0^\top v_{\hcoarse}) \|_{L^2(\elem)}^2 \\
&  \lesssim \frac{p^4}{\hfine^2} \| v_h - R_0^\top v_{\hcoarse} \|_{L^2(\Om)}^2 \lesssim \frac{p^4}{q^2} \frac{\hcoarse^2}{\hfine^2} \normDG[v_h][h,1]^2. \label{eq:BoundForTheFirstTerm}
\end{align}
The second term on the right-hand side of~\eqref{eq:VhMinusR0TVHDG} can be bounded by recalling the definition of $\sigma_{h,1}$,~\Cref{lem:inversepoly} and~\eqref{eq:L2_2} as follows:
\begin{align}\label{eq:BoundForTheSecondTerm}
\begin{aligned}
\| \sigma_{h,1}^{\nicefrac{1}{2}} \jump{ v_h - R_0^\top v_{\hcoarse} }\|_{L^2(\mcal[F][\hfine])}^2 & \lesssim \frac{p^2}{h} \sum_{\elem \in \mcal[T][\hfine]} \| v_h - R_0^\top v_{\hcoarse} \|_{L^2(\partial \elem)}^2 \\
& \lesssim  \frac{p^2}{h}  \frac{p^2}{h} \| v_h - R_0^\top v_{\hcoarse} \|_{L^2(\Om)}^2  \lesssim \frac{p^4}{q^2} \frac{\hcoarse^2}{\hfine^2} \normDG[v_h][h,1]^2.
\end{aligned}
\end{align}
Inserting~\eqref{eq:BoundForTheFirstTerm} and~\eqref{eq:BoundForTheSecondTerm} into~\eqref{eq:VhMinusR0TVHDG} we obtain~\eqref{eq:H1_2}.
\end{proof}

With~\Cref{lem:vh_R0TvH_2} in hand, we can prove the following result.
\begin{theorem}\label{thm:ValidityStableDecompositionNonNested}
~\Cref{ass:StaDec} holds with
\begin{equation}
C_{\sharp}^2 \eqsim \Bigl( \frac{p^4}{q^2} \frac{\hcoarse^2}{\hfine^2} \Bigr).
\end{equation}
\end{theorem}

\begin{proof}
Let $v_{\hfine} \in V_{\hfine}$. Proceeding as in the proof of~\Cref{thm:TheoremValidityAssumption3}, by selecting $v_0 = v_{\hcoarse}$ as in~\Cref{lem:vh_R0TvH_2}, $v_{\hfine}$ can be decomposed as $v_h = \sum_{i=0}^{N_{\hfine}} R_i^\top v_i$, with $v_i = R_i (v_h - R_0^{\top} v_0) \in V_i$, $i=1,\dots,N_{\hfine}$, so that
\begin{align}\label{eq:Sum_Ai_vi_3}
\Bigl| \sum_{i=0}^{N_\hfine} \Aa[i][v_i][v_i] \Bigr| & \lesssim \normDG[v_h - R_0^\top v_0][h,1]^2 + \Aa[h][v_h][v_h],
\end{align}
where we have used~\eqref{eq:Sum_Ai_vi},~\eqref{eq:AplusB} and~\eqref{eq:BoundOfTermC} with the hypothesis $\mcal[T][\hlocal] = \mcal[T][\hfine]$. The result then immediately follows by noting~\eqref{eq:H1_2} together with the coercivity of $\mcal[A][h]$, cf.~\Cref{lem:contcoerc_2}.
\end{proof}

\begin{remark}\label{rmrk:CondNonNested}
Based on~\Cref{thm:ValidityStableDecompositionNonNested}, for non-nested coarse and fine spaces $V_{\hcoarse}$ and $V_{\hfine}$, respectively, the condition number $K(P_{ad})$ of the additive Schwarz operator $P_{ad}$ can be bounded as follows:
\begin{equation}\label{eq:CondBoundNonNested}
K(P_{ad}) \lesssim \Bigl( \frac{p^4}{q^2} \frac{\hcoarse^2}{\hfine^2} \Bigr) (N_{\mathbb{S}}+1).
\end{equation}
\end{remark}

\bibliographystyle{abbrv}
\bibliography{biblio_AddSch}
\end{document}